\documentclass{siamart171218}
\usepackage{pdfsync}
\usepackage{amsmath,amssymb}
\usepackage{enumerate}
\usepackage{xspace}
\usepackage{boxedminipage}
\usepackage{algorithmicx}
\usepackage{algorithm}
\usepackage{algpseudocode}
\usepackage{listings}
\usepackage{tabularx}
\usepackage{subfigure}
\usepackage{tristan}
\usepackage{fullpage}
\usepackage{mathtools}
\usepackage{multirow}
\usepackage{booktabs}
\renewcommand{\bi}[2]{\ensuremath{\cA\qp{#1,#2}}}
\renewcommand{\bih}[2]{\ensuremath{\cA_h\qp{#1,#2}}}
\newcommand{\ah}[2]{\ensuremath{\cA_h\qp{#1,#2}}}

\renewcommand{\vec}[1]{\geovec{#1}}
\renewcommand{\mat}[1]{\geomat{#1}}

\renewcommand{\enorm}[2]{\ensuremath{\Norm{#1}}_{dG,{#2}}}

\renewcommand{\leb}[1]{\ensuremath{\LL^{#1}}}
\newcommand{\lift}{r_h^*}

\newcommand{\uhs}{u^*}
\newcommand{\uhss}{u^{**}}

\renewcommand{\D}{\vec{\mathfrak D}}

\numberwithin{equation}{section}
\setboolean{showtodo}{true}
\setboolean{shownotes}{true}
\setboolean{showchanges}{false}
\setboolean{usemathrsfs}{false}
\renewcommand{\hoz}{\sobh1_0(\W)}
\usepackage{tkz-graph}
\usetikzlibrary{arrows}

\usepackage{lipsum}
\usepackage{amsfonts}
\usepackage{graphicx}
\usepackage{epstopdf}
\ifpdf
  \DeclareGraphicsExtensions{.eps,.pdf,.png,.jpg}
\else
  \DeclareGraphicsExtensions{.eps}
\fi

\newsiamremark{remark}{Remark}
\newsiamremark{hypothesis}{Hypothesis}
\crefname{hypothesis}{Hypothesis}{Hypotheses}

\headers{}{}

\author{
  Andreas Dedner
  \thanks{
        Mathematics Institute, University of Warwick, Coventry CV4 7AL UK.
    (\email{A.S.Dedner@warwick.ac.uk}).}
  \and
  Jan Giesselmann
  \thanks{  Department of Mathematics,  TU Darmstadt, 64293 Darmstadt, Germany.
  (\email{giesselmann@mathematik.tu-darmstadt.de}).}
  \and
  Tristan Pryer
  \thanks{
    Mathematical Sciences, University of Bath, Bath BS2 7AY, UK.
  Corresponding author. (\email{tmp38@bath.ac.uk}).}
  \and
  Jennifer K Ryan
  \thanks{  
  Applied Mathematics \& Statistics, Colorado School of Mines, Golden, Colorado 80401, USA.
  (\email{jkryan@mines.edu}).}
}

\usepackage{amsopn}

\title{Residual estimates for post-processors in elliptic problems}
\date{\today}

\begin{document}
\maketitle
\begin{abstract}
  In this work we examine {\it a posteriori} error control for post-processed
  approximations to elliptic boundary value problems.
  
  We introduce a class of post-processing operator that ``tweaks'' a
  wide variety of existing post-processing techniques to enable
  efficient and reliable {\it a posteriori} bounds to be proven. This
  ultimately results in optimal error control for all manner of
  reconstruction operators, including those that superconverge.

  We showcase our results by applying them to two classes of very
  popular reconstruction operators, the Smoothness-Increasing
  Accuracy-Conserving filter and Superconvergent Patch
  Recovery. Extensive numerical tests are conducted that confirm our
  analytic findings.
\end{abstract}

\section{Introduction}

Post-Processing techniques are often used in numerical simulations for
a variety of reasons from visualisation purposes \cite{benzley1995pre}
to designing superconvergent approximations \cite{BrambleSchatz:1977}
through to becoming fundamental building blocks in constructing
numerical schemes \cite{georgoulis2018recovered,GuoZhangZou:2018, ChenGuoZhangZou:2017}.  Another application
of these operators is that they are a very useful component in the
{\it a posteriori} analysis for approximations of partial differential
equations (PDEs) \cite{Ainsworth:2000,Verfurth:1996}. The goal of an
{\it a posteriori} error bound is to computationally control the error
committed in approximating the solution to a PDE.  In order to
illustrate the ideas, let $u$ denote the solution to some PDE and let
$u_h$ denote a numerical approximation.  Then, the simplest possible
use of post-processing in {\it a posteriori estimates} is to compute
some $\uhs$ from $u_h$ and to use
 \[ \Norm{u-u_h} \approx  \Norm{\uhs-u_h} \]
 as an error estimator.

However, a key observation here (and in several more sophisticated
approaches) is that $u^*$ must be at least as good of an approximation
of the solution $u$ as $u_h$ is. In fact, in many cases, $u^*$ is
actually expected to be a better approximation. This raises a natural
question: If an adaptive algorithm computes (on any given mesh) not
only $u_h$ but also $u^*$ and $u^*$ is a better approximation of $u$
than $u_h$ is, why is $u_h$ and not $u^*$ considered as the
``primary'' approximation of $u$?  Indeed, the focus of this paper is
to consider $u^*$ as the primary approximation of $u$.  We are
therefore interested in control of the error $\Norm{u-u^*}$ and in
adaptivity based on an a posteriori estimator for $\Norm{u-u^*}$.
Specifically, {\it we aim to provide reliable and efficient error
  control for $\Norm{u-u^*}$.}

Note that our goal is not to try to construct ``optimal'' superconvergent
post-processors.  Rather we try to determine, from the {\it a posteriori}
viewpoint, the accuracy of some given post-processed solution and to 
determine how this is useful for the construction of adaptive numerical schemes
based on an error tolerance for $u - u^*$.

There are several examples of superconvergent post-processors, 
includig SIAC and
superconvergent patch recovery. More details on the history, properties and implementation of these methods will be provided in Sections \ref{sec:siac} and \ref{sec:patchrecovery}. However, our a posteriori analysis,
aims at being applicable for a wide variety of 
post-processors and, therefore, we avoid special assumptions that are
only valid for specific post-processors. 
Indeed, our analysis makes only very mild assumptions on the
post-processing operator. Specifically, we only require that:
\begin{enumerate}
\item The post-processed solution $u^*$ belongs to a finite dimensional space
  that contains piecewise polynomials, although it does not necessarily
  need to be piecewise polynomial itself.
\item The post-processed solution should be piecewise smooth over the
  same triangulation, or a sub-triangulation, of the finite element
  approximation.
\end{enumerate}
Given a post-processor, $u^*$ that satisfies these rather mild
assumptions, we perturb it slightly and call the result $u^{**}$. This
is to ensure an orthogonality condition holds which then allows us to
show various desirable properties including:
\begin{enumerate}
\item The orthogonal post-processor provides a better approximation
  than the original post-processor, i.e. $\Norm{u-u^{**}}_{\cA_h} \leq
  \Norm{u-u^*}_{\cA_h}$ in the energy norm, see Lemma \ref{lem:uhssimpro}.
\item The orthogonal post-processor has an increased convergence order in the
  $\leb{2}$ norm. Practically, this is not always the case for the
  original post-processor, see Lemma \ref{lem:dual}.
\item Efficient and reliable {\it a posteriori} bounds are available for the
  error committed by the orthogonal post-processor.
\end{enumerate}

This, motivates us to consider $\uhss$ (and not $u^*$ or $u_h$) 
as the primary approximation. Since the improved accuracy of $u^{**}$, 
compared to $u_h$, stems from superconvergence it is much more sensitive 
with respect to smoothness of the exact solution, i.e. in regions where the
exact solution is $C^\infty$ we expect $u^{**}$ to be much more accurate than
$u_h$ whereas in places where the exact solution is less regular, e.g. has kinks, 
$u_h$ and $u^{**}$ are expected to have similar accuracy. Therefore, meshes constructed 
based on error estimators for $u-u_h$ will usually not be optimal when used for computing $u^{**}$
in the sense that the ratio of degrees of freedom to error $\| u- u_h\|$ would be much better for other meshes, this is elaborated upon in Section \ref{sec:numerics}.

We will demonstrate the good approximation properties of our modification strategy for post-processors and the benefits of basing mesh adaptation on an estimator for $\| u- u_h\|$
in  a series of numerical experiments. In order to highlight the versatility of our approach, we conduct experiments based on two popular post-processing techniques: The Smoothness Increasing Accuracy Conserving (SIAC) filter and superconvergent patch recovery (SPR). Background on these methods is provided in Sections \ref{sec:siac} and \ref{sec:patchrecovery} respectively.

The rest of the paper is structured as follows: In \S \ref{sec:setup} we
introduce the model elliptic problem and its dG approximation. We also recall
 some standard results for this method. In \S \ref{sec:recon}, for a
given reconstruction, we perturb it so it satisfies Galerkin
orthogonality  and show some {\it a priori} type results. We then
study {\it a posteriori} results and give upper and lower bounds for a
residual type estimator. In \S\ref{sec:recovery} we describe the two
families of post-processor that we consider in this work. Finally, in
\S\ref{sec:numerics} we perform extensive numerical tests on the SIAC
and SPR post-processors to show the performance of the {\it a
  posteriori} bounds, the effect of smoothness of the solution on the
post-processors and to study adaptive methods driven by these
estimators.

\section{Problem setup and notation}
\label{sec:setup}

Let $\Omega \subset \rR^d$, $d=1,2,3$ be bounded with Lipschitz
boundary $\partial \Omega$. We denote by $\leb{p}(\Omega)$, $p\in
[1,\infty]$, the standard Lebesgue spaces and $\sobh{s}(\Omega)$, the
Sobolev spaces of real-valued functions defined over $\Omega$. Further
we denote $\sobh1_0(\Omega)$ the space of functions in
$\sobh1(\Omega)$ with vanishing trace on $\partial \Omega$.

For $f \in \leb{2}(\Omega)$ we consider the problem
\begin{equation}\label{ellipt}
\begin{split}
  - \div\qp{\D \nabla u} &=f \text{ in } \Omega\\
  u&=0 \text{ on } \partial \Omega,
\end{split}
\end{equation}
where $\D:\W \to \reals^{d\times d}$ is a uniformly positive definite
diffusion tensor and $\D \in \qb{\sobh1(\W)\cap
  \leb{\infty}(\W)}^{d\times d}$. Weakly, the problem reads: find
$u\in \hoz$ such that
\begin{equation}
  \label{eq:abstract-prob}
  \bi{u}{v}
  :=
  \int_\Omega \D \nabla u \cdot \nabla v \d \vec x
  =
  \int_\Omega f v \d \vec x  \Foreach v\in \hoz.
\end{equation}

Let $\T{}$ be a triangulation of $\Omega$ into disjoint simplicial or
box-type (quadrilateral/hexahedral) elements $K\in \T{}$ such that
$\closure{\Omega}=\union{K\in\T{}} \closure{K}$. Let $\E$ be the set
of edges which we split into the set of interior edges $\cE_i$ and the
set of boundary edges $\cE_b$.

We introduce the standard broken Sobolev spaces. For $s \in \rN_0$ we
define
\begin{equation}
  \sobh{s}(\T{})
  :=
  \{ v \in \leb{2}(\Omega) \, : \, v|_K \in \sobh{s}(K) \, \Foreach K \in \T{}\},
\end{equation}
and we will use the notation
\begin{equation}
  \Norm{v}_{\sobh{s}(\T{})} = \qp{\sum_{K\in\T{}} \Norm{v}_{\sobh{s}(K)}^2}^{1/2}
\end{equation}
as an elementwise norm for the broken space.

For $p\in\naturals$ we denote the set of all polynomials over $K$ of
total degree at most $p$ by $\poly{p}(K)$. For $p \geq 1$, we consider
the finite element space
\begin{equation}
  \fes_h^p := \{ \phi \in \leb{2}(\Omega) : v|_K \in \poly{p}(K), K\in\T{} \}.
\end{equation}

Let $v\in\sobh{1}(\T{})$ be an arbitrary scalar function.  For any
interior edge $e \in \cE_i$ there are two adjacent triangles $K^-,
K^+$ and we can consider the traces $v^\pm$ of $v$ from $K^\pm$
respectively. We denote the outward normal of $K^\pm$ by $\vec
n^\pm$ and define average and jump operators for one $\cE_i$ by
\begin{equation}
  \begin{split}
    \avg{v}
    &:=
        {\frac12\qp{ v^+ + v^-}}:=  \frac12\qp{\lim_{s \searrow 0} v(\cdot + s \vec n^+) + \lim_{s \searrow 0} v(\cdot + s \vec n^-)}
        \\
        \jump{v}
        &:= {\qp{v^- \vec n^- + v^+\vec n^+}}:=  \lim_{s \searrow 0} v(\cdot + s \vec n^-) \vec n^- +\lim_{s \searrow 0} v(\cdot + s\vec n^+) \vec n^+.
  \end{split}
\end{equation}
For boundary edges there is only one trace of $v$ and one outward pointing normal vector $\vec n$ and we define
\begin{equation}
  \avg{v}:= v
  \qquad 
  \jump{v} := v \vec n.
\end{equation}
For vector valued functions $\vec v\in[\sobh{1}(\T{})]^d$ we define 
jumps and averages on interior edges by
\begin{gather}
  \avg{\vec v} :=   \frac12 \vec v^+ + \frac12 \vec v^-
  ,
 \quad
  \jump{\vec v} := {\qp{\vec v^- \cdot \vec n^- +  \vec v^+ \cdot \vec n^+}}
\end{gather}
As before, for boundary edges, we define jumps and averages using
traces from the interior only. Note that $\jump{\vec v},\avg{v} \in \leb{2}(\E)$ and $\jump{v}, \avg{\vec v}\in [\leb{2}(\E)]^d$.

For any triangle $K \in \T{}$ we define $h_K:= \diam{K}$ and collect
these values into an element-wise constant function $h:\Omega\to
\reals$ with $h|_K = h_K$. We denote the radius of the largest ball
inscribed in $K$ by $\rho_K$. For every edge $e$ we denote by $h_e =
\avg{h}$, i.e., the mean of diameters of adjacent triangles. For our
analysis we will assume that $\T{}$ belongs to a family of
triangulations which is quasi-uniform and shape-regular. Let us
briefly recall the definitions of these two notions: The triangulation
$\T{}$ is called
\begin{itemize}
\item shape-regular if there exists $C>0$ so that 
  \begin{equation}
    h_K < C \rho_K \Foreach K \in \T{}
\end{equation}
\item quasi-uniform if there exists $C>0$ so that 
  \begin{equation}
    \max_{K \in \T{}} h_K < C h_K    \Foreach K \in \T{}.
\end{equation}
\end{itemize}
Note that for shape-regular triangulations we have inverse and trace
inequalities \cite[Lemmas 1.44, 1.46]{Di-PietroErn:2012}. Note that
the quasi-uniformity assumption is only required for the first part of
our analysis, in \S \ref{sec:apriori} and can be relaxed in \S
\ref{sec:apost}.


In this work we will consider a standard interior penalty method to
approximate solutions of \eqref{eq:abstract-prob}. We consider the
Galerkin method to seek $u_h \in \fes_h^p$ such that
\begin{equation}
  \label{eq:galerkin-method}
  \bih{u_h}{v_h} = \int_\W f v_h\d \vec x \Foreach v_h\in \fes_h^p,
\end{equation}
where $\cA_h: \sobh{2}(\T{}) \times \sobh{2}(\T{}) \rightarrow \rR$ is given by 
\begin{equation}\label{eq:ipdg}
  \bih{u}{v}= \int_{\T{}} \D\nabla u \cdot \nabla v
  -
  \int_\cE \jump{v} \cdot \avg{\D\nabla u}
  -
  \int_\cE \jump{u} \cdot \avg{\D\nabla v}
  +
  \int_\cE  \sigma h_e^{-1} \jump{u} \cdot \jump{v} 
\end{equation}
Note that the bilinear form \eqref{eq:ipdg} is stable provided $\sigma
= \sigma(\D)$ is large enough, see
\cite{ArnoldBrezziCockburnMarini:2001}.

\begin{remark}[Continuous Galerkin methods]\label{rem:vh}
  Note that if we restrict test and trial functions to $\fes_h^p \cap
  \sobh{1}(\W)$ then all jumps on interior edges vanish and
  \eqref{eq:galerkin-method}, \eqref{eq:ipdg} reduces to a
  (continuous) finite element method with weakly enforced boundary
  data. Our analysis is equally valid in this case.
\end{remark}
We introduce two dG norms
\begin{equation}
  \begin{split}
    \Norm{v}_{dG}^2
    &:=
    \Norm{\nabla v}_{\leb{2}(\T{})}^2 + \Norm{h_e^{-\frac{1}{2}} \jump{v}}_{\leb{2}(\cE)}^2
    \\
    \Norm{v}_{\cA_h}^2
    &:=
    \bih{v}{v},
  \end{split}
\end{equation}
which are equivalent provided $\sigma >0$ is sufficiently large and
conclude this section by stating a-priori estimates for the Galerkin
method as is standard in the literature
\cite{ArnoldBrezziCockburnMarini:2001,KarakashianPascal:2003}.

\begin{theorem}[Error bounds for the dG approximation]
  \label{the:error-est}
  Let $u\in\sobh{s}(\W)$ for $s\geq 2$ be the solution of
  (\ref{ellipt}) and $u_h\in\fes_h^p$ be the unique solution to the
  problem (\ref{eq:galerkin-method}). Then,
  \begin{equation}
    \Norm{u-u_h}_{\leb{2}(\W)} + h\Norm{u-u_h}_{dG}
    \leq
    C_1 h^{\min\,\qp{p+1,s}} \norm{u}_{\sobh{s}(\W)}.
  \end{equation}
  Further, for $u\in\sobh{1}(\W)$, we have the {\it a posteriori} error
  bound
  \begin{equation}\label{eq:primalerror_dg}
    \Norm{ u - u_h}_{dG} \leq 
    C_2 R_h
    :=
    C_2
    \qp{
      \sum_{K\in\T{}}
      \qp{
        \eta_K^2
        +
        \frac{1}{2} \sum_{e \in \partial K} \eta_e^2
      }
    }^{\frac{1}{2}},
  \end{equation}
  where
  \begin{equation}
    \begin{split}
      \eta_K^2 &:= \Norm{h_K ( f + \div\qp{\D\nabla u_h})}_{\leb{2}(K)}^2
      \\
      \eta_e^2
      &:=
      \Norm{h_e^{\frac{1}{2}} \jump{\D\nabla u_h}}_{\leb{2}(e)}^2
      +
      \Norm{h_e^{-\frac{1}{2}} \jump{ u_h}  }_{\leb{2}(e)}^2.
    \end{split}
  \end{equation}
  and $C_1$ is a constant depending on the shape-regularity and
  quasi-uniformity constants of $\T{}$ and $C_2$ depends only upon the
  shape-regularity. Here $R_h$ is a computable residual that we refer
  to during our numerical simulations.
\end{theorem}

\section{The orthogonal reconstruction, {\it a priori} and {\it a posteriori} error estimates}
\label{sec:recon}

In this section, we derive robust and efficient error estimates. We
make the assumption that we have access to a computable
reconstruction, $\uhs \in \fes_h^* \subset \sobh{2}(\T{})$ generated
from our numerical solution $u_h$, where
 $\fes_h^*$ is required to contain the original finite element space,
 that is $\fes_h^p \subset \fes_h^*$. We are
unable to provide reliable {\it a posteriori} error estimates for $\uhs$ directly, but we can
modify, and, as we shall demonstrate, improve any such reconstruction
such that a robust and efficient error estimate can be obtained for
the modified reconstruction.

We split this section into two parts, the first subsection contains
the definition of the improved reconstruction and some of its
properties. In particular, we study this from an {\it a priori}
viewpoint, show that it satisfies Galerkin orthogonality as well as
some desirable {\it a priori} bounds. Throughout this subsection we
assume that $u\in\sobh2(\W)$ and the underlying mesh is
quasi-uniform. In the second part we derive reliable and efficient
{\it a posteriori} estimates under the weaker assumption that
$u\in\sobh1(\W)$ and the mesh is shape regular.

\subsection{Improved reconstruction}
\label{sec:apriori}

In the following assume that $u\in\sobh2(\W)$ solves
\eqref{eq:abstract-prob} and
let $\uhs\in\fes_h^* \subset \sobh2(\T{})$ be a reconstruction of the
discrete solution $u_h$, e.g. a SIAC reconstruction as described in Section
\ref{sec:siac} or obtained by some patch recovery operator as described in
Section \ref{sec:patchrecovery}.

\begin{definition}\label{def:nrec}
  Let $R: \sobh{2}(\T{}) \rightarrow \fes_h^p$ denote the Ritz
  projection with respect to $\ah{\cdot}{\cdot}$, i.e.,
  \begin{equation}
    \ah{Rv}{\phi_h} = \ah{v}{\phi_h} \Foreach \phi_h \in \fes_h^p.
  \end{equation}

  We define the improved reconstruction as
  \begin{equation}\label{eq:nrec}
    \uhss := \uhs - R\uhs + u_h \in \fes_h^*
  \end{equation}
\end{definition}

\begin{remark}
  We make the following remarks:
 \begin{enumerate}
 \item
   The finite element approximation from \eqref{eq:galerkin-method}
   satisfies $u_h = R u$.
 \item
   {We work under the assumption that a post-processor
     $\uhs$ is already being computed. To realise $\uhss$ we are
     required to solve the original elliptic problem a second time
     with a different forcing term. This means the improved
     reconstruction $\uhss$ is computable at a small additional cost
     to $\uhs$. Once $\uhs$ has been computed, $\uhss$ can be computed
     by solving a discrete elliptic problem over $\fes_h^p$.  A
     typical scenario is that the user already has a good scheme for
     computing $u_h$, and that the cost of computing $u^{**}$ (after
     the post-processing to obtain $u^*$) is just that of solving the
     same system as that for $u_h$ with a different right hand
     side. This means the assembly and preconditioning, perhaps ILU or
     AMG, can be reused without change.
   }

{
  Estimating the cost of computing $u^*$ is more complicated and will
  depend on the method used and the implementation. While our
  implementation for solving $u_h$ and the correction are optimized (and
  implemented in C++) the computation of $u^{**}$ is a proof of concept
  implementation in Python and is therefore not competitive. 
}

   The improved reconstruction $\uhss$ is computable at a small
   additional cost to $\uhs$. Once $\uhs$ has been computed, $\uhss$
   can be computed by solving a discrete elliptic problem over
   $\fes_h^p$.  A typical scenario is that the user already has a good
   scheme for computing $u_h$, and that the cost of computing $u^{**}$
   (after the post-processing to obtain $u^*$) is just that of solving
   the same system as that for $u_h$ with a different righthand side,
   e.g if an $LU$ decomposition of the system matrix was determined
   for computing $u_h$ this $LU$ decomposition can be reused.
 \item
   Note that
   \begin{equation} \label{eq:id_R}
     u - \uhss =  u - u_h - \uhs + R \uhs = (id - R)( u - \uhs),
   \end{equation}
   {where $id$ is the identity mapping,}
   i.e., the error of $ \uhss $ is the Ritz-projection of the error of
   $\uhs $ onto the orthogonal complement of $\fes_h^p$.
 \item Even if $\uhs$ is continuous, this does not  necessarily hold for $\uhss$
   as $\fes_h^p$ may contain discontinuous functions.
   \end{enumerate}
  \end{remark}
     
One of the key properties of the improved reconstruction is that it
satisfies a Galerkin orthogonality result.
  \begin{lemma}[Galerkin orthogonality]\label{lem:uhssortho}
 The reconstruction $\uhss$ from \eqref{eq:nrec} satisfies Galerkin orthogonality, i.e.,
 \begin{equation}\label{eq:uhssortho}
 \ah{u - \uhss}{ v_h} =0 \quad \forall v_h \in \fes_h^p. 
 \end{equation}

\end{lemma}
\begin{proof}
For any $v_h \in \fes_h^p$, we have
using \eqref{eq:id_R}
  \begin{equation}
    \begin{split}
      \ah{u - \uhss}{ v_h}
      &= \ah{ (id - R)( u - \uhs) }{v_h}
      =0
    \end{split}
  \end{equation}
 by definition of the Ritz projection, as required.
\end{proof}

Now, we show that with respect to $\Norm{\cdot}_{\cA_h}$ the new
reconstruction $\uhss$ indeed improves upon $\uhs$:
\begin{lemma}[Better approximation of the improved reconstruction]\label{lem:uhssimpro}
  Let $\uhss$ be defined by \eqref{eq:nrec}, then the following holds:
  \begin{equation}\label{better}
    \Norm{ u - \uhss}_{\cA_h}
    \leq \Norm{ u - \uhs}_{\cA_h}.
  \end{equation}
  In \eqref{better} the inequality is an equality if and only if
  $\uhss=\uhs$, i.e., if the original
  reconstruction $\uhs$ itself satisfies Galerkin
  orthogonality.
\end{lemma}
\begin{proof}
  Since the images of $R$ and $(id - R)$ are
  orthogonal with respect to $\ah{\cdot}{\cdot}$, Pythagoras'
  theorem implies
  \begin{equation}
      \label{ortho}
      \Norm{u - \uhs}_{\cA_h}^2
      = \Norm{(id-R)(u-\uhs)+R(u-\uhs)}_{\cA_h}^2
      \geq \Norm{(id-R)(u-\uhs)}_{\cA_h}^2
      = \Norm{u-\uhss}_{\cA_h}^2
    \end{equation}
  We have used  \eqref{eq:id_R} in the third step.
  Note that if $\uhs$ is not Galerkin orthogonal then
  $\Norm{R(u-\uhs)}_{\cA_h} > 0$ leading to a strict inequality in the
  first step. This completes the proof.
\end{proof}

\begin{remark}
  One appealing feature of the new reconstruction that results from
  Galerkin orthogonality is that if the reconstruction $\uhs$ has some
  superconvergence properties in the energy norm this is inherited by
  $\uhss$ and also immediately implies an additional order of accuracy
  in $\leb{2}$. This results from an Aubin-Nitsche trick being
  available.
\end{remark}

\begin{lemma}[Dual bounds]\label{lem:dual}
 Let $\Omega$ be a convex polygonal domain and let $\uhss$ be defined
 by \eqref{eq:nrec}, then there exists a constant $C>0$ (only
 depending on the shape regularity of the mesh) such that
 \begin{equation}
  \Norm{u - \uhss}_{\leb{2}(\Omega)} \leq Ch \Norm{u - \uhss}_{\cA_h}.
 \end{equation}
\end{lemma}

\begin{proof}
  Let $\psi \in \sobh{2}(\Omega) \cap \hoz$ solve
  $$- \div\qp{\D\nabla \psi} = u - \uhss$$ which implies
  \begin{equation}\label{eq:dual1}
    \int_{\T{}} \D \nabla \psi \cdot \nabla v
   -
   \int_\cE \jump{v} \cdot \avg{\D\nabla \psi}
   =
   \int_\Omega (u - \uhss) v \quad \forall v \in \sobh{1}(\T{}).
 \end{equation}
Thus, by choosing $v = u - \uhss$ in \eqref{eq:dual1} we obtain
\begin{equation}
 \begin{split}
   \Norm{u - \uhss}_{\leb{2}(\Omega)}^2
   &=
   \int_{\T{}}  \D\nabla \psi \cdot \nabla (u - \uhss)
   -
   \int_\cE \jump{u - \uhss} \cdot \avg{\D\nabla \psi}
   \\
   &=
   \ah{\psi}{u - \uhss}
   \\
   &=
   \ah{\psi - \psi_h}{u - \uhss}
 \end{split}
\end{equation}
for any $\psi_h \in \fes_h^p$ where the last equality follows from
Galerkin orthogonality \eqref{eq:uhssortho}.  Thus, choosing $\psi_h$ as the best
approximation of $\psi$ in the piecewise linear subspace of $\fes_h^p$,
we obtain
\begin{equation}
 \begin{split}
   \Norm{u - \uhss}_{\leb{2}(\Omega)}^2
   &\leq
   \Norm{\psi- \psi_h}_{\cA_h} \Norm{u - \uhss}_{\cA_h}
   \\
   &\leq
   Ch \Norm{u - \uhss}_{\cA_h} \Norm{\nabla^2 \psi}_{\leb{2}(\Omega)}
   \\
   & \leq C h \Norm{u - \uhss}_{\leb{2}(\Omega)} \Norm{u - \uhss}_{\cA_h},
 \end{split}
\end{equation}
by elliptic regularity of the dual problem, concluding the proof.
\end{proof}

\subsection{{\it A posteriori} error estimates}
\label{sec:apost}
Now that we have shown some fundamental results on the improved
reconstruction, we relax the regularity requirements on $u$ in this
subsection allowing for weak solutions to (\ref{ellipt}), that is,
$u\in\sobh1(\W)$. With that in mind we modify the definition of
$\ah{\cdot}{\cdot}$ such that it is a suitable extension over
$\sobh{1}(\T{})\times \sobh{1}(\T{})$ to
\begin{equation}
  \label{eq:ipdg_e}
  \ah{u}{v}
  :=
  \int_{\T{}} \D\nabla u \cdot \nabla v
  -
  \lift(\jump{v}) \cdot \D\nabla u
  -
  \lift( \jump{u} ) \cdot \D\nabla v
  +
  \int_\cE \sigma h_e^{-1} \jump{u} \cdot \jump{v},
\end{equation}
for $u,v \in \sobh{1}(\T{})$ and where $\lift : [\leb{2}(\cE)]^d
\rightarrow \qb{\fes_h^*}^d$ is the lifting operator that we recall from
\cite[Section 4.3.1]{Di-PietroErn:2012}
\begin{equation}
  \label{eq:liftdef}
  \int_\Omega \lift (\varphi) \cdot \D \vec \psi_h
  =
  \int_\cE \varphi \cdot \avg{\D \vec \psi_h} \quad \Foreach \vec \psi_h \in \qb{\fes_h^*}^d.
\end{equation}
The lifting operators satisfy the stability estimate, \cite[Lemma
  4.34]{Di-PietroErn:2012},
\begin{equation}\label{eq:liftstab}
  \Norm{\lift(\varphi)}_{\leb{2}(\Omega)}
  \leq  C \Norm{ h_e^{-\frac{1}{2}} \varphi}_{\leb{2}(\cE)},
\end{equation}
For test and trial functions in $\fes_h^*$ (which contains $\fes_h^p$
by assumption) the new definition of $\ah{\cdot}{\cdot}$ is equivalent
to the one given in \eqref{eq:ipdg}. Therefore for any function
$v^*\in\fes_h^*$ the Ritz projection given in Definition
\ref{def:nrec} remains the same still satisfying $\ah{Rv^*}{\phi_h} =
\ah{v^*}{\phi_h}$ for all $\phi_h \in \fes_h^p$.  But note that we no
longer have $u_h=Ru$ and Galerkin orthogonality for $\uhss$ no longer
holds in general, it only holds for a $\sobh1(\W)$ conforming
subspace of $\fes_h^p$:

\begin{lemma}\label{lem:condortho}
  For $u\in\sobh{1}(\W)$ and $z_h\in\fes_h^p\cap\sobh1_0(\W)$ it holds that
  \begin{equation}
    \ah{\uhss - u}{z_h} = 0.
  \end{equation}
\end{lemma}
\begin{proof}
  By definition of $\uhss$ we have that
  \begin{equation}
    \begin{split}
      \ah{\uhss - u}{z_h}
      &=
      \ah{\uhs - R \uhs + u_h - u}{z_h}
      \\
      &=
      \ah{\uhs - R\uhs}{z_h}
      +
      \ah{u_h - u}{z_h}
      \\
      &=
      \ah{u_h - u}{z_h},
    \end{split}
  \end{equation}
  since $\uhs \in \fes_h^*$. 
  Now, notice that by definition
  \begin{equation}
    \ah{u_h - u}{z_h} = \ltwop{f}{z_h} - \int_{\T{}} \D\nabla u \cdot \nabla z_h - \lift(\jump{z_h}) \cdot \D\nabla u = 0
  \end{equation}
  as $z_h $ is an element of $ \fes_h^p\cap \sobh1_0(\W)$ and, hence, continuous, as required.
\end{proof}

Let a quantity of interest be given by the linear functional
$\cJ\in\sobh{-1}(\T{})$, the dual space of $\sobh{1}_0(\T{})$. Note that $\sobh{-1}(\T{}) \subset \sobh{-1}(\W)$ where the latter is the dual space of $\sobh{1}_0(\W)$. We begin
by deriving an error representation formula. Following
\cite{HoustonSchoetzauWihler:2005}, we split $\uhss$
into a continuous part $\uhss_C \in \fes_h^{*} \cap \sobh{1}_0(\Omega)$
and a discontinuous part $\uhss_\perp \in \fes_h^{*}$ so that
\begin{equation}
  \uhss= \uhss_C + \uhss_\perp
  \quad
  \text{ and } 
  \quad
  \bih{ \uhss_\perp }{\psi_h} = 0
  \Foreach \psi \in \fes_h^{*}\cap \sobh{1}_0(\Omega).
\end{equation}
\begin{theorem}[Dual error representation]
  \label{the:err-rep}
  Let $u\in\sobh1_0(\W)$ be the solution of \eqref{eq:abstract-prob} and
  let $\uhss$ be given by \eqref{eq:nrec}, then
  \begin{equation}
    \cJ(u - \uhss)
    =
    \langle f, z - z_h \rangle - \bih{\uhss}{z - z_h} + \bih{\uhss_\perp}{z} - \cJ(\uhss_\perp)
  \end{equation}
  where $\langle \cdot,\cdot \rangle$ denotes the $\leb{2}$ scalar product and $z \in \hoz$ is the solution of the dual problem
  \begin{equation}
    \label{eq:dual-prob}
    \bi{v}{z} = \cJ( v ) \Foreach v\in \hoz
  \end{equation}
  and  $z_h$ is an arbitrary function in $\fes_h^p\cap \sobh{1}_0(\W)$.
\end{theorem}

\begin{proof}
  By definition of $z$, we have, for any
  $z_h \in \fes_h^p\cap \sobh{1}_0(\W)$,
  \begin{equation}
    \begin{split}
      \cJ(u - \uhss) &= \cJ(u- \uhss_C - \uhss_\perp)
      \\
      &=
      \cJ(u- \uhss_C) - \cJ(\uhss_\perp)
      \\
      &=
      \bi{u-\uhss_C}{ z}- \cJ(\uhss_\perp)
      \\
      &=
      \langle f,z\rangle - \bih{\uhss_C}{z}- \cJ(\uhss_\perp)
      \\
      &=
      \langle f,z\rangle - \bih{\uhss }{z} + \bih{\uhss_\perp}{z}- \cJ(\uhss_\perp)
      \\
      &=
      \langle f,z\rangle - \bih{\uhss }{z-z_h}  -  \bih{\uhss }{z_h} + \bih{\uhss_\perp}{z} - \cJ(\uhss_\perp)
      \\
      &=
      \langle f,z-z_h\rangle - \bih{\uhss }{z-z_h}  + \bih{\uhss_\perp}{z}- \cJ(\uhss_\perp),
    \end{split}
  \end{equation}
  where we made use of Galerkin orthogonality, Lemma \ref{lem:condortho}, in the last step.
\end{proof}

\begin{theorem}[Primal error estimate]\label{the:primalerror2}
  There exists some constant $C_A >0$ depending on mesh geometry and
  polynomial degree such that
  \begin{equation}\label{eq:primalerror2}
    \Norm{ u - \uhss}_{dG} \leq 
    C_A R^{**}
    :=
    C_A \qp{
      \sum_{K\in\T{}}
      \qp{
        (\eta^{**}_K)^2
        +
        \frac{1}{2} \sum_{e \in \partial K} (\eta^{**}_e)^2
      }
    }^{\frac{1}{2}},
  \end{equation}
  where
  \begin{equation}
    \begin{split}
      (\eta^{**}_K)^2 &:= \Norm{h_K ( f + \div\qp{\D\nabla \uhss})}_{\leb{2}(K)}^2
      \\
      (\eta^{**}_e)^2
      &:=
      \Norm{h_e^{\frac{1}{2}} \jump{\D\nabla \uhss}}_{\leb{2}(e)}^2
      +
      \Norm{h_e^{-\frac{1}{2}} \jump{ \uhss}  }_{\leb{2}(e)}^2
    \end{split}
  \end{equation}
\end{theorem}

\begin{proof}
  Since 
  \begin{equation}
    \Norm{ u -  \uhss}_{dG}
    \leq C
    \left( \Norm{\D^{\tfrac{1}{2}}\qp{\nabla u - \nabla \uhss}}_{\leb{2}(\T{})}^2  + \sum_{e \in \cE} \Norm{h_e^{-\frac 12} \jump{\uhss}}_{\leb{2}(e)}^2 \right)^{\frac12},
  \end{equation}
  where $C$ is some constant depending on $\D$ only, it is sufficient to show that $\Norm{\D^{\tfrac{1}{2}}\qp{\nabla u - \nabla
      \uhss}}_{\leb{2}(\T{})}$ is bounded by the right hand side of
  \eqref{eq:primalerror2}.
  
  In Theorem \ref{the:err-rep} we may choose
  \begin{equation}
    \cJ(v):= \int_{\T{}}
    \D(\nabla u - \nabla \uhss) \nabla v.
  \end{equation}
  Note that, by definition, $z \in \hoz$ so that, if $\fes_h^p$ contains
  discontinuous functions, $z \not= u - \uhss$.  Nevertheless, $z$
  satisfies the stability estimate
  \begin{equation}
    \Norm{\D^{\tfrac{1}{2}}\nabla z}_{\leb{2}(\Omega)}
    \leq
    \Norm{\D^{\tfrac{1}{2}}\qp{\nabla u - \nabla \uhss}}_{\leb{2}(\T{})}.
  \end{equation}
  Then, for any $z_h \in \fes_h^p \cap \hoz$, Theorem \ref{the:err-rep}
  implies
  \begin{equation}
    \label{eq:rely1}
    \begin{split}
      \Norm{\D^{\tfrac{1}{2}}\nabla \qp{u - \uhss}}_{\leb{2}(\T{})}^2
      &=
      \cJ(u - \uhss)            
      \\
      &=
      \langle f,z-z_h\rangle - \bih{\uhss }{z-z_h}  + \bih{\uhss_\perp}{z}- \cJ(\uhss_\perp)
      \\
      &=
      \langle f,z-z_h\rangle - \int_{\T{}} (\nabla \uhss - \lift(\jump{\uhss})) \D\nabla (z - z_h)
      \\
      &\qquad + \bih{\uhss_\perp}{z}- \cJ(\uhss_\perp).
    \end{split}
  \end{equation}
  Integrating by parts in \eqref{eq:rely1} and using
  \eqref{eq:liftstab} we obtain
  \begin{equation}
    \label{eq:rely1b}
    \begin{split}
      \Norm{\D^{\tfrac{1}{2}}\nabla \qp{u - \uhss}}_{\leb{2}(\T{})}^2
      &=
      \int_{\T{}} (f + \div\qp{\D\nabla \uhss})(z - z_h )
      -
      \int_\cE \jump{\D\nabla \uhss} ( z - z_h )
      \\
      &\quad + \int_{\T{}}  \lift(\jump{\uhss}) \D\nabla (z - z_h) + \bih{\uhss_\perp}{z}- \cJ(\uhss_\perp)
      \\
      &\leq
      \sum_K h_K \Norm{f + \div\qp{\D\nabla \uhss}}_{\leb{2}(K)} h_K^{-1}\Norm{z - z_h }_{\leb{2}(K)}
      \\
      &\quad +
      \frac{1}{2} \sum_{e \in \partial K} h_e^{\frac{1}{2}} \Norm{\jump{\D\nabla \uhss}}_{\leb{2}(e)} h_e^{- \frac{1}{2}} \Norm{z - z_h }_{\leb{2}(e)}
      \\
      & \quad
      +
      C \Norm{h_e^{-\frac{1}{2}}\jump{\uhss} }_{\leb{2}(\cE)} \Norm{ \D^{\tfrac{1}{2}}\nabla (z - z_h)}_{\leb{2}(\Omega)}
      \\
      & \quad +
      C \Norm{\uhss_\perp}_{dG} \Norm{\D^{\tfrac{1}{2}}\nabla \qp{z - z_h}}_{\leb{2}(\Omega)}
      +
      \Norm{\cJ}_{\sobh{-1}(\W)} \Norm{\uhss_\perp}_{dG}
    \end{split}
 \end{equation}
 From \cite[Theorem 5.3]{HoustonPerugiaSchoetzau:2004} we obtain
 \begin{equation}\label{eq:uperp}
  \Norm{\uhss_\perp}_{dG} \leq C_P \Norm{h_e^{-\frac{1}{2}} \jump{\uhss}}_{\leb{2}(\cE)}
 \end{equation}
with a constant $C_P>0$ which is independent of $h$ but depends on the
shape regularity of the mesh and the polynomial degree and we also
note that
\begin{equation}\label{eq:stabcJ}
  \Norm{\cJ}_{\sobh{-1}(\Omega)}
  \leq
  C\Norm{\D^{\tfrac{1}{2}} \nabla \qp{u - \uhss}}_{\leb{2}(\T{})}.
\end{equation}

We insert \eqref{eq:uperp} and \eqref{eq:stabcJ} into
\eqref{eq:rely1b} and apply trace inequality and Cauchy-Schwarz
inequality and obtain
\begin{equation}
  \label{eq:rely2}
  \begin{split}
    \Norm{\D^{\tfrac{1}{2}}\nabla \qp{u - \uhss}}_{\leb{2}(\T{})}^2   
    &\leq
    \left( \sum_K \left(
    (\eta^{**}_K)^2
    +
    C \sum_{e \in \partial K} (\eta^{**}_e)^2 \right)\right)^{\frac{1}{2}}
    \Norm{h^{-1} (z - z_h) }_{\leb{2}(\Omega)}
    \\
    &\qquad +
    C \left( \sum_K      \sum_{e \in \partial K} (\eta^{**}_e)^2 \right)^{\frac{1}{2}}
    \Norm{\D^{\tfrac{1}{2}}\nabla (z - z_h) }_{\leb{2}(\Omega)}.
  \end{split}
\end{equation}
Now, we choose $z_h \in \fes_h^p \cap \hoz$
as the Cl\'ement interpolant of $z$ so that
\begin{equation}\label{eq:rely3}
\begin{split}
  \Norm{h^{-1} (z - z_h) }_{\leb{2}(\Omega)} + \Norm{\D^{\tfrac{1}{2}}\nabla (z - z_h) }_{\leb{2}(\Omega)}
  &\leq
  C \Norm{\nabla z }_{\leb{2}(\Omega)}
  \\
  &\leq C \Norm{\D^{\tfrac{1}{2}} \nabla \qp{u - \uhss}}_{\leb{2}(\T{})},
\end{split}
\end{equation}
and insert \eqref{eq:rely3} into \eqref{eq:rely2} to obtain the
assertion of the theorem.
\end{proof}

The error estimator, derived in Theorem \ref{the:primalerror2}, is
locally efficient in the following sense:
\begin{theorem}[Local efficiency]
  \label{the:efficiency}
  Assume $f$ and $\D$ are piecewise polynomial on $\T{}$. Then, there exists a
  constant $C>0$ independent of $h$ such that for any $K \in \T{}$ and
  any $e \in \cE$ the following estimates hold:
  \begin{equation}\label{eq:loceff}
    \eta^{**}_K \leq  C \Norm{\D^{\tfrac{1}{2}} \nabla \qp{u - \uhss}}_{\leb{2}(K)}
  \end{equation}
  and
  \begin{equation}
    \eta^{**}_e \leq C\Norm{\D^{\tfrac{1}{2}}\nabla \qp{u - \uhss}}_{\leb{2}(K_e)}
  \end{equation}
    where $K_e$ denotes the union of cells sharing common edge $e$.
\end{theorem}

\begin{proof}
 Both proofs are
  standard and follow \cite{Verfurth:1996}.
 \end{proof}

\begin{remark}[Data oscillation]
  In case $f$ or $\D$ are not polynomial the right hand side of
  \eqref{eq:loceff} contains additional data oscillation terms.
\end{remark}

\section{Post-Processors}
\label{sec:recovery}

In order to show the versatility of our results, we consider two
families of reconstruction operators. Namely, the
Smoothness-Increasing Accuracy-Conserving (SIAC) post-processing
\cite{Thomee1977,BrambleSchatz:1977,ryan2015} as well as
patch reconstruction via the Zienkiewicz and Zhu
\cite{ZienkiewiczZhu,ZHANG1998159} Superconvergent Patch
Recovery (SPR) technique. Below we outline the procedure for
performing these reconstructions as well as error estimates for the
ideal case.

\subsection{SIAC post-processors}
\label{sec:siac}
One example of a superconvergent post-processor that we examine is the
Smoothness Increasing Accuracy Conserving (SIAC) filter. The SIAC
filter has its roots in an accuracy-enhancing post-processor developed
by Bramble and Schatz \cite{BrambleSchatz:1977}. The original analysis
was done for finite element approximations for elliptic equations. 
This technique has desirable qualities including its locality,
allowing for efficient parallel implementations, and its effectiveness
in almost doubling the order of accuracy rather than increasing the
order of accuracy by one or two orders. This post-processor was also
explored from a Fourier perspective and for derivative filtering by
Thome\'{e} \cite{Thomee1977} and Ryan and Cockburn \cite{ryan2009}.

SIAC filters are an extension of the above ideas and have
traditionally been used to reduce the error oscillations and recover
smoothness in the solution and its derivatives for visualization
purposes
\cite{MirzaeeRyanKirby:2010,WalfischRyanKirbyHaimes:2009,Steffanetal}
or to extract accuracy out of existing code \cite{ryan2005}.
It has been extended to a variety of PDEs as well as meshes
\cite{convdiff}.  A quasi-interpolant perspective on SIAC can be found
in \cite{MirzargarRyanKirby:2016}.  The important property of these
filters is that, in addition to increasing the smoothness, for smooth
initial data and linear problems, the filtered solution is more
accurate than the DG solution.  To combat the high computational cost
of the tensor-product nature of the multi-dimensional kernel, a line
filter was introduced in
\cite{Docampo-SanchezRyanMirzargarKirby:2017}.


For ease of presentation the following discussion only details the
design of the filter and presents a-priori error estimates for the case of
a smooth solution. Although the discussion is limited to
one-dimension, it can be extended to Cartesian meshes in more than one
space dimension using a tensor product approach.  More advanced
applications of the multi-dimensional SIAC post-processor are the
Hexagonal SIAC \cite{Mahsa} or Line SIAC \cite{Docampo-SanchezRyanMirzargarKirby:2017}.

The basic idea is that the reconstruction is done via convolution post-processing:
\[
\uhs(\bar{x})=K^{2r+1,m+1}_{H} * u_h = \frac{1}{H}\int_{-\infty}^{\infty}\, K\left(\frac{\bar{x}-y}{H}\right)u_h(y)\, dy,
\]
where $h$ is the mesh size of the numerical scheme and $H$ is the scaling of the post-processor.  The convolution kernel, $K^{2r+1,m+1}(\cdot),$ is defined as
\[
K^{2r+1,m+1}(x) =
\sum_{\gamma=-r}^{r}c_{\gamma}^{2r+1,m+1}\psi^{(m+1)}\left(x-\gamma\right).
\]
This is a linear combination of $2r+1$ shifted copies of some function,  $\psi^{(m+1)}(x).$  The function weights are real scalars, $c_{\gamma}^{2r+1,m+1} \in \mathbb{R}.$
For the kernel, $r$ is chosen to satisfy consistency as well as $2r$ moment requirements, \ie polynomial reproduction conditions, which are necessary for preserving the accuracy of the Galerkin scheme  and $m$ is chosen for smoothness requirements. 
We focus on kernels built from B-splines which are defined via the B-Spline recurrence relation:
\begin{equation}\label{eq:spline-rec}
  \begin{split}
    \psi^{(1)}
    &=
    \chi_{[-1/2,1/2]}
    \\
    \psi^{(m+1)}
    &=
    \frac{1}{m}\left[\left(x+\frac{m+1}{2}\right)\psi^{(m)}\left(x+\frac{1}{2}\right)+\left(\frac{m+1}{2}-x\right)\psi^{(m)}\left(x-\frac{1}{2}\right)\right],
  \end{split}
 \end{equation}
for $m\geq1$.

%
%
%

\begin{remark}[Kernel scaling.]
For cartesian grids, the kernel scaling is typically chosen to be the element size, $H=h$.  In adaptive meshes or structured triangular meshes and tetrahedral meshes, $H$ is typically chosen to be the length of the mesh pattern \cite{StrucTri}. For Line SIAC, the kernel scaling is taken to be the element diagonal \cite{Docampo-SanchezRyanMirzargarKirby:2017}, for unstructured meshes, the kernel scaling is taken to be largest element side \cite{SIACunstrct,TetMesh}.
\end{remark}

It can be shown that when the solution is sufficiently smooth the
post-processed numerical solution $\uhs$ is a superconvergent
approximation.

In particular, if $u\in\cont{\infty}(\W)$, then the Galerkin solution
converges in Theorem \ref{the:error-est} as
\begin{equation}
  \Norm{ u - u_h}_{\leb{2}(\Omega)} = \Oh(h^{p+1}).
\end{equation}
If we choose
$r=p$ and $m = p-2$ then
\begin{equation}
 \Norm{ u - \uhs}_{\leb{2}(\Omega)} = \Oh(h^{2p}),
\end{equation}
see Theorem 1 in \cite{Thomee1977,BrambleSchatz:1977}, which for $p
\geq 2$ constitutes an improvement. It is possible to obtain the same
estimates in $\sobh{1}$ by taking higher order B-Splines.

In this paper, in order to apply the post-processor globally, we
mirror the underlying approximation as an odd function at the boundary as discussed in
\cite{BrambleSchatz:1977}.  

 \begin{remark}[Impact of in-cell regularity of $\uhss$]\label{rem:icr}
   If $\uhs,\uhss \not \in \sobh{2}(\T{})$ Theorem
   \ref{the:primalerror2} does not hold. Still, as long as $\uhs \in
   \sobh{1}(\T{})$ similar results can be obtained by slightly
   modifying the proof of Theorem \ref{the:primalerror2}.  One
   interesting example is SIAC reconstruction with $m=0$. In this
   case, for any $K \in \T{}$ the restriction $\uhss|_K$ contains
   several kinks, i.e. there are hypersurfaces (points for $d=1$, lines for $d=2$) across which $u$ is continuous but not differentiable.
   However, for any $K$ there exists a triangulation of  $\T{}_K$ of $K$, such that $\uhss|_K \in \sobh{2}(\T{}_K)$.
    If we follow the steps of the proof of Theorem
   \ref{the:primalerror2} we realise that integration by parts can
   only be carried out on elements of $\cup_{K \in \T{}} \T{}_K$ and
   each term $\Norm{h_K ( f + \Delta \uhss)}_{\leb{2}(K)}^2$ in the
   error bound needs to be replaced by
   \begin{equation}
     \sum_{T \in \T{}_K} \Norm{h_K ( f + \Delta \uhss)}_{\leb{2}(T)}^2  + \frac 1 2\sum_{e \in \cE_K}\Norm{h_K^{\frac{1}{2}} \jump{\nabla \uhss}  }_{\leb{2}(e)}^2,
   \end{equation}
   where  $\cE_K$ denotes the set of interior
   edges of $\T{}_K$.
  Efficiency of this modified estimator can be shown along the same
  lines as in Theorem \ref{the:efficiency} but bubble functions with
  respect to the elements and edges in the sub-triangulation $\T{}_K$
  need to be used.
 \end{remark}

\subsection{Superconvergent Patch Recovery}
\label{sec:patchrecovery}

The second post-processing operator we study is based on the
superconvergent patch recovery (SPR) technique. This was originally studied
numerically and showed a type of superconvergence for elliptic
equations using finite element approximations
\cite{Zienkiewicz:1987}. The mathematical theory behind this recovery
technique was addressed by Zhang and Zhu \cite{ZhangZhu:1995} for the
two-point boundary value problems and for two-dimensional problems and
extended to parabolic problems in
\cite{Leykekhman:2006,LakkisPryer:2011c}. The superconvergent patch
recovery method works by recovering the derivative approximation
values for one element from patches surrounding the nodes of that
element using a least squares fitting of the superconvergent values at
the nodes and edges.  In typical derivative recovery, the derivative
approximation is a continuous piecewise polynomial of some given degree. For
overlapping patches, the recovered derivative is just an average of
the approximations obtained on the surrounding patches.  Unlike SIAC
post-processing, this recovery technique does not rely on translation
invariance for the high-order recovery.  
The superconvergent patch recovery technique has been shown
to work well for elliptic equations that have a smooth solution, and
for less smooth solutions with a suitably refined mesh.

The usual application
of this technique is for gradient recovery.  However, in this article
we apply this technique to recover function values.

As mentioned, we suitably modify the algorithm
{to construct a function $\uhs \in \fes_h^*$ with
$\fes_h^*=\fes_h^{2p+1}\cap C^0(\Omega)$}.
The construction of $\uhs$, given some finite element function
$u_h$, is carried out in two steps:
\begin{enumerate}
  \item Construct a polynomial $q_i$ of order $2p$ at each
    node $v_i$ of the mesh using a least squares fitting of function
    values of $u_h$ evaluated at suitable points in elements surrounding
    $v_i$.
  \item Given an element $K$ we use linear interpolation of the
    values of $q_i$ for the three nodes of $K$ to compute
    $\uhs\in\fes_h^*$.
\end{enumerate}
{There are many approaches for constructing the polynomials $q_i$ in the first step at a given node $v_i$ with surrounding triangles $K'$. For our experiments, we use the following approach. For $p=1$, we construct a quadratic polynomial $q_i$ by fitting the values of $u_h$ at the nodes of all $K'$. For a piecewise quadratic
$u_h$ ($p=2$) we also use the midpoints of all edges of the $K'$.
Finally for our tests with
$p=3$ we evaluate $u_h$ at two points on each edge chosen
symmetrically around the midpoint of the edge
(we use the Lobatto points with local coordinates
$\frac{1}{2}\pm\frac{\sqrt{5}}{10}$) and also add the value of $u_h$ at the barycentre
of the $K'$. This is depicted in Figure~\ref{fig:sprsketch}.}
To guarantee that we have enough function values to compute the
least squares fits, we add a second layer of triangles around $v_i$ if necessary, e.g., at boundary nodes.

Note that this procedure is similar, although not the same,
as the approach investigated in \cite{ZhangNaga:2005}.
Another related procedure was proposed in \cite{Ovall:2007}.

\begin{figure}
\centering
\includegraphics[width=0.8\textwidth]{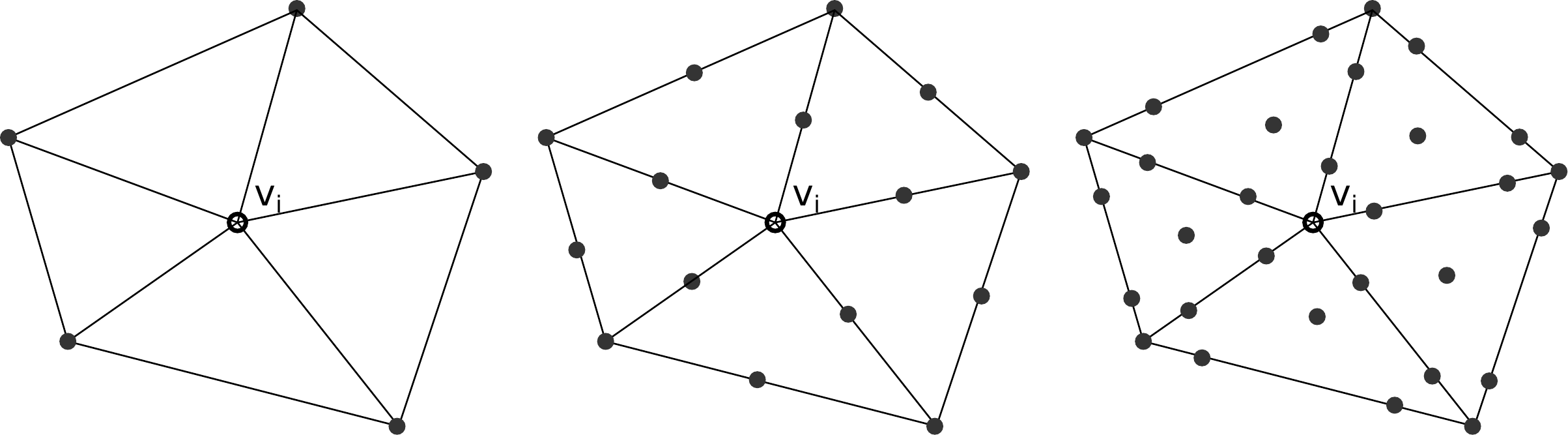}
\label{fig:sprsketch}
\caption{{Evaluation points of $u_h$ used for the least squares fit of a polynomial at node $v_i$ for $p=1,2,3$ (from left to right).}}
\end{figure}

\section{Numerical Results}
\label{sec:numerics}
In this section we study the numerical behaviour of the error
indicators proposed for the SIAC and SPR post-processing operators. We
compare this behaviour with the true error on some typical model problems. The
computational work was done in the DUNE package
\cite{BastianBlattDednerEngwerKlofkornKornhuberOhlbergerSander:2008}
based on the new Python frontend for the DUNE-FEM module
\cite{Dune-Fem,Dune-Py}.

\subsection{Smoothness-Increasing Accuracy-Conserving post-processors}

The implementation of the post-processor
is done through simple matrix-vector multiplication and is discussed
in \cite{mirzaee2012}.

We first investigate the behaviour of the error and the residual
estimator for the problem \eqref{ellipt} in one space dimension with $\D = 1$, i.e. the
Laplace problem
\begin{equation}
\begin{split}
- u'' &= f \text{ in } \W
\\
u&=0  \text{ on } \partial \W,
        \end{split}
\end{equation}
where the forcing function $f$ is chosen so that
the exact solution is
\begin{equation}
u(x)=\sin{6\pi x}^2\cos{\frac{9}{2}\pi x}
\end{equation}
on the interval $(0,1)$. We show both the $\leb{2}$ and $\sobh1$ errors
for the Galerkin approximation $u_h$, the SIAC postprocessed approximation, $\uhs$
and the orthogonal postprocessor, $\uhss$. We also show the two
residual indicators $R_h$ (from Theorem \ref{the:error-est}) and $R^{**}$ 
from (Theorem \ref{the:primalerror2}). For the basis we consider the continuous Lagrange
polynomials for $u_h$ and impose the boundary conditions weakly with a penalty
parameter $\frac{10p^2}{h}$, where $p$ is the polynomial degree and $h$
is the grid spacing. Additional experiments were conducted using a discontinuous Galerkin
approximation, but no significant differences in the outcome where found and therefore do not include
the results.  We solve the resulting
linear system using an exact solver \cite{Davis:2004} to avoid issues
with stopping tolerances.

We will mainly focus on $p=2$ but also show results for $p=1$ and
$p=3$.  The SIAC postprocessing is constructed using a continuous
B-spline, $m=1,$ as well as setting $r=\ceil{\frac{p+1}{2}}.$ %
This leads to an inner stencil of
$2\ceil{r+\frac{1}{2}-1}+1 = 2\ceil{\frac{p+1}{2}+3}$ elements. 
We also tested other choices of
$r,\, m$ for $p=2$ but the above choice provided the best results and 
these are the results shown.

In Figure~\ref{fig:p2errors} we show the errors for $p=2$ for a series
of grid refinement levels starting with $20$ intervals and doubling
that number on each level. In Figure~\ref{fig:p2eocs} we plot the
corresponding Experimental Orders of Convergence (EOCs).  As can clearly be
seen, SIAC postprocessing ($\uhs$) improves the convergence
rate in $\sobh 1$ from $2$ to $3$ and in $\leb 2$ from $3$ to $4$.  While in
$\sobh 1$ the Galerkin orthogonality trick only leads to a small
improvement in the error, in $\leb 2$ we see an improvement of a full
order leading to a convergence rate of $5$. As expected from the
theory the residual indicators follow the $\sobh 1$ errors of $u_h$ and
$\uhss$ closely.  The efficiency index is comparable between $R_h$ and
$R^{**}$.

For a better understanding of how the error is reduced by utilizing SIAC
postprocessing and  the Galerkin orthogonality treatment, we show the
pointwise errors of the approximations in Figure~\ref{fig:p2pointwise}. It is
evident that the function values are much smoother when applying SIAC.  The move
from $\uhs$ and $\uhss$ does reintroduce  small scale errors, but at a
far lower level compared to the original approximation, $u_h$. As expected
from the errors, the differences in $\sobh 1$ are less pronounced.

\begin{figure}
\centering
\includegraphics[width=0.45\textwidth]{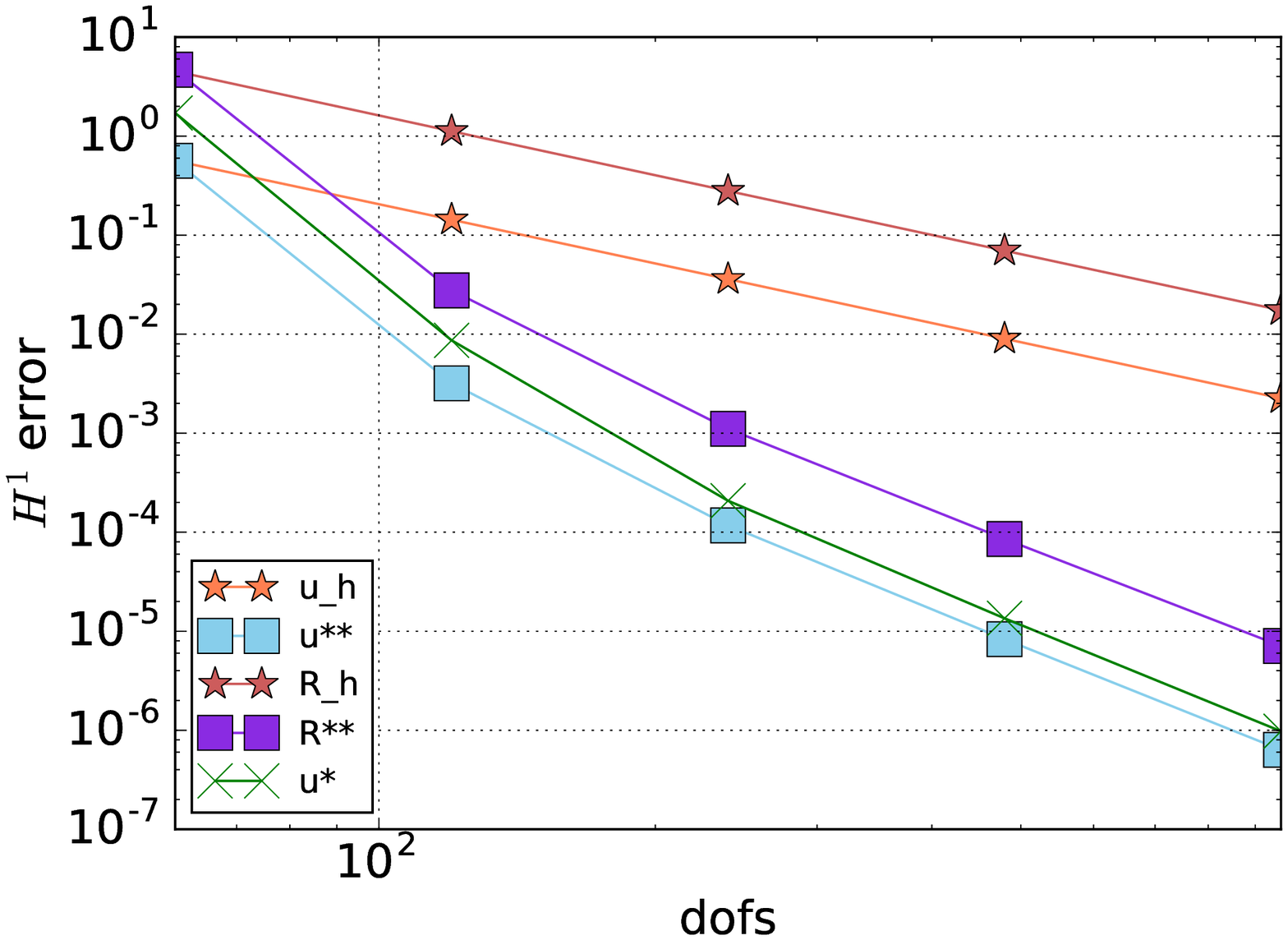}
\includegraphics[width=0.45\textwidth]{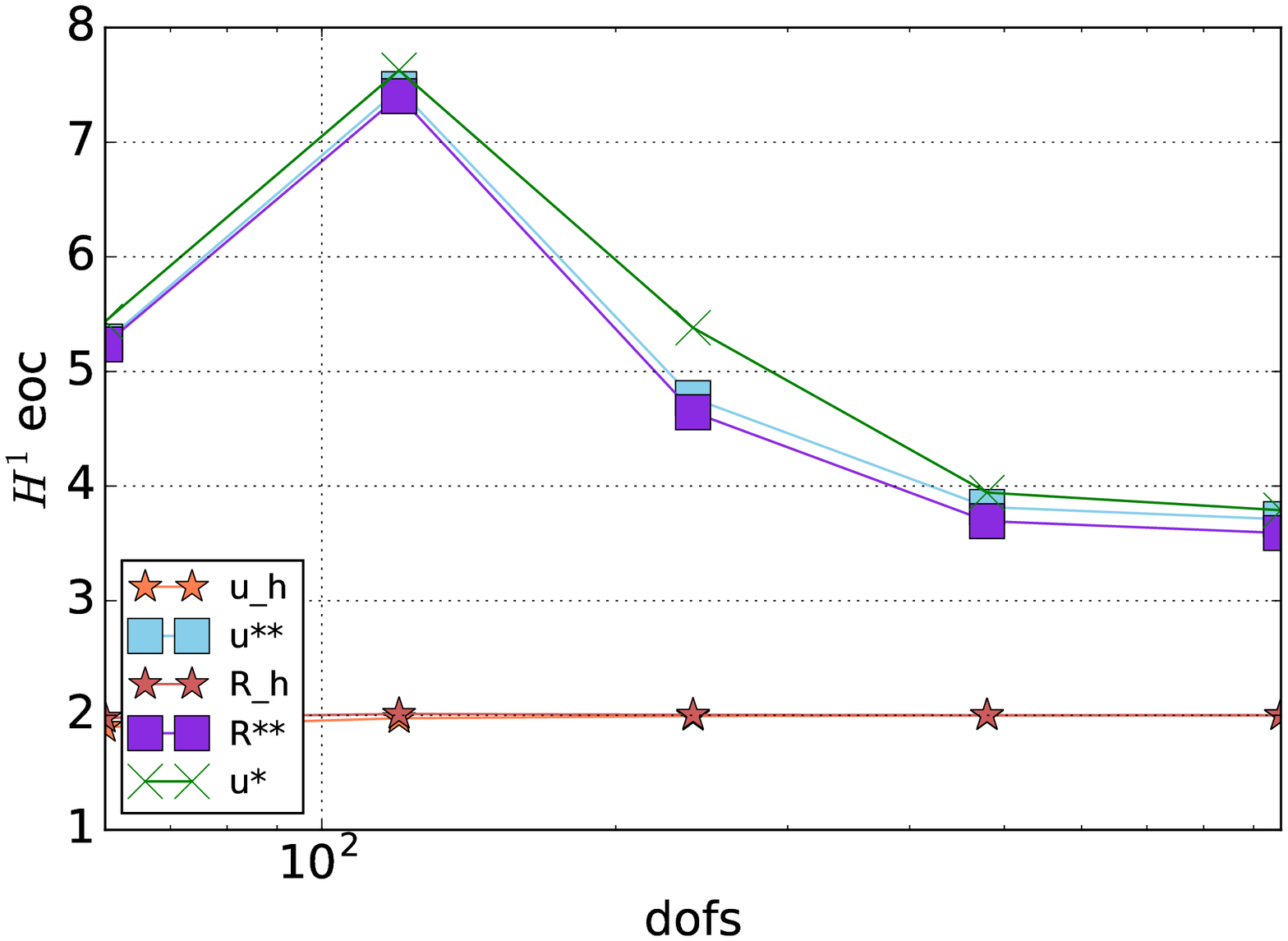}
\\
\includegraphics[width=0.45\textwidth]{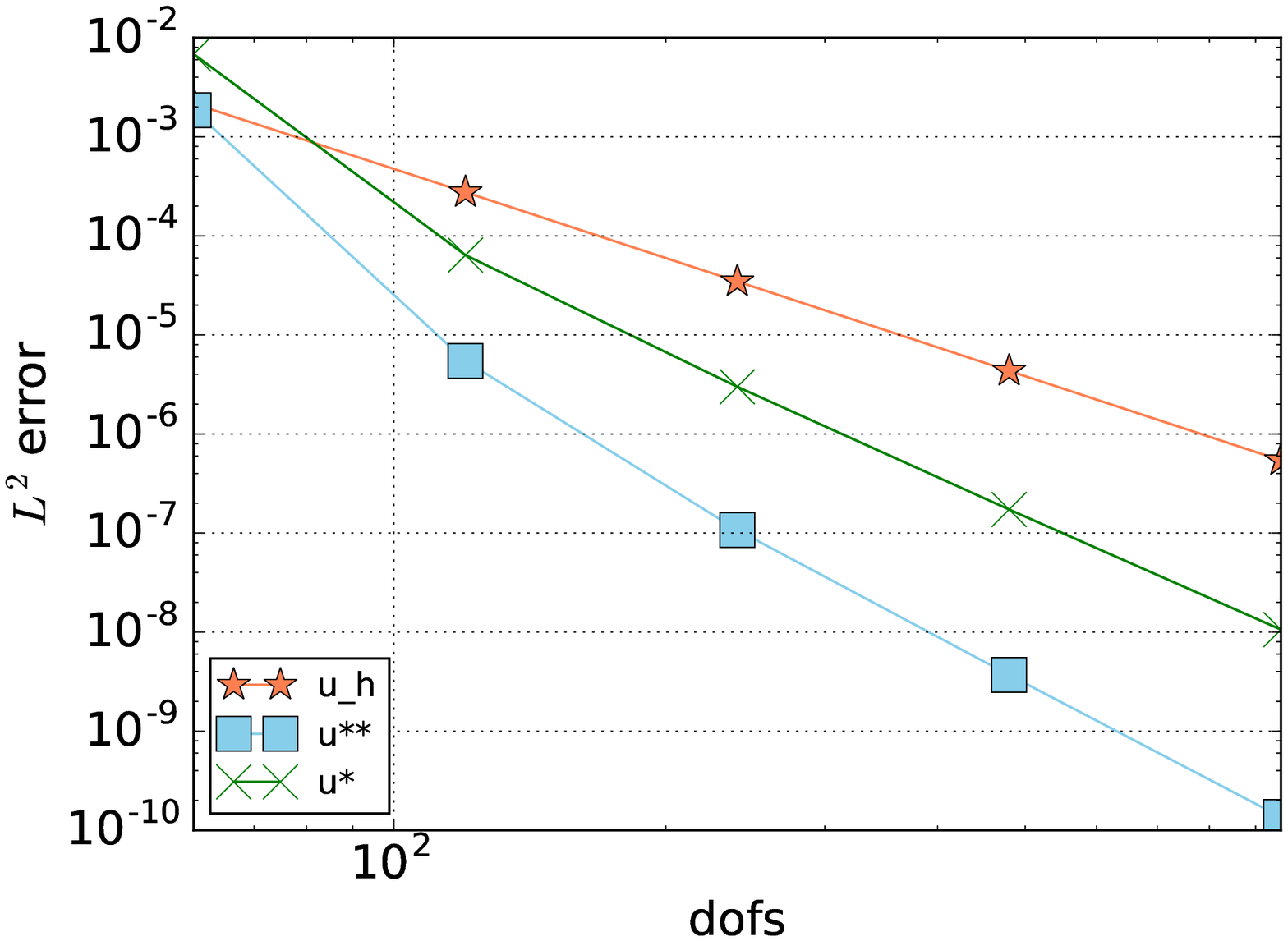}
\includegraphics[width=0.45\textwidth]{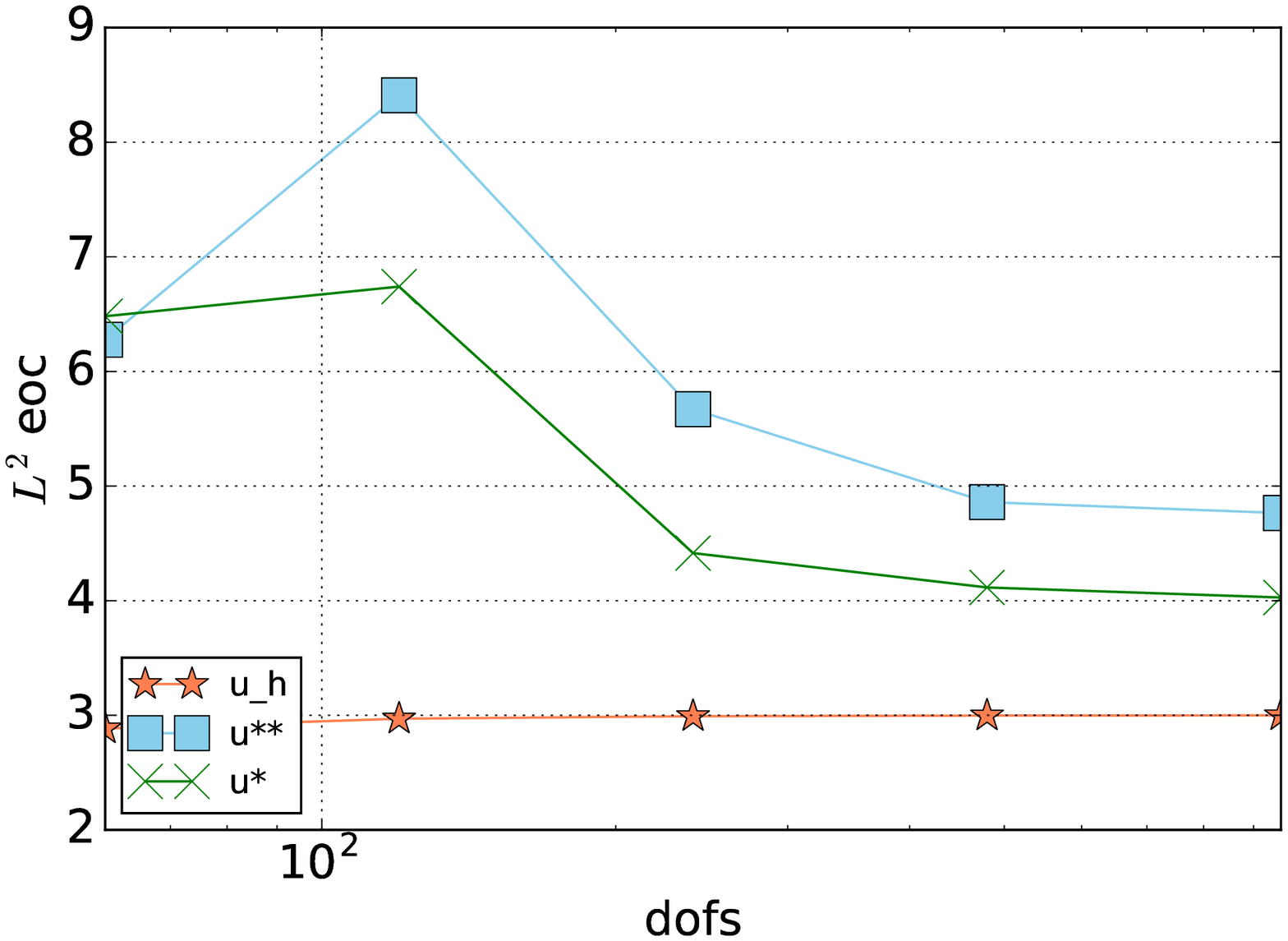}
\caption{Errors and convergence rates for $\sobh1$ (left two) and $\leb2$
  (right two) for polynomial degree $p=2$ using a SIAC reconstruction.}
\label{fig:p2errors}
\label{fig:p2eocs}
\end{figure}

\begin{figure}
\centering
\includegraphics[width=0.45\textwidth]{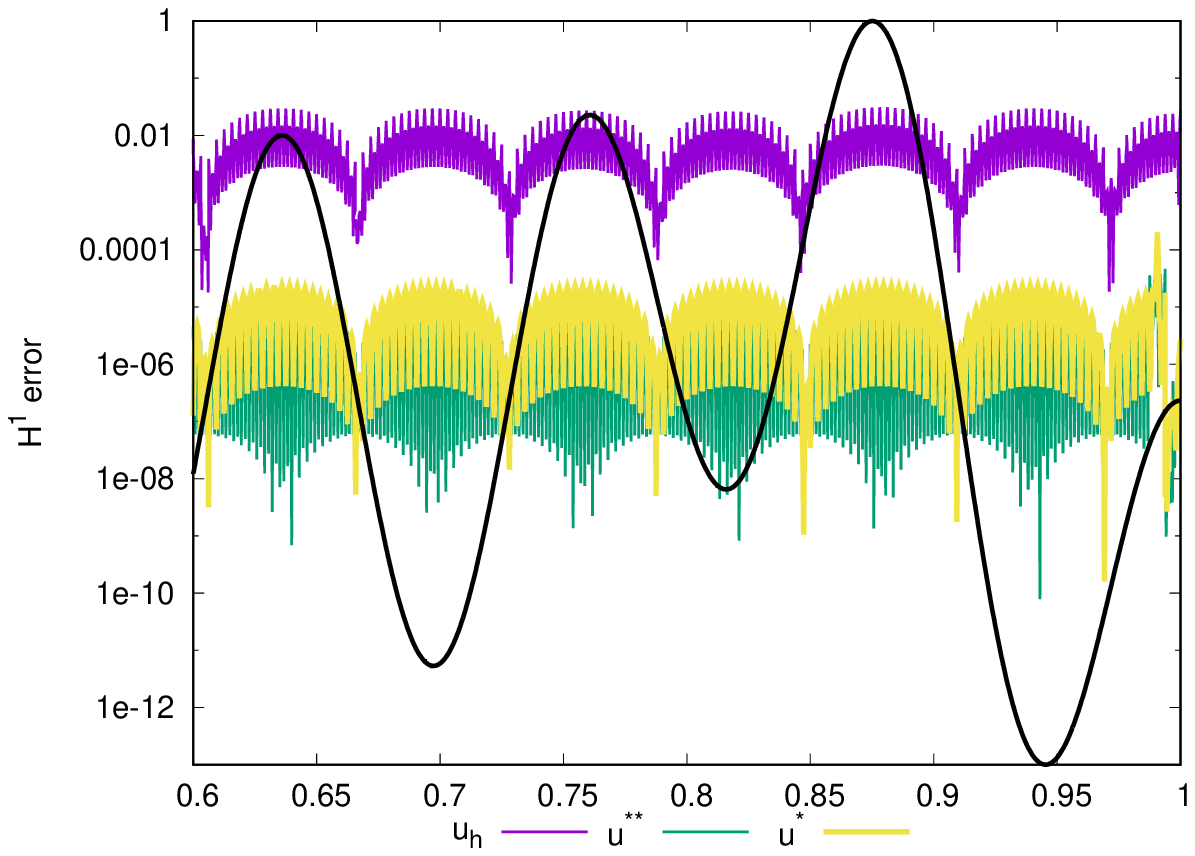}
\includegraphics[width=0.45\textwidth]{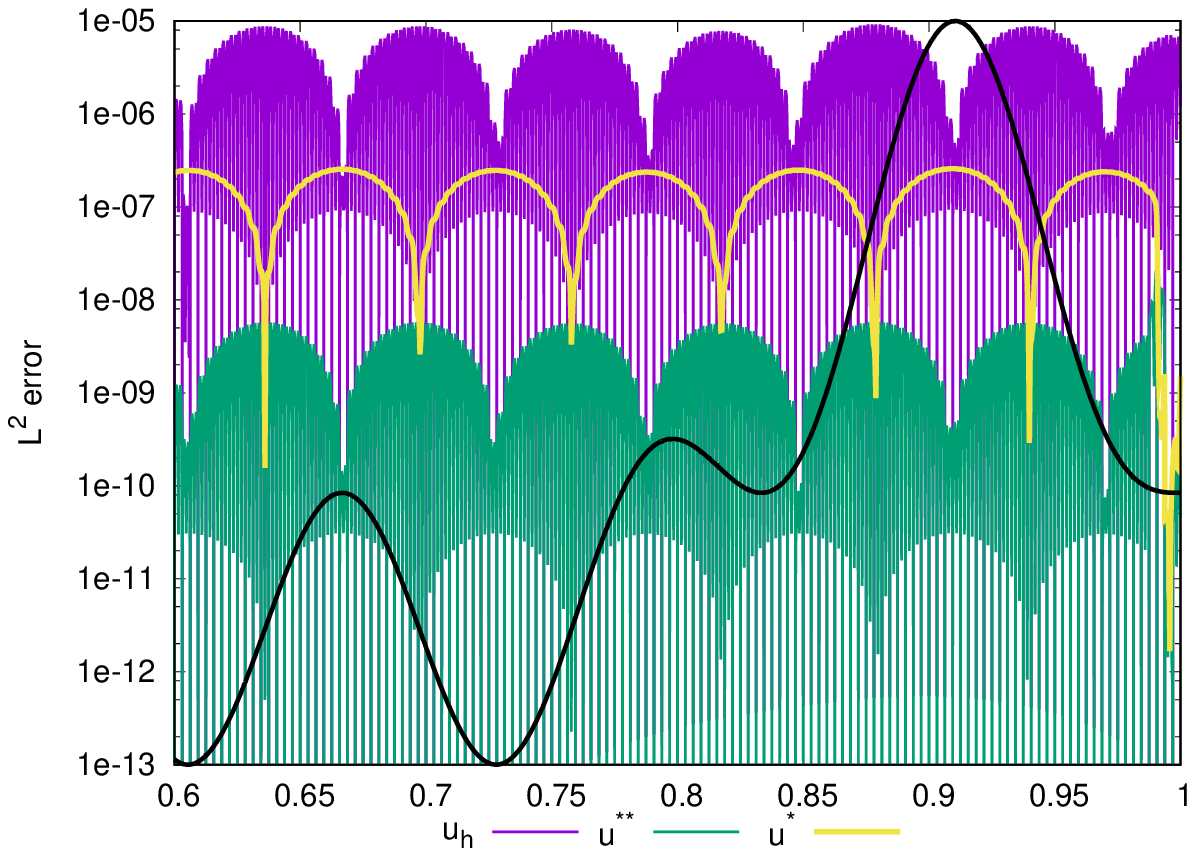}
\caption{Pointwise errors for the three different solutions on the right half
of the interval, evaluating the function in a number of points per
interval. These are results with $p=2$ and $h=1/320$ showing the errors in
the gradient (left) and in the function values (right).}
\label{fig:p2pointwise}
\end{figure}

In Figures \ref{fig:p3errors} we show errors and EOCs for $p=3$.
Due to the very low errors on the final grid the actual convergence rates
for $\uhs$ and $\uhss$ is not clear.  However, the improvement due to the Galerkin orthogonality trick, especially
in $\leb2$  is quite noticable as it reduced 
the error by two orders of magnitude.
\begin{figure}
\centering
\includegraphics[width=0.45\textwidth]{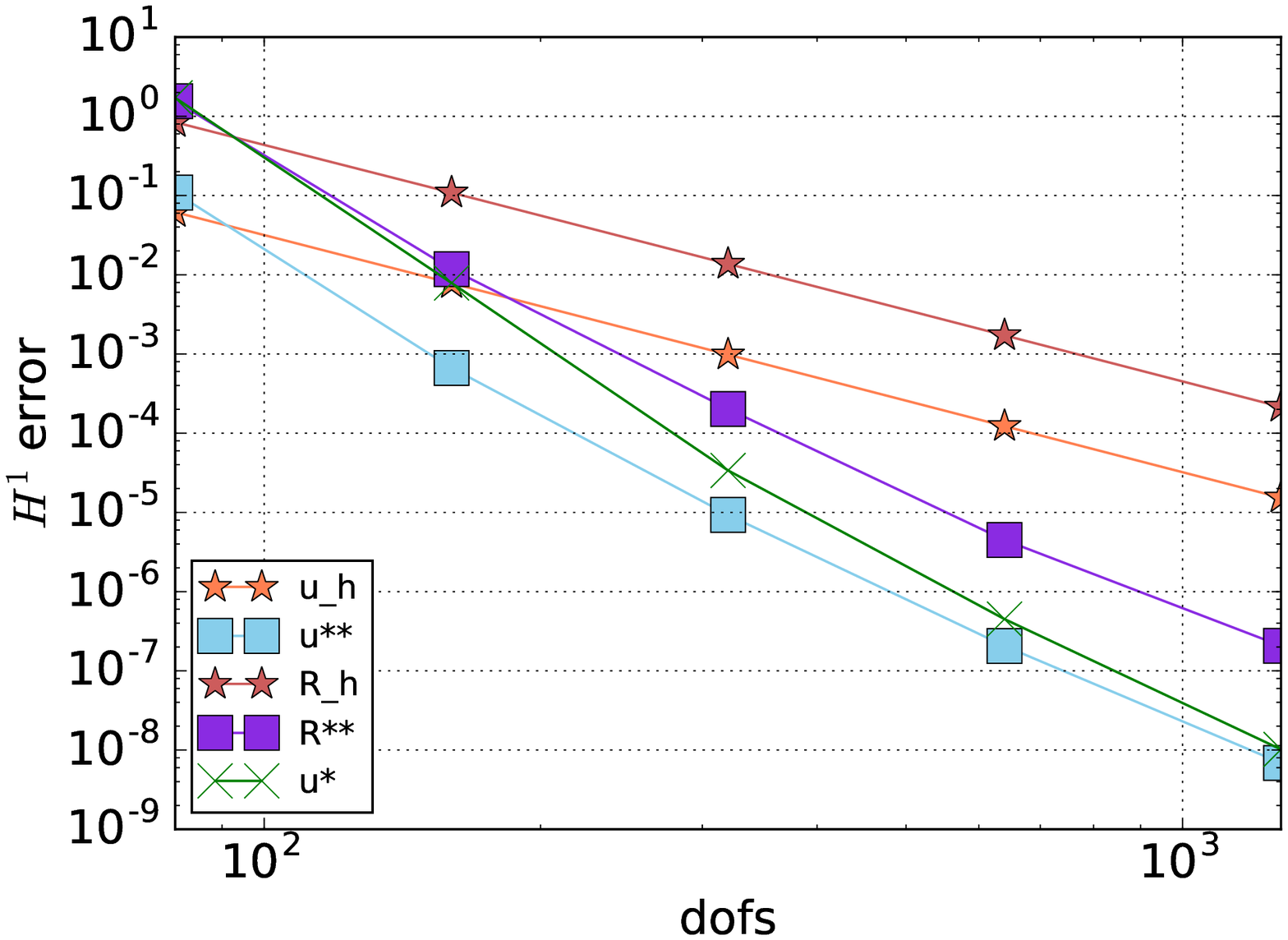}
\includegraphics[width=0.45\textwidth]{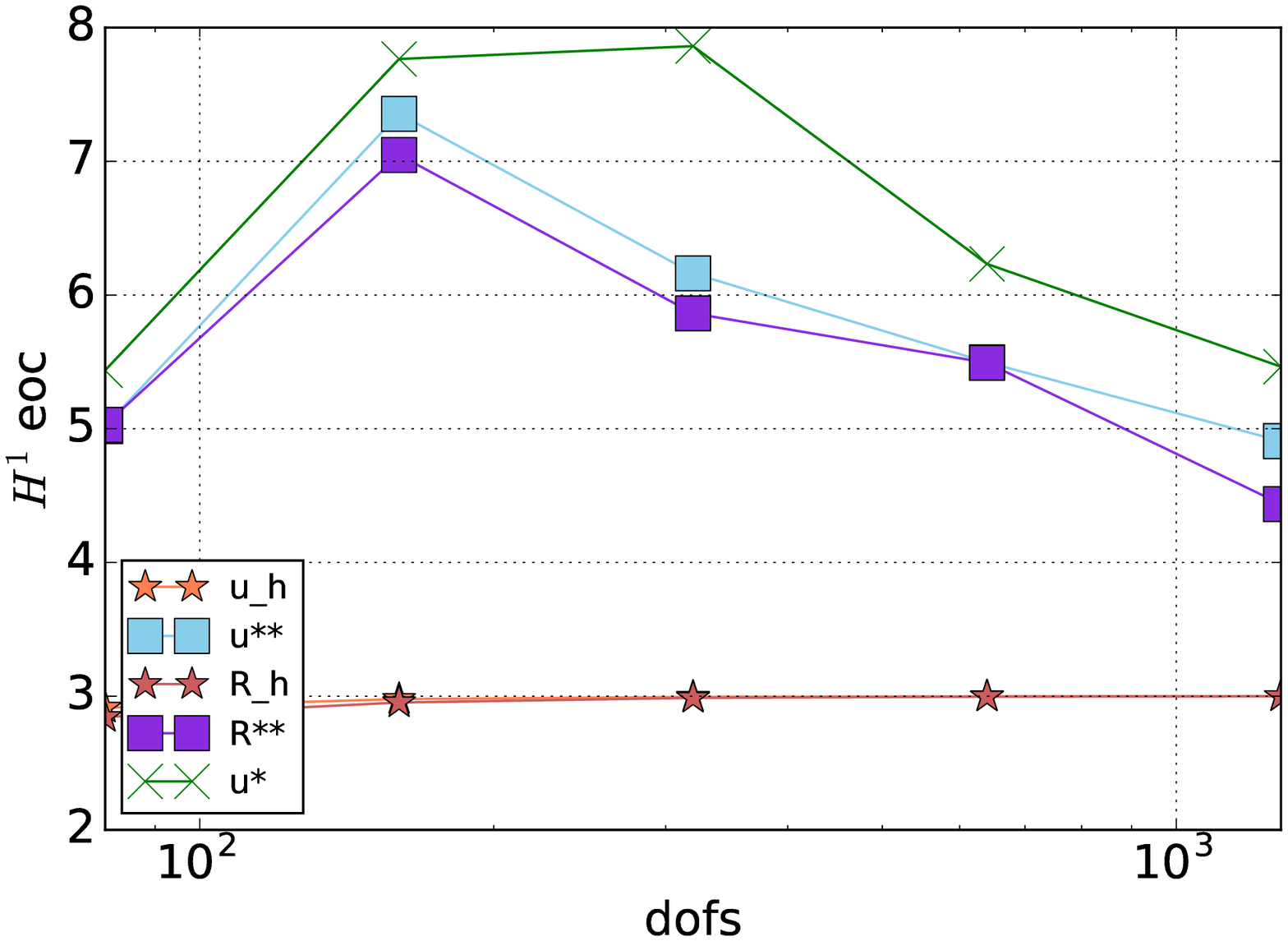}
\\
\includegraphics[width=0.45\textwidth]{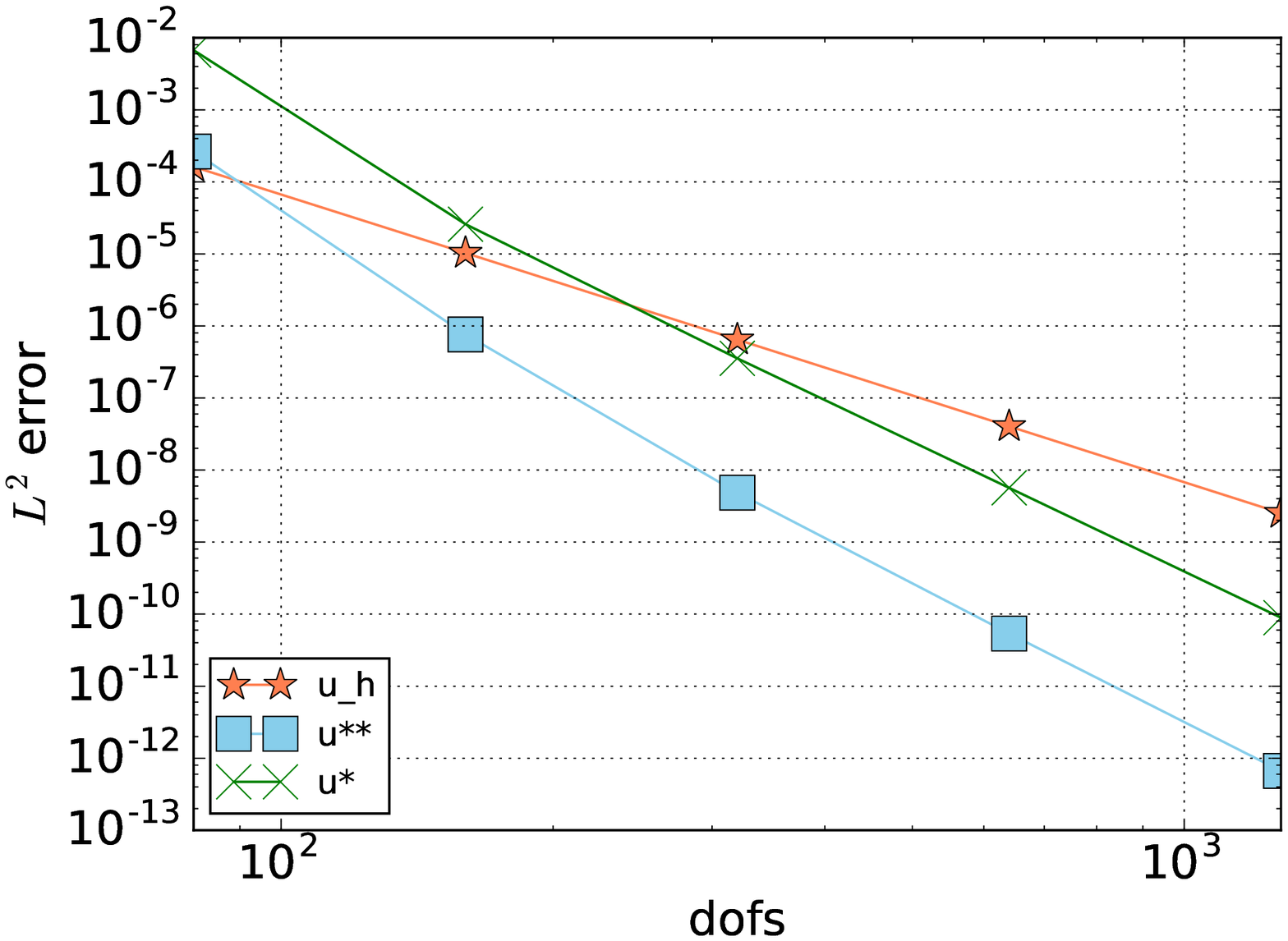}
\includegraphics[width=0.45\textwidth]{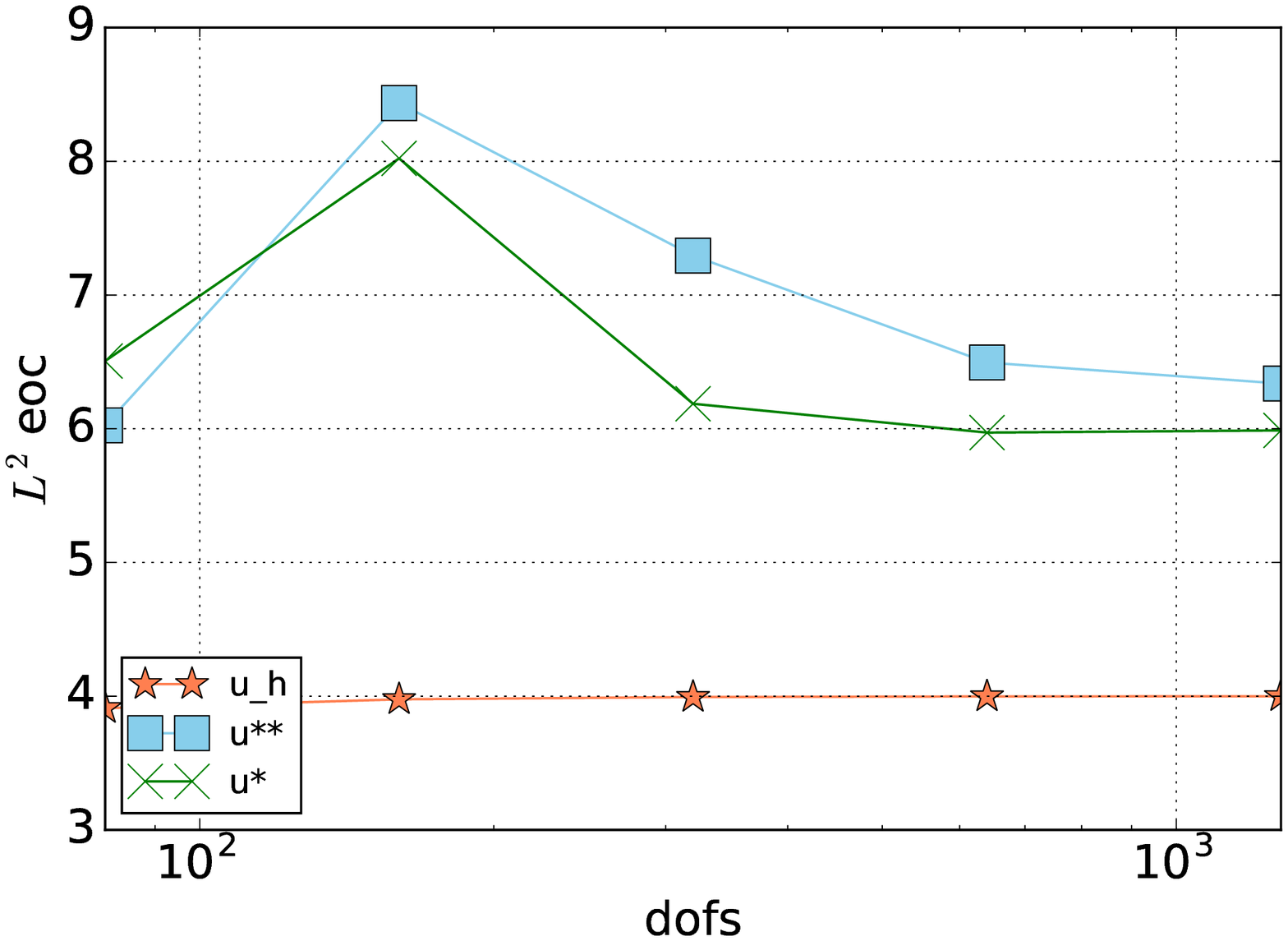}
\caption{Errors and convergence rates for $\sobh1$ (left two) and $\leb2$
  (right two) for polynomial degree $p=3$ using SIAC reconstruction.}
\label{fig:p3errors}
\label{fig:p3eocs}
\end{figure}


We next show results for $p=1$ in Figure~\ref{fig:p1errors}.  Again, there is 
a clear reduction in the values of the errors from $u_h$ to $\uhs$ to
$\uhss$ in $\sobh1$ together with an improvement in the convergence rate due to
the SIAC postprocessing.  This improvement in the convergence rate is about $1$ order.
While the convergence rate  going from $u_h$
to  $\uhss$ seems to only be half an order on the higher grid resolution, it is important to
note that the error 
using $\uhss$ is still significantly smaller
than the error between the exact solution and $u^{*}$ by at least a factor
of $2$.  Hence the results do not contradict the theory.
In $\leb2$, SIAC leads to no improvement while the convergence rate of the
error using $\uhss$ is at least half an order higher. Overall the
improvement in the convergence rate is not quite as good as for the higher
polynomial degrees.
The following tests summarized in
Figure~\ref{fig:p1superweakerrors}
show that the weak form of the boundary conditions is responsible for the
reduced order improvement. The figure shows results using a hyperpenalty of
the form $\frac{10p^2}{h^2}$. Applying this hyperpenalty term leads to 
 improvements that are again more in line with our observations for higher order
polynomials.  We note that strong enforcement of the boundary conditions also lead to similar
results.

\begin{figure}
  \centering
  \includegraphics[width=0.45\textwidth]{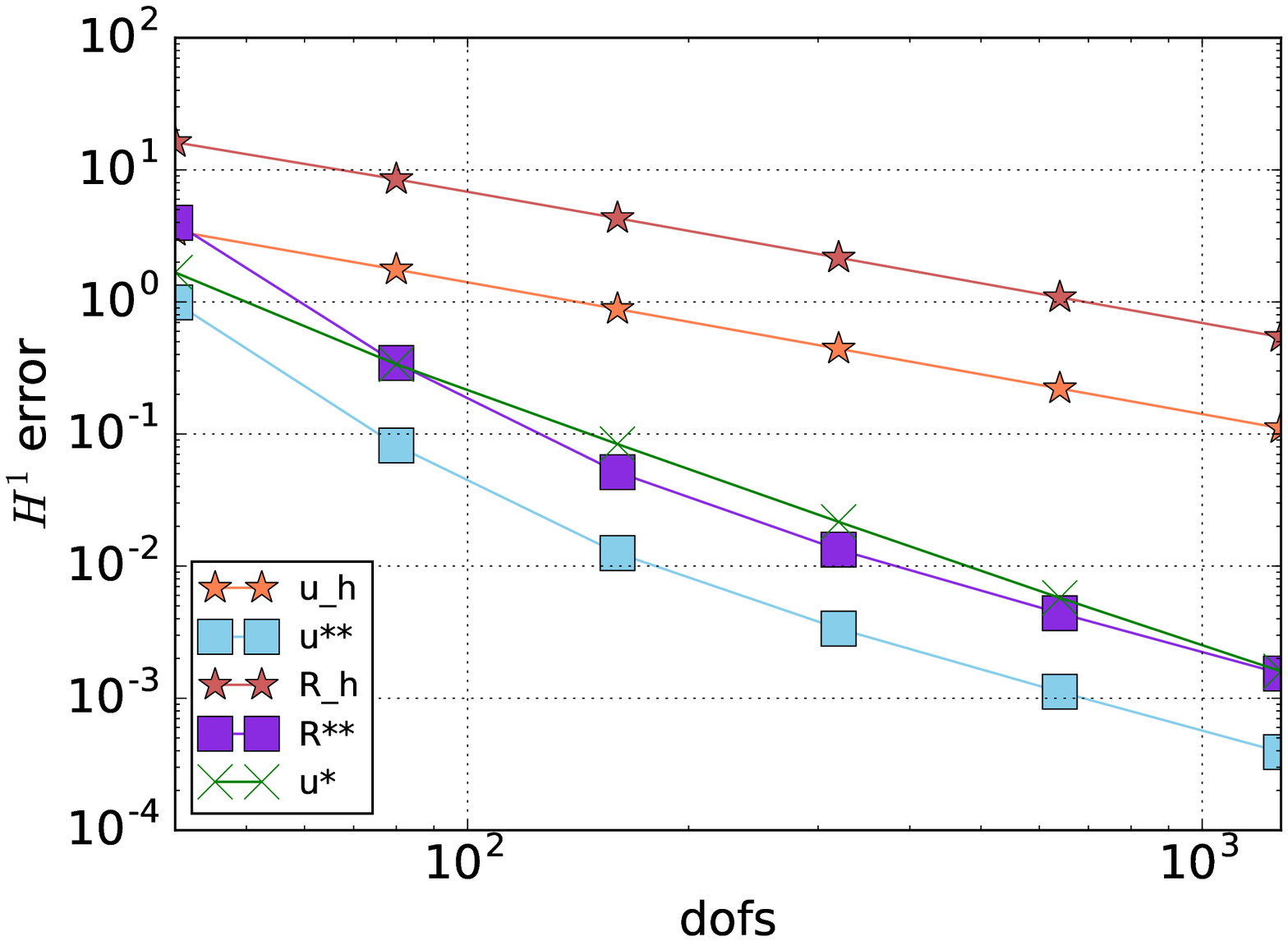}
  \includegraphics[width=0.45\textwidth]{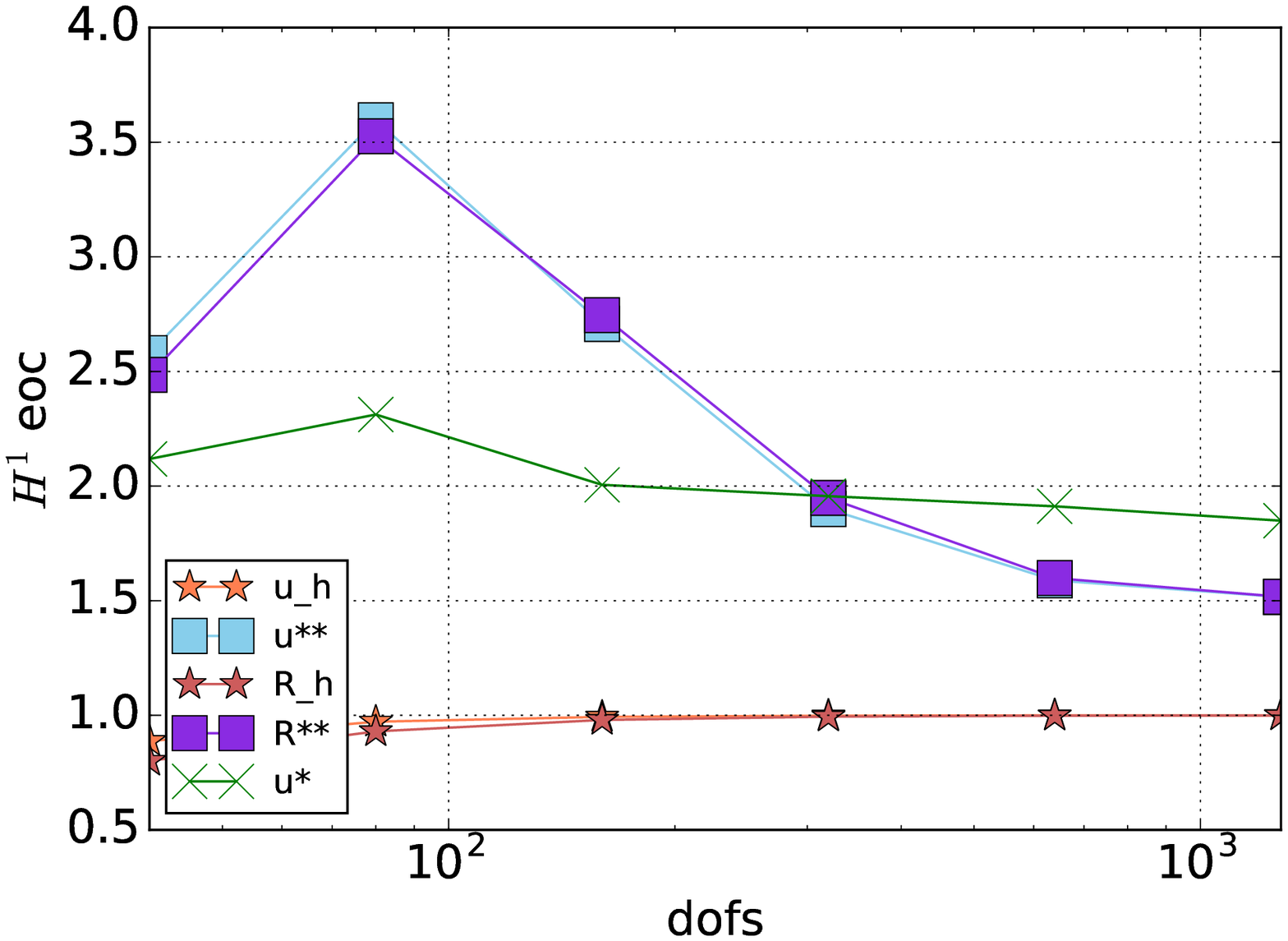}
  \\
  \includegraphics[width=0.45\textwidth]{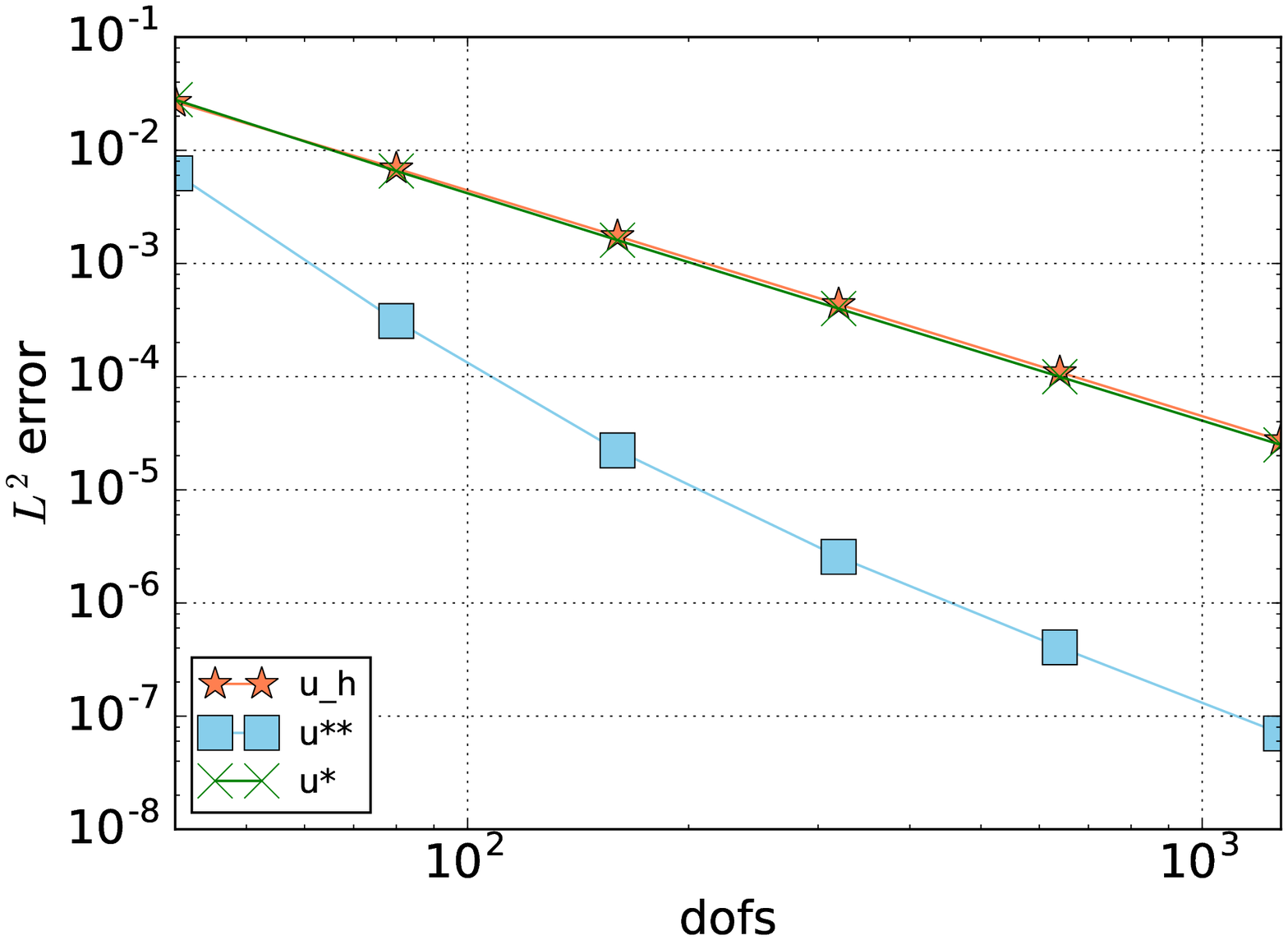}
  \includegraphics[width=0.45\textwidth]{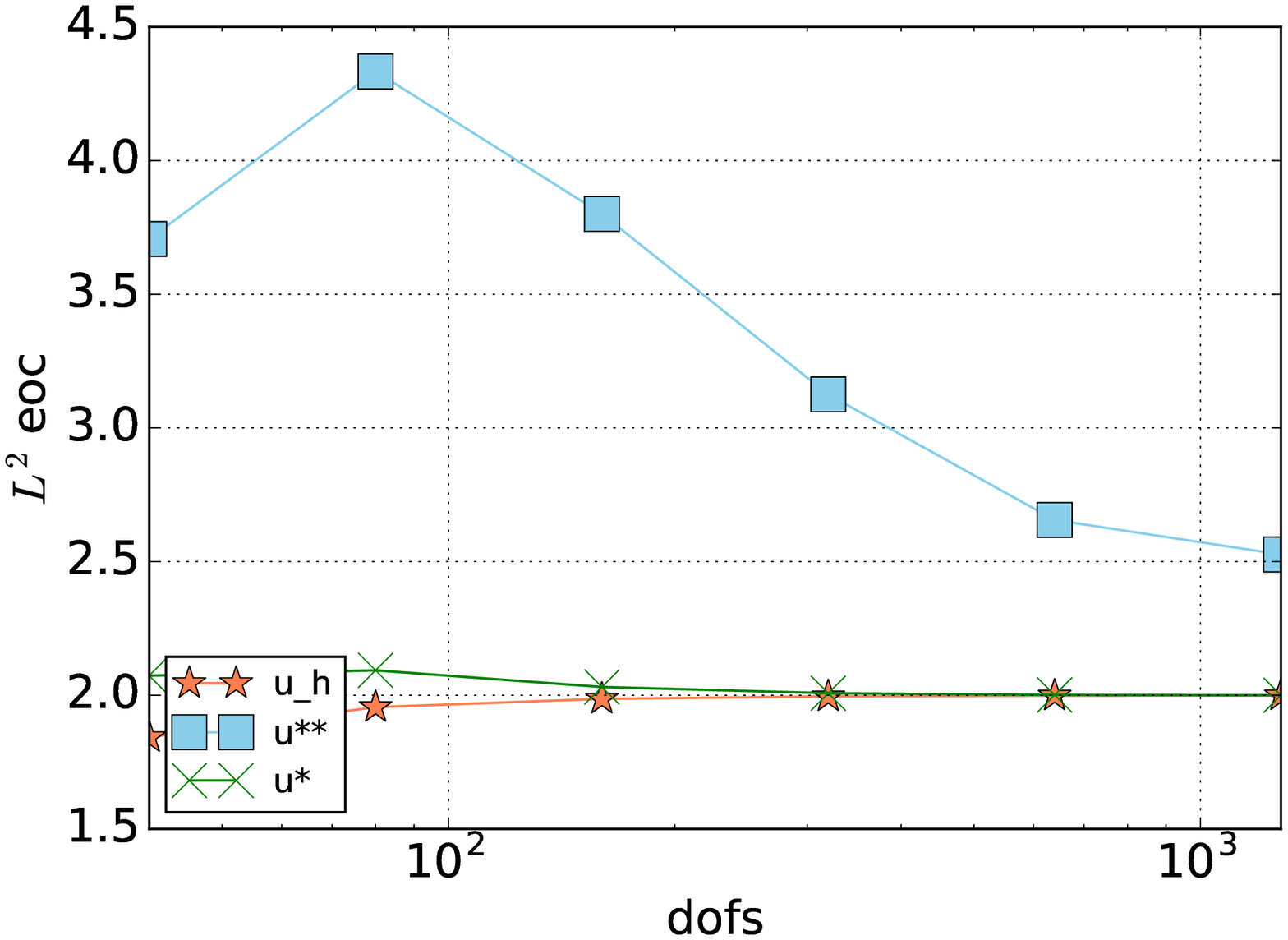}
\caption{Errors and convergence rates for $\sobh1$ (left two) and $\leb2$
  (right two) for polynomial degree $p=1$ using SIAC reconstruction.}
\label{fig:p1errors}
\label{fig:p1eocs}
\end{figure}


\begin{figure}
  \centering
  \includegraphics[width=0.45\textwidth]{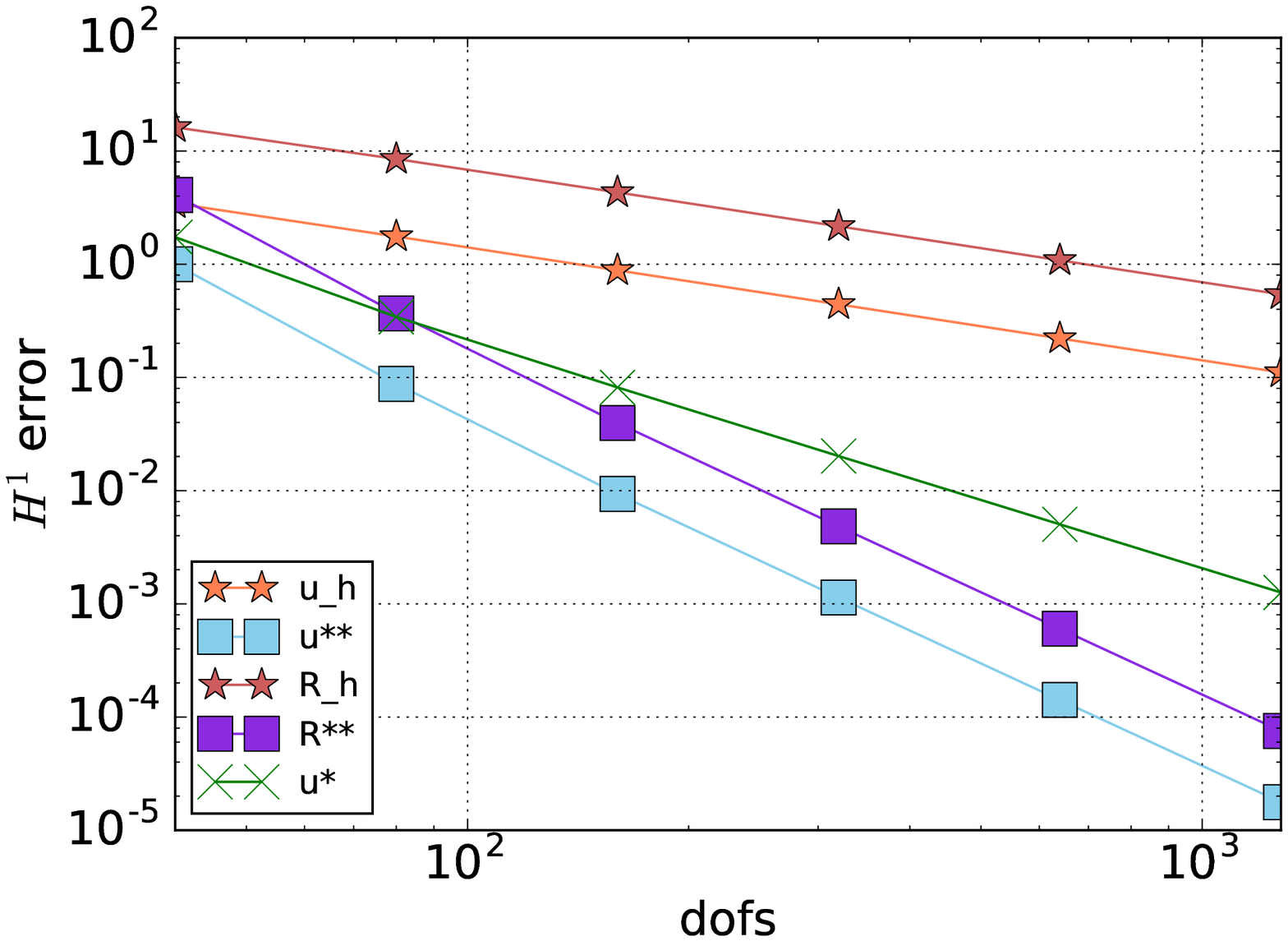}
  \includegraphics[width=0.45\textwidth]{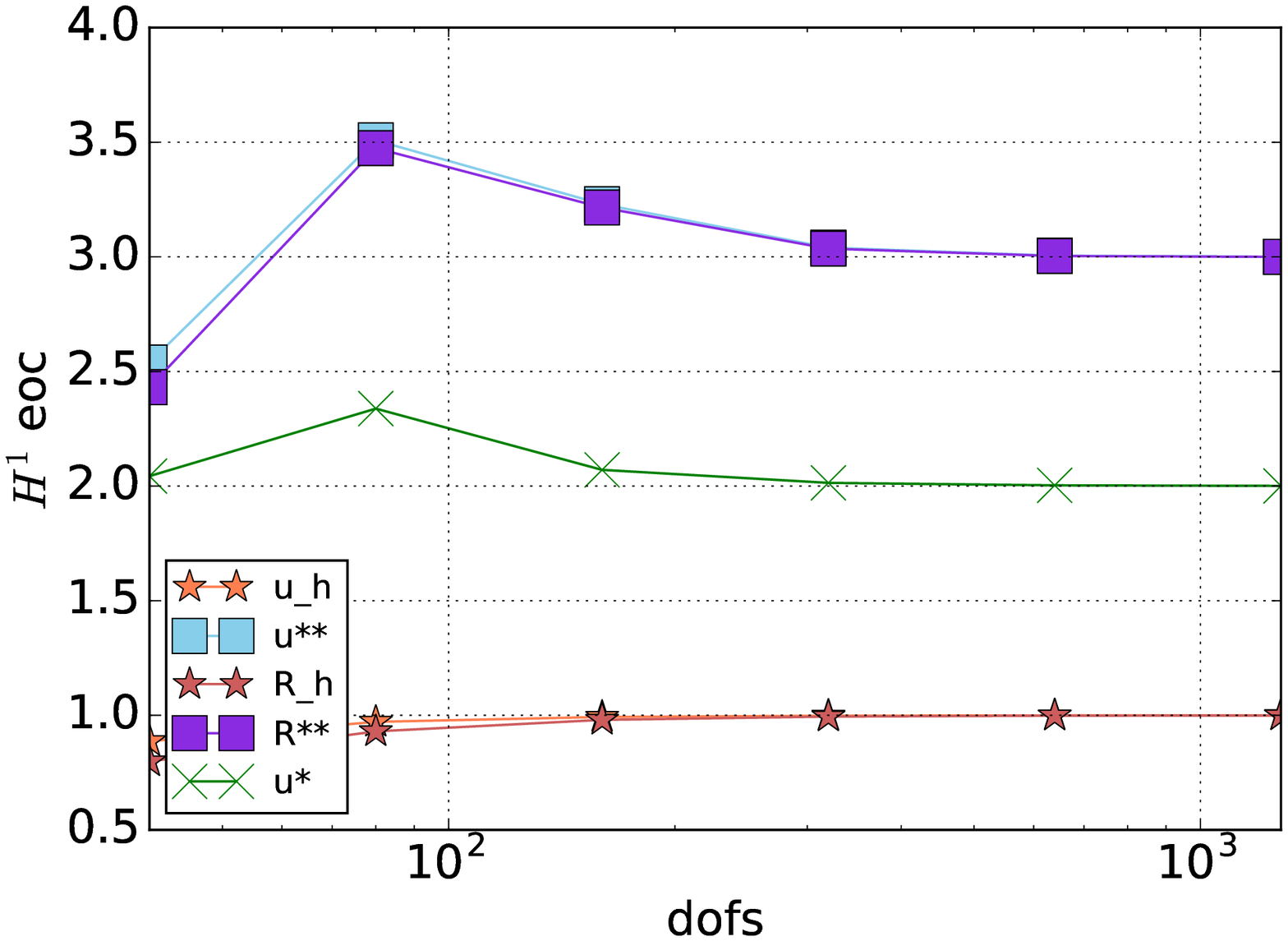}
  \\
  \includegraphics[width=0.45\textwidth]{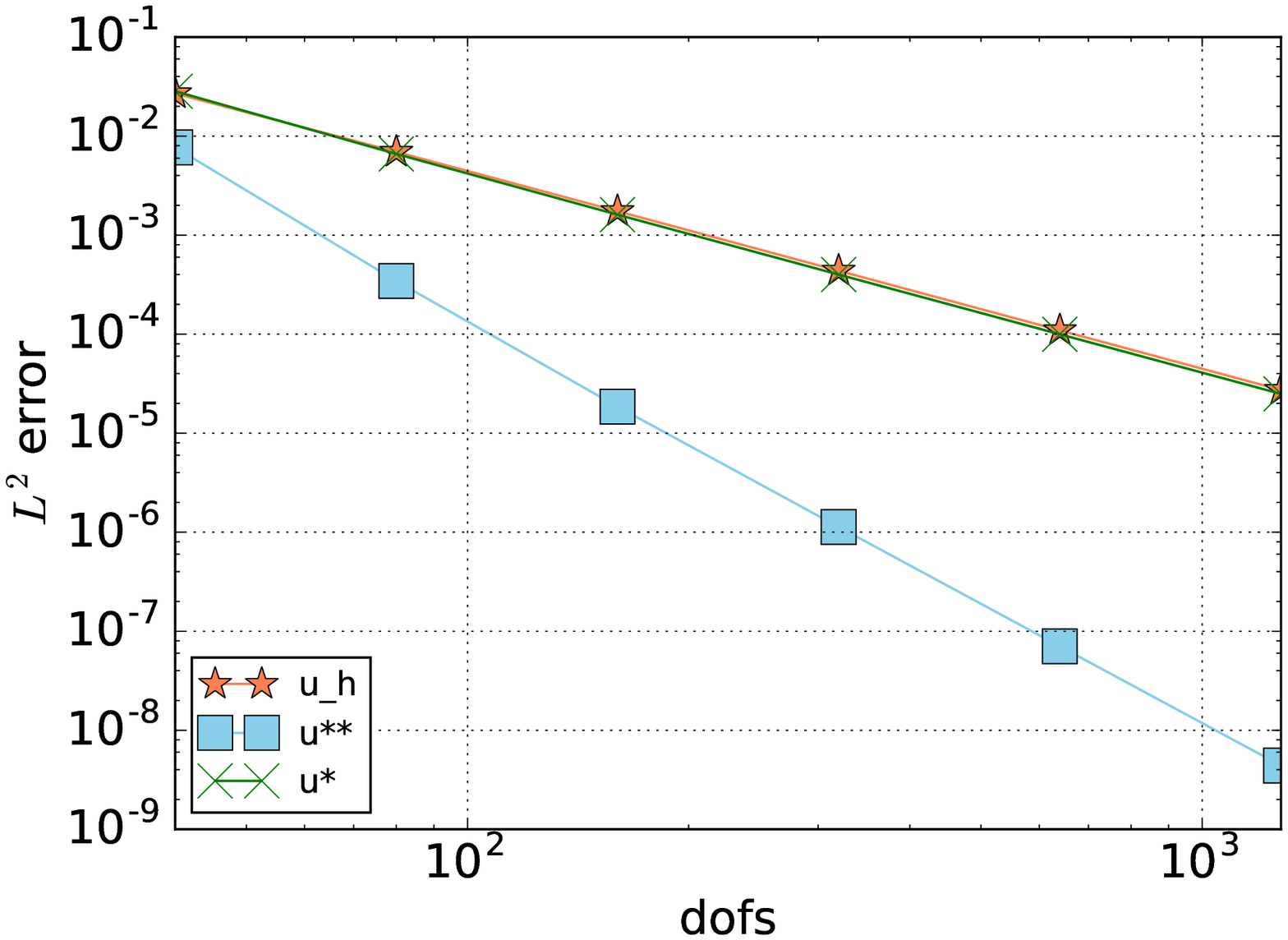}
  \includegraphics[width=0.45\textwidth]{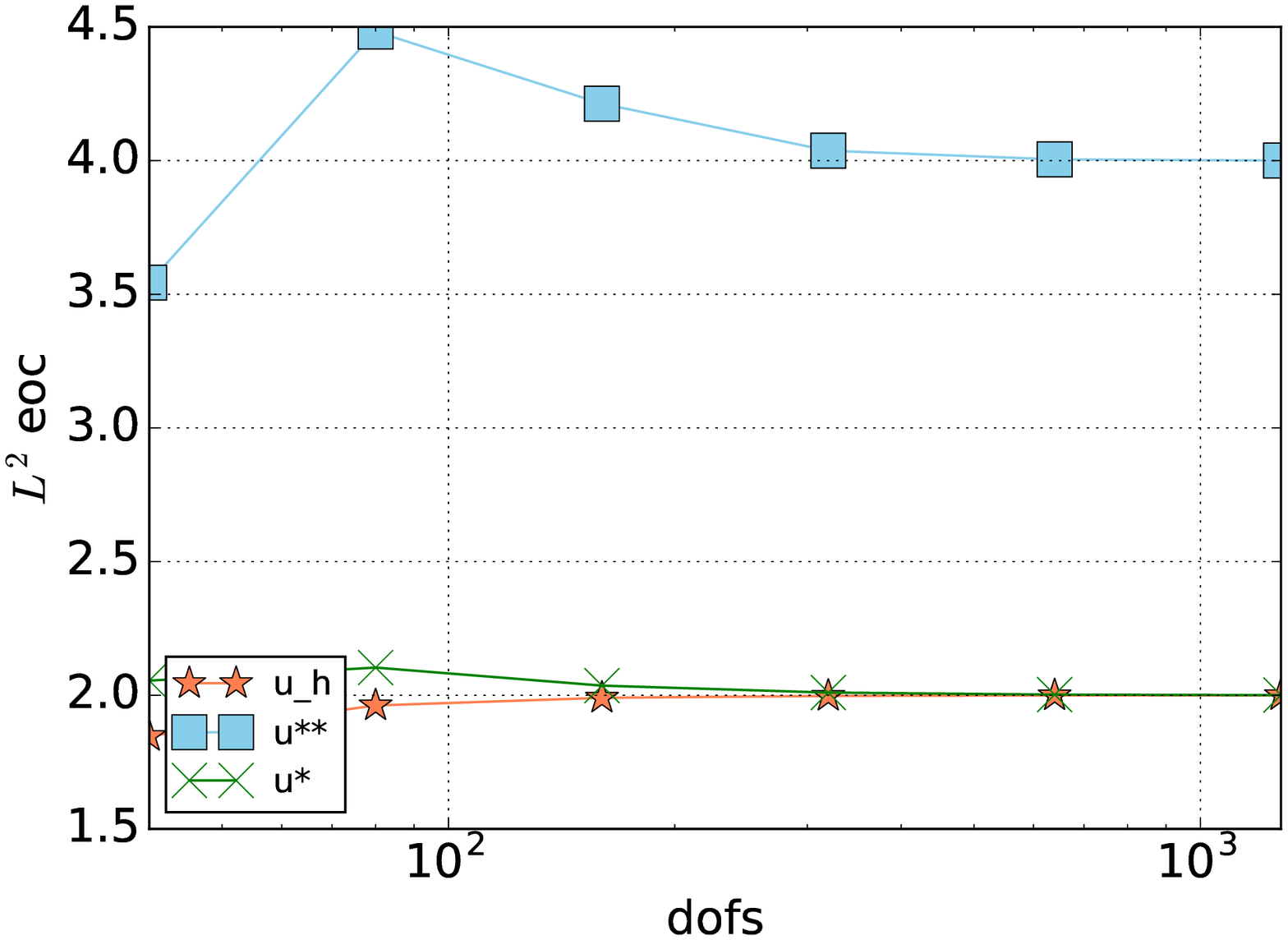}
\caption{Errors and convergence rates for $\sobh1$ (left two) and $\leb2$
  (right two) for polynomial degree $p=2$
  with hyperpenalty at the boundary.}
\label{fig:p1superweakerrors}
\label{fig:p1superweakeocs}
\end{figure}

We summarize our results for the smooth problem in Table~\ref{tab:eoc1d}.
It can be clearly seen that the step from $\uhs$ to $\uhss$, which requires
solving one additional low order problem, is quite advantageous and  increases the convergence
rate in the $\leb 2$ norm by at least one. In the linear case this improves by two
and by one in the $\sobh 1$ norm.  This makes it  highly efficient
in this case, at least when implementing the  hyperpenalization or strong constraints
to enforce the Dirichlet boundary conditions. The reason for this
restriction will be investigated further in future work.
For $p=3$ the actual EOCs of the postprocessed
solutions are difficult to determine and therefore we provide  approximate numbers. In
this case, SIAC shows a higher order in $\leb 2$ compared to the $\sobh 1$
norm. The Galerkin orthogonalty trick does not improve the rate
further, but note that the overall error is still a factor of $100$
smaller. Additionally note that in the other cases where there is no improvement in the rate,
the error is reduced by enforcing Galerkin orthogonality,
e.g., in the $\sobh 1$ norm with $p=2$ the error is still reduced by about a
factor of two. In addition, the orthogonality of $\uhss$ allows us to compute
a reliable and efficient error estimator with a comparable efficiency index
to the error estimator for $u_h$ using $R_h$. 

\begin{table}[htb]
  \centering
  \caption{Experimental rates of convergence for the smooth problem using
  different values for the polynomial degree $p$. The convergence rates are shown for
  the three approximations, i.e., $u_h,\uhs,\uhss$.}
  \label{tab:eoc1d}
  \resizebox{\columnwidth}{!}{%
  \begin{tabular}{cccccccccc}
    \toprule
    &
    \multicolumn{3}{c}{$p=1$}
    &
    \multicolumn{3}{c}{$p=2$}
    &
    \multicolumn{3}{c}{$p=3$}
    \\
    &
    $\EOC(u_h)$ & $\EOC(\uhs)$ & $\EOC(\uhss)$
    &
    $\EOC(u_h)$ & $\EOC(\uhs)$ & $\EOC(\uhss)$
    &
    $\EOC(u_h)$ & $\EOC(\uhs)$ & $\EOC(\uhss)$
    \\
    $\leb2$-error&
    2 & 2 & 4
    &
    3 & 4 & 5
    &
    4 & 6 & 6
    \\
    $\sobh1$-error&
    1 & 2 & 3
    &
    2 & 4 & 4
    &
    3 & 5 & 5
    \\
    \bottomrule
  \end{tabular}}
\end{table}

We conclude our investigations of the SIAC reconstruction and the residual
estimates by studying problems with less smooth solutions. We change the 
forcing function so that the exact solution is of the form
\begin{align*}
  u(x) &= \begin{cases}
    w\left(\frac{x-0.3}{0.4}\right) & x\in (0.3,0.7)~, \\
    0 & \text{otherwise}
  \end{cases}
\end{align*}
where
\begin{equation}
w(s)=\sin{6\pi s}^2\cos{\frac{9}{2}\pi s}
\end{equation}
 is the smooth function from previous studies.  We show results for
polynomial degree $p=2$.  We again iplement a simple $O(h^{-1})$
penalty term at the boundary. Note that the solution is in
$\cont{2}\setminus \cont 3$ at $x=0.7$ and only in $\cont
1\setminus \cont 2$ for $x=0.3$. {Overall the solution is an element of
$\sobh2(0,1)$ but not of $\sobh3(0,1)$, i.e., it is not smooth enough
to achieve optimal convergence rates for $p=2$ when the mesh is not
aligned. Even when the mesh is aligned, as in our experiments, we do
not expect an increase of the convergence rate using the SIAC
reconstruction as can be seen in Figure~\ref{fig:p2q2errors}.}  The
local loss of regularity at $x=0.3$ and $x=0.7$ is clearly visible for
the pointwise errors of the two reconstructions as shown in
Figure \ref{fig:p2q2pointwise}. Examining the errors in the original
approximation, $u_h$, the reduced smoothness is hardly visible.
However, in both of the reconstructions a jump in the error is clearly
visible. At $x=0.7$, where the solution is still in $\cont 2$ the
error in $\uhss$ increases approximately by two orders while at
$x=0.3$ it is close to four orders of magnitude larger since the
solution is only $\cont 1$ at this point. The lack of smoothness is
also identified by the residual indicator $R^{**}$, the spatial
distribution of which is shown in Figure~\ref{fig:p2q2pointwiseres}
together with the distribution of $R_h$.  It is worthwhile to note
that the region of the 'reduced smoothness' is better isolated by
$R^{**}$ than by $R_h$.  Hence it would be easier for an adaptive
algorithm to separate these different smoothness regions which would
lead to more optimal meshes.  The picture clearly shows that $R_h$
does not 'see' the kink so that an adaptive algorithm would either
refine the whole non-constant region or nothing at all depending on
the tolerance.  In contrast, with $R^{**}$ (and the right algorithm)
refinement could be isolated to the kinks.

\begin{figure}
\centering
\includegraphics[width=0.45\textwidth]{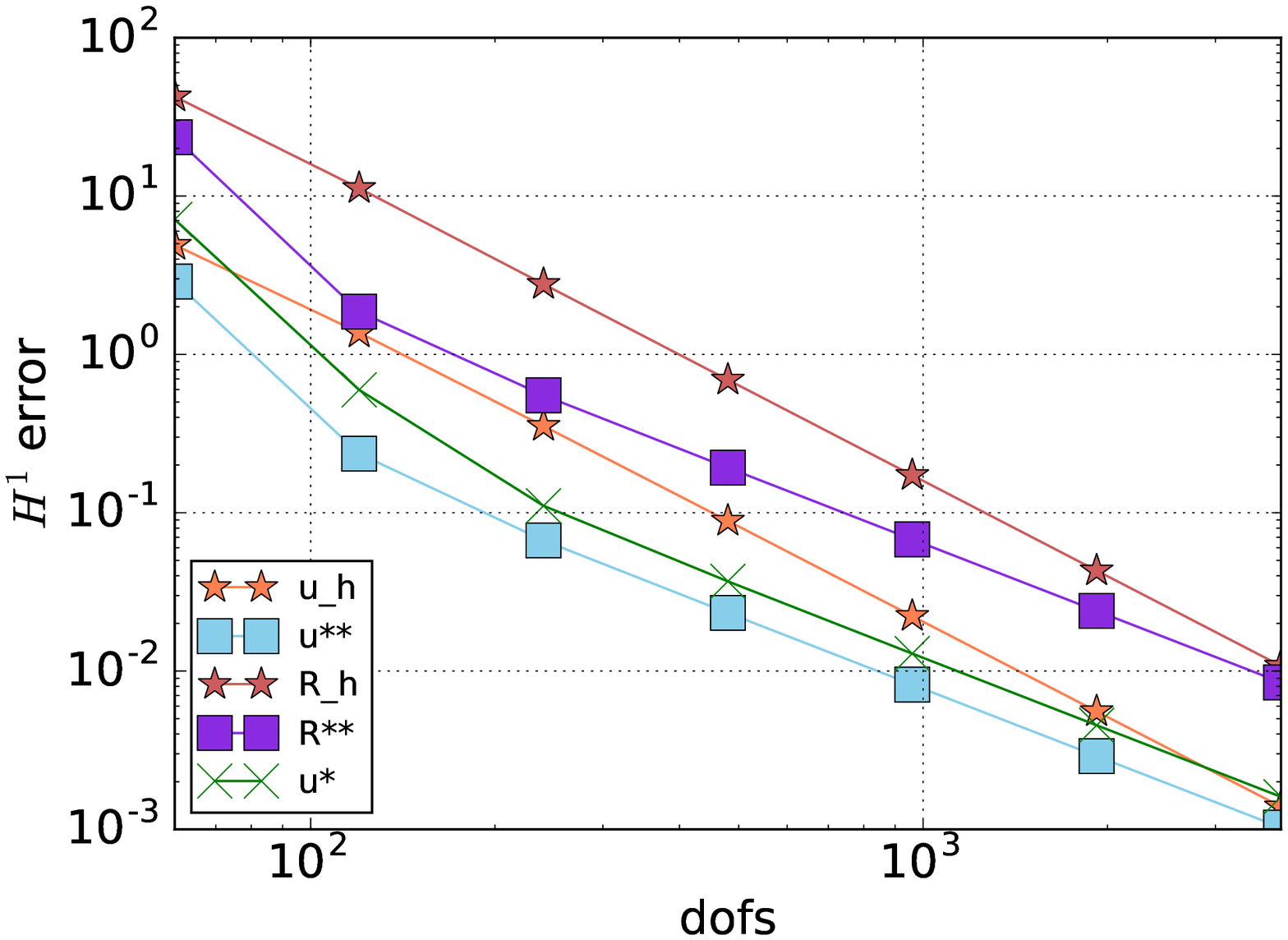}
\includegraphics[width=0.45\textwidth]{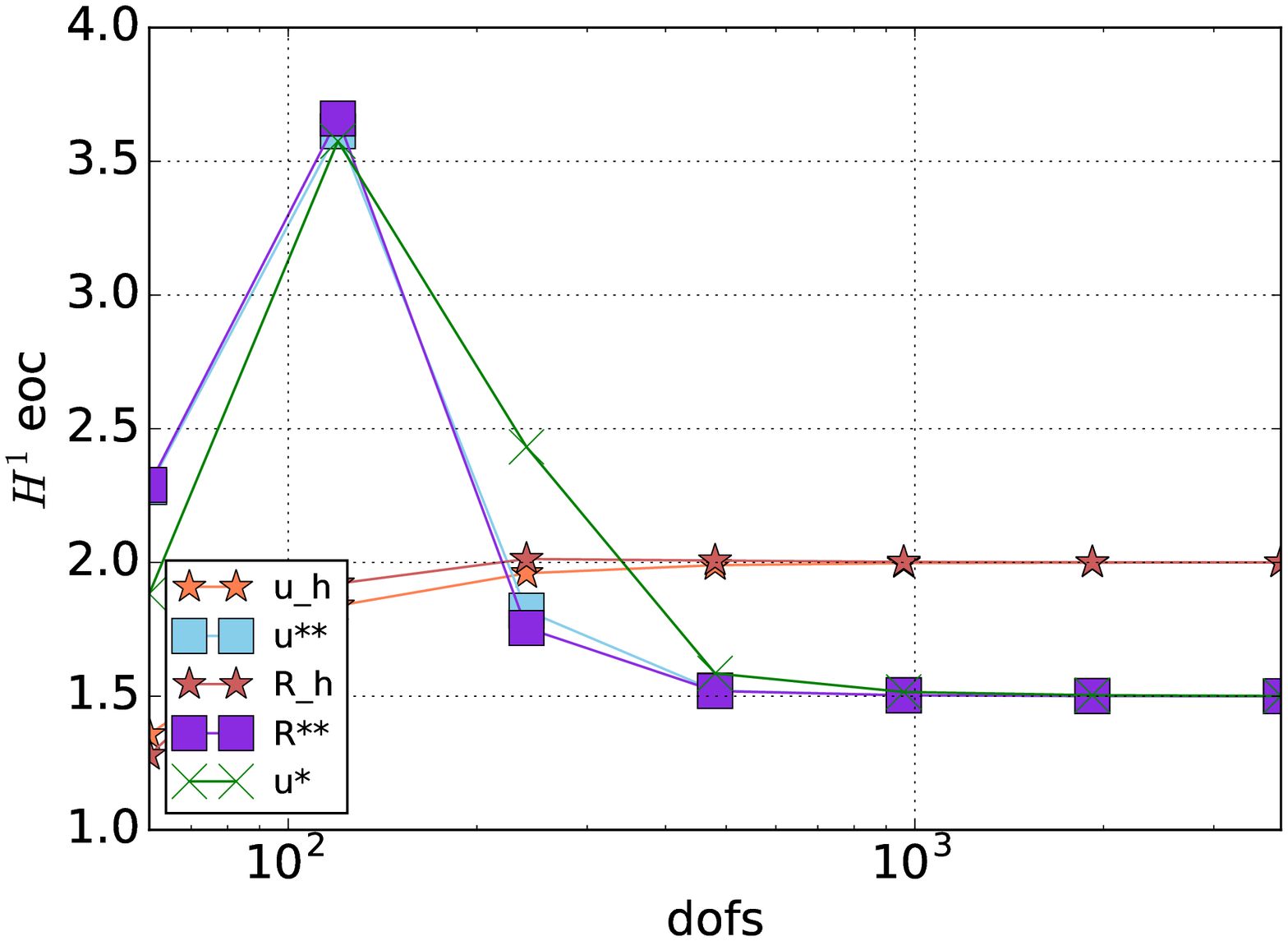}
\\
\includegraphics[width=0.45\textwidth]{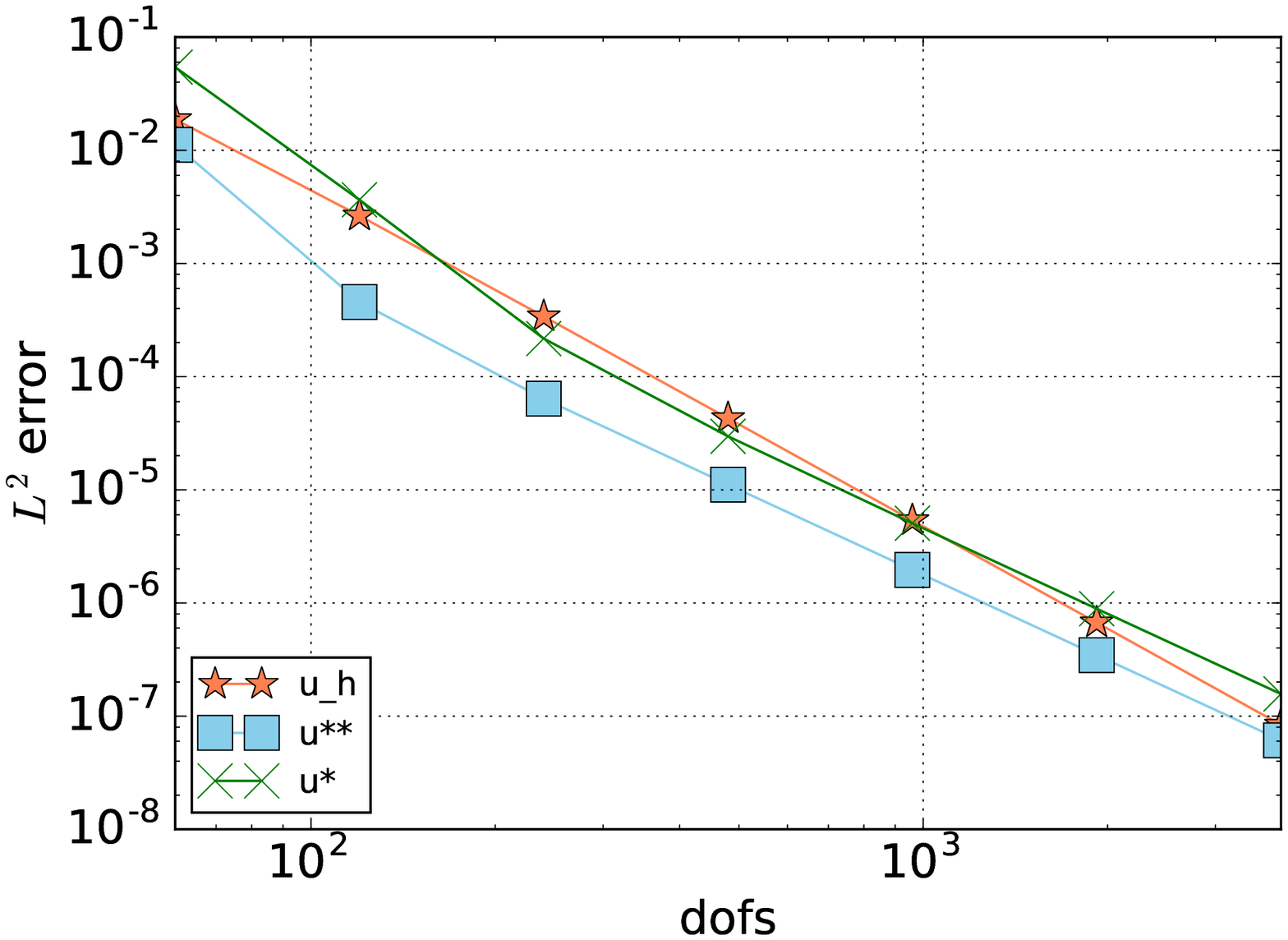}
\includegraphics[width=0.45\textwidth]{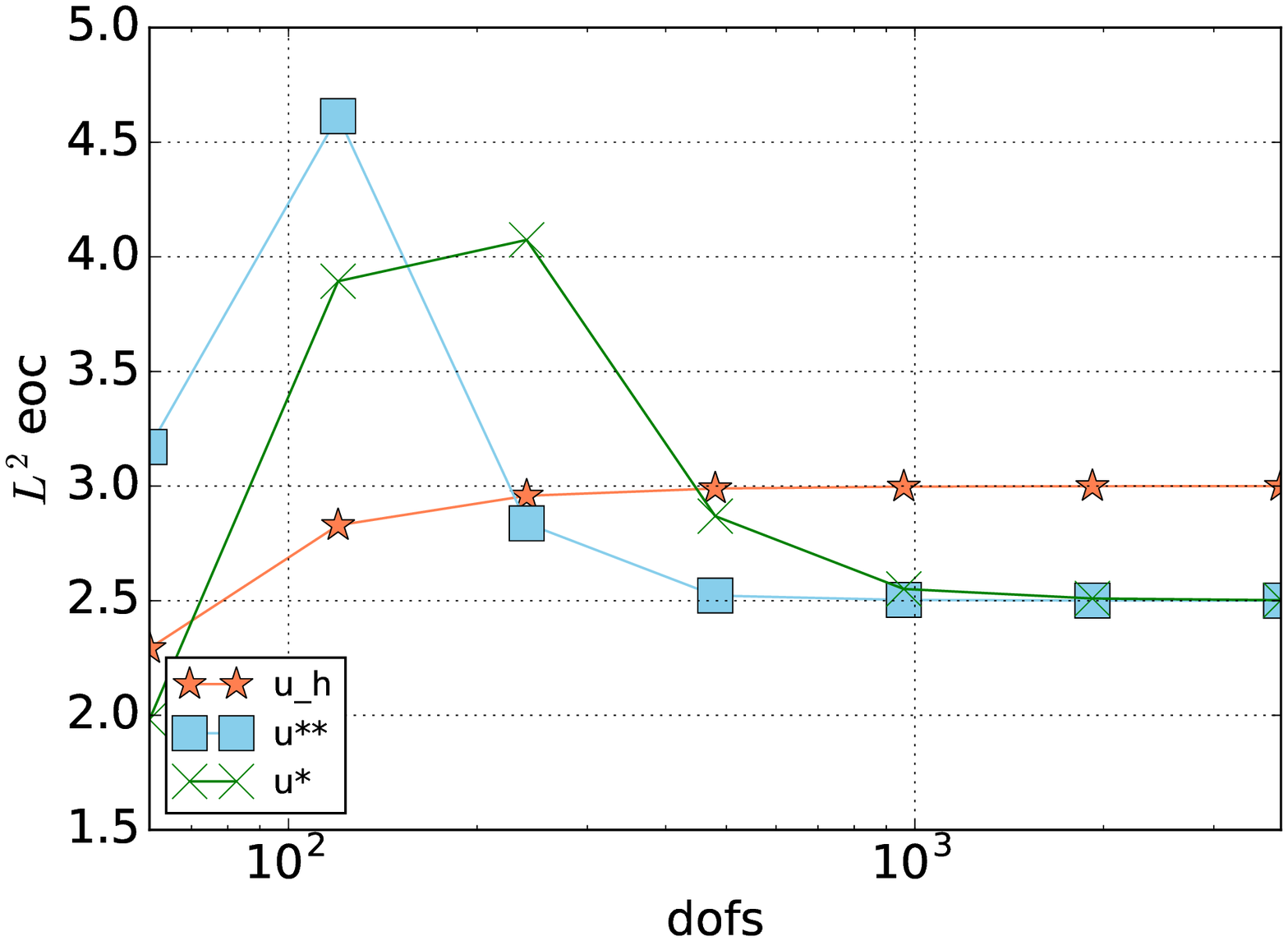}
\caption{Errors and convergence rates for $\sobh1$ (left two) and $\leb2$
  (right two) for polynomial degree $p=2$ using SIAC for solution with
  reduced smoothness}
\label{fig:p2q2errors}
\label{fig:p2q2eocs}
\end{figure}

\begin{figure}
\centering
  \includegraphics[width=0.45\textwidth]{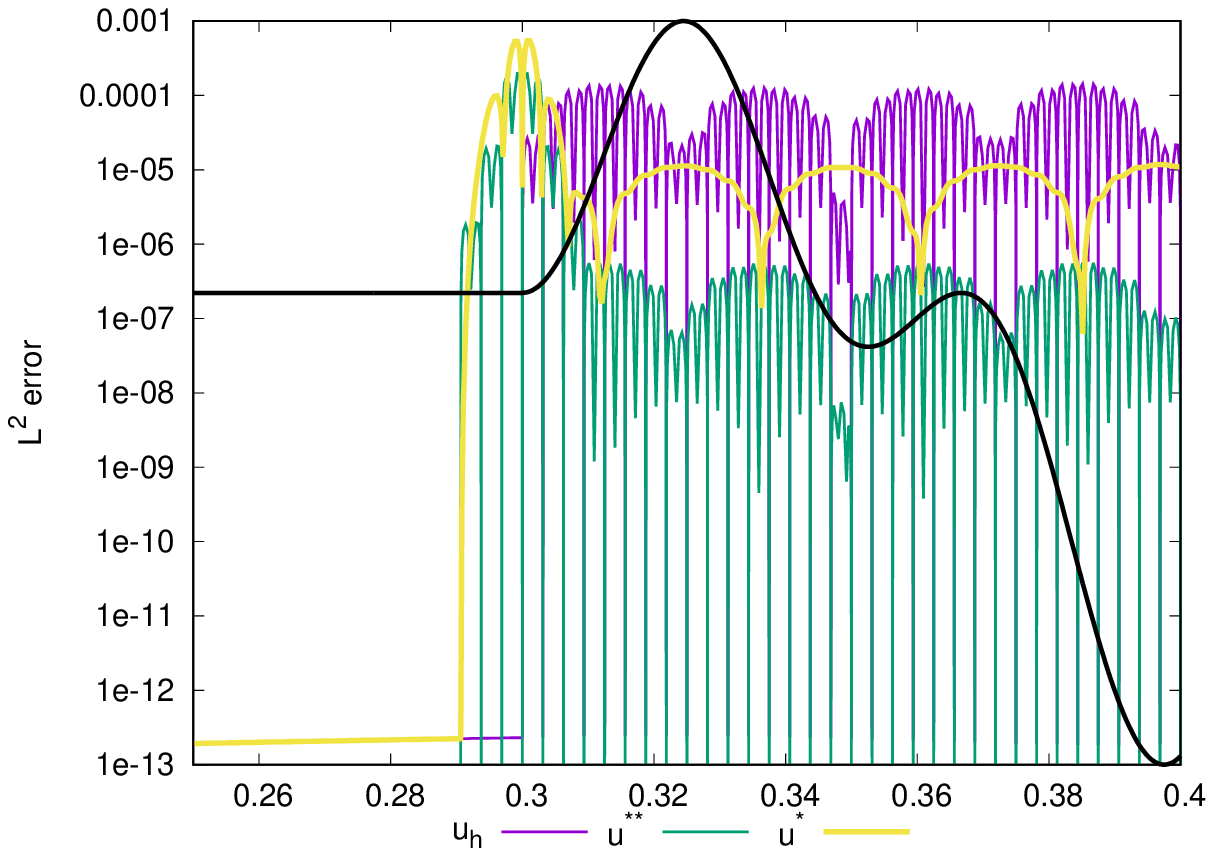}
  \includegraphics[width=0.45\textwidth]{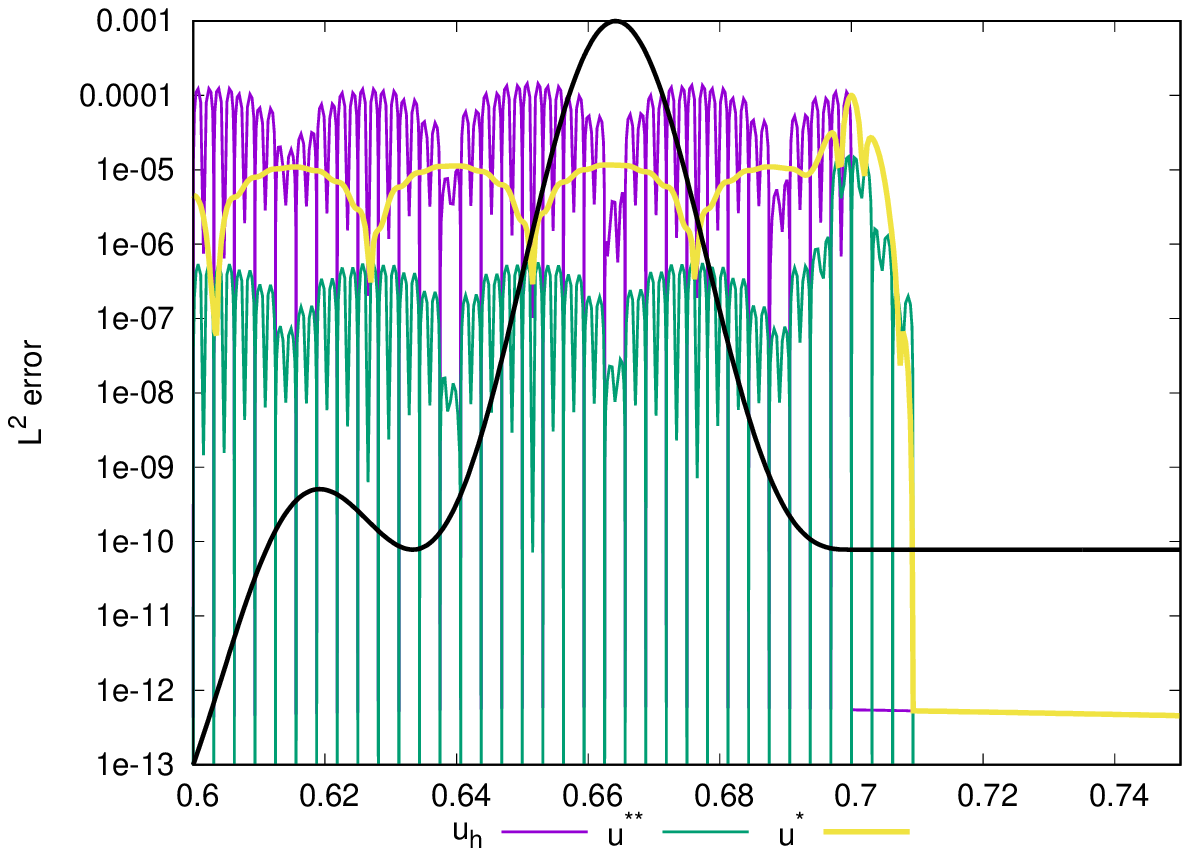}

\centering
  \includegraphics[width=0.45\textwidth]{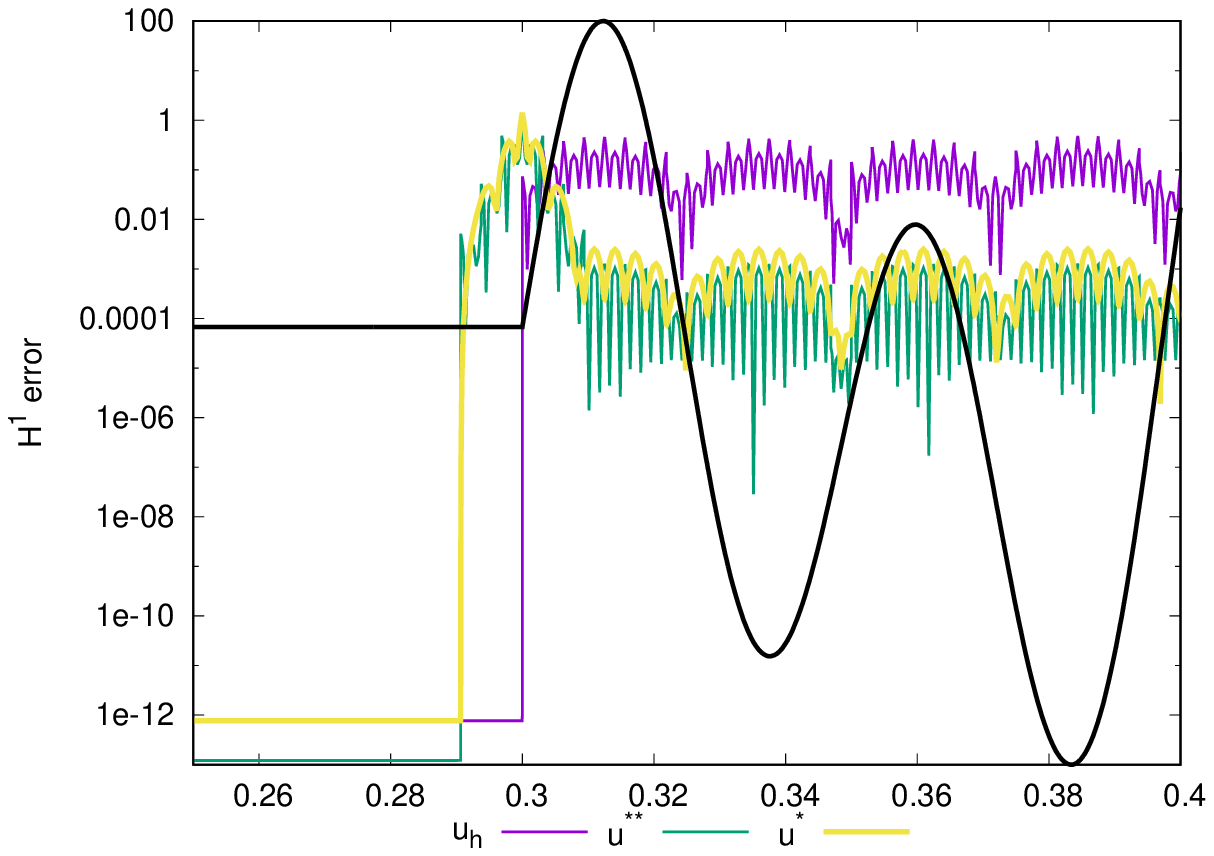}
  \includegraphics[width=0.45\textwidth]{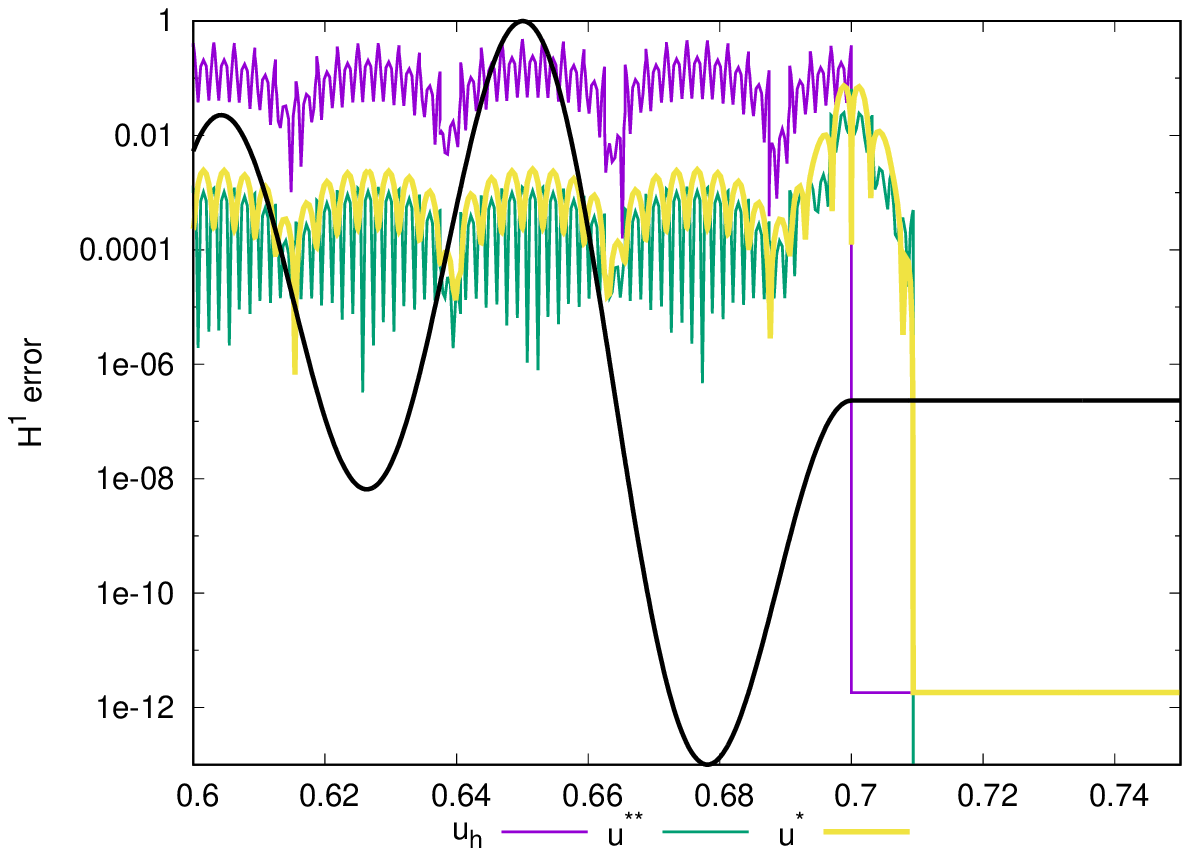}
\caption{Pointwise errors of the three different approximations 
  for the piecewise smooth problem. Top row shows the difference in the
  solution values around $x=0.3$ (left) and around $x=0.7$ (right). The
  bottom row shows gradient errors in the same two regions.
  These are results with $p=2$ and $h=1/320$.}
\label{fig:p2q2pointwise}
\end{figure}

\begin{figure}
  \includegraphics[width=0.45\textwidth]{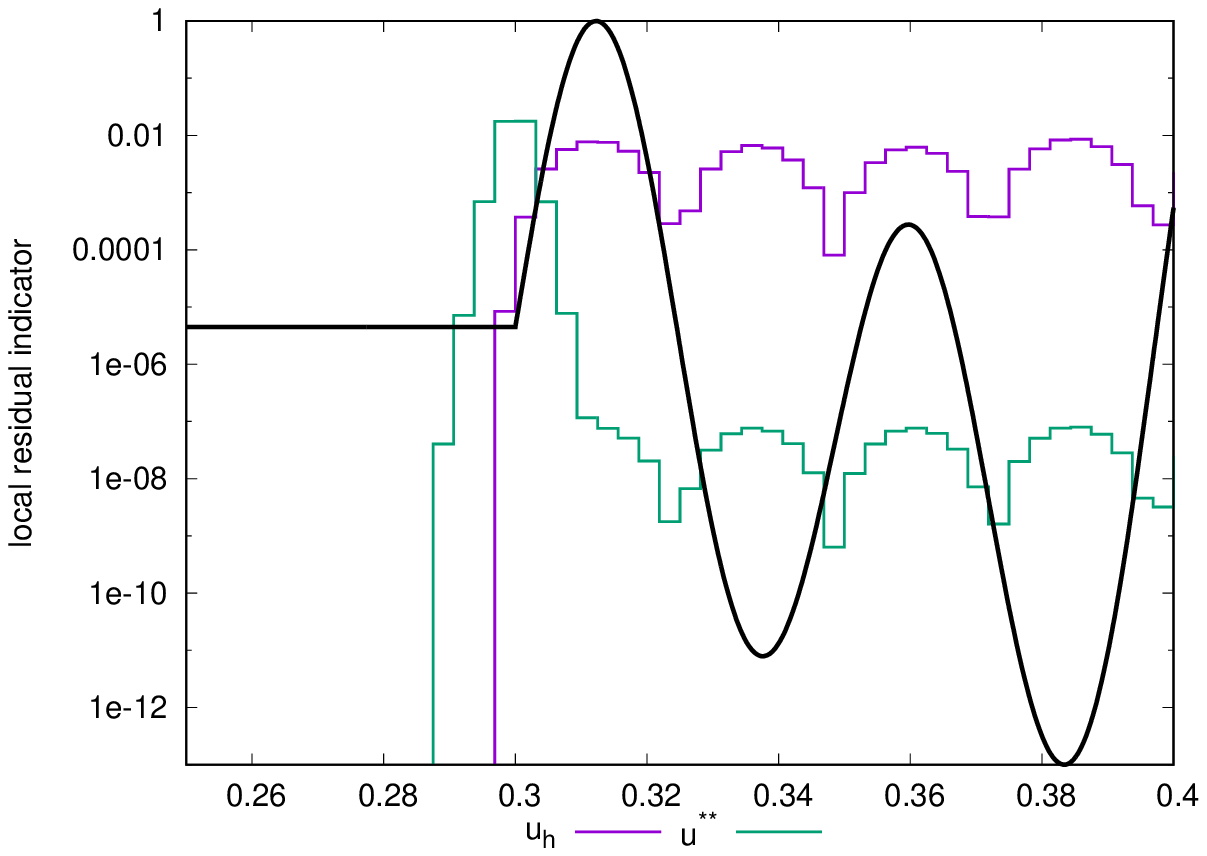}
  \includegraphics[width=0.45\textwidth]{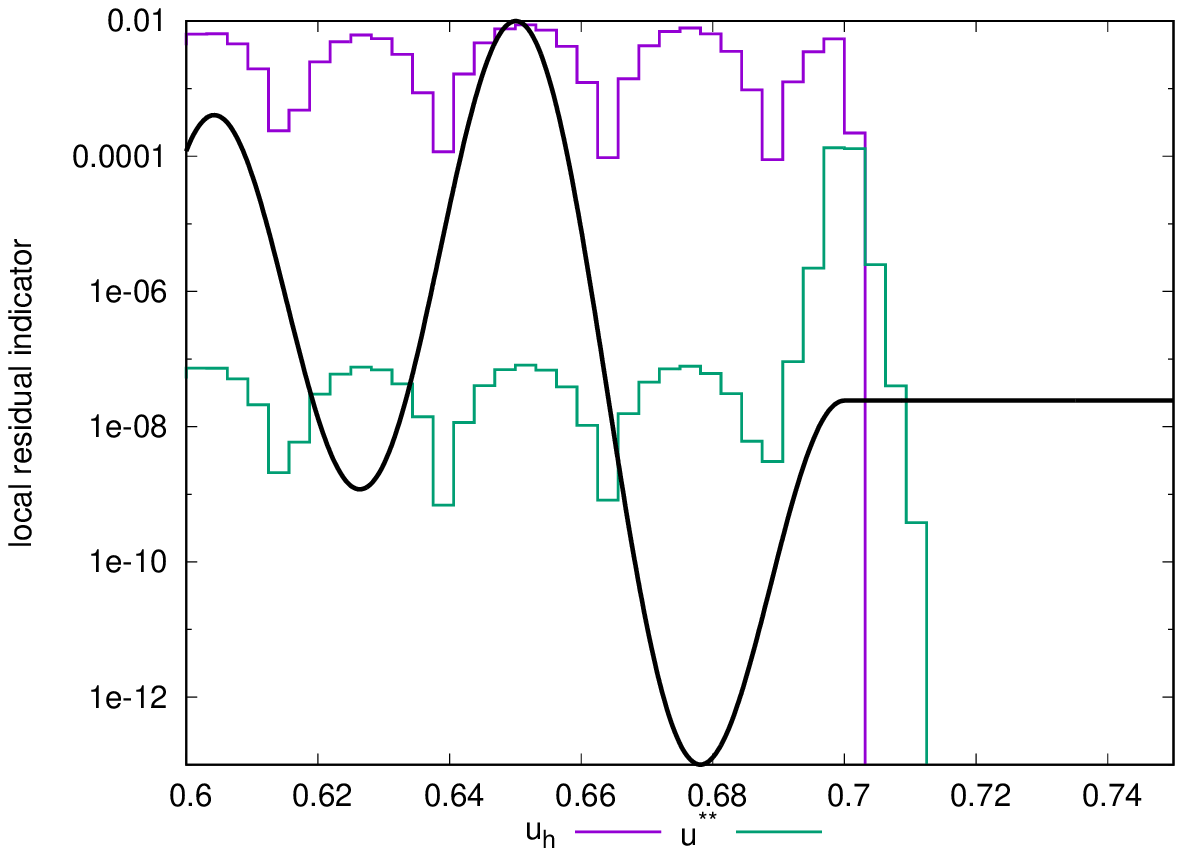}
  \caption{Elementwise residual indicators using $u_h$ and $\uhss$
  for the piecewise smooth problem around $x=0.3$ (left) and around $x=0.7$ (right).
  bottom row shows gradient errors in the same two regions.
  These are results with $p=2$ and $h=1/320$.}
\label{fig:p2q2pointwiseres}
\end{figure}

\subsection{Superconvergent Patch Recovery}
In the following we solve
\begin{equation}
-\div\qp{\D \nabla u} = f
\end{equation}
in a two dimensional domain $\Omega$ where the forcing function $f$ is chosen by
prescribing an exact solution $u$. This function is also used to prescribe
Dirichlet boundary conditions on all of $\partial\Omega$. In the first
example we chose a smooth exact solution $u$ with a scalar diffusion
coefficient $\D=I_2\qp{|x|^2+\frac{1}{2}}$, while for the second test we
use a solution with a corner singularity and $\D = I_2$.

{In the following we show results using a DG scheme on a triangular grid.
The grid is refined by splitting each element into four elements. In the final examples with local adaptivity, this leads to a grid with hanging nodes. We also carried out experiments using a continuous ansatz space with very simular results.} 

Note that in all figures depicting errors and EOCs, the $x$-axis shows the number of degrees of freedom for
$u_h$. While the other approximations have a larger number of degrees of
freedom, the global problem that has to be solved, i.e. solving the linear
system for $u_h$ and for $R\uhs$, scales with the number of degrees of
freedom for $u_h$ and thus this seems a reasonable indication of the
computational complexity.

For our first test we choose $u(x,y) = \sin{\pi x/(0.25+xy)}\sin{\pi
(x+y)}$, and $\Omega=(0,1)^2$. We start with an initial grid which is
slightly irregular as shown in Figure~\ref{fig:macrogrid2d}. This is
to avoid any superconvergence effects due to a structured layout of the
triangles.
\begin{figure}
\centering
  \includegraphics[width=0.47\textwidth]{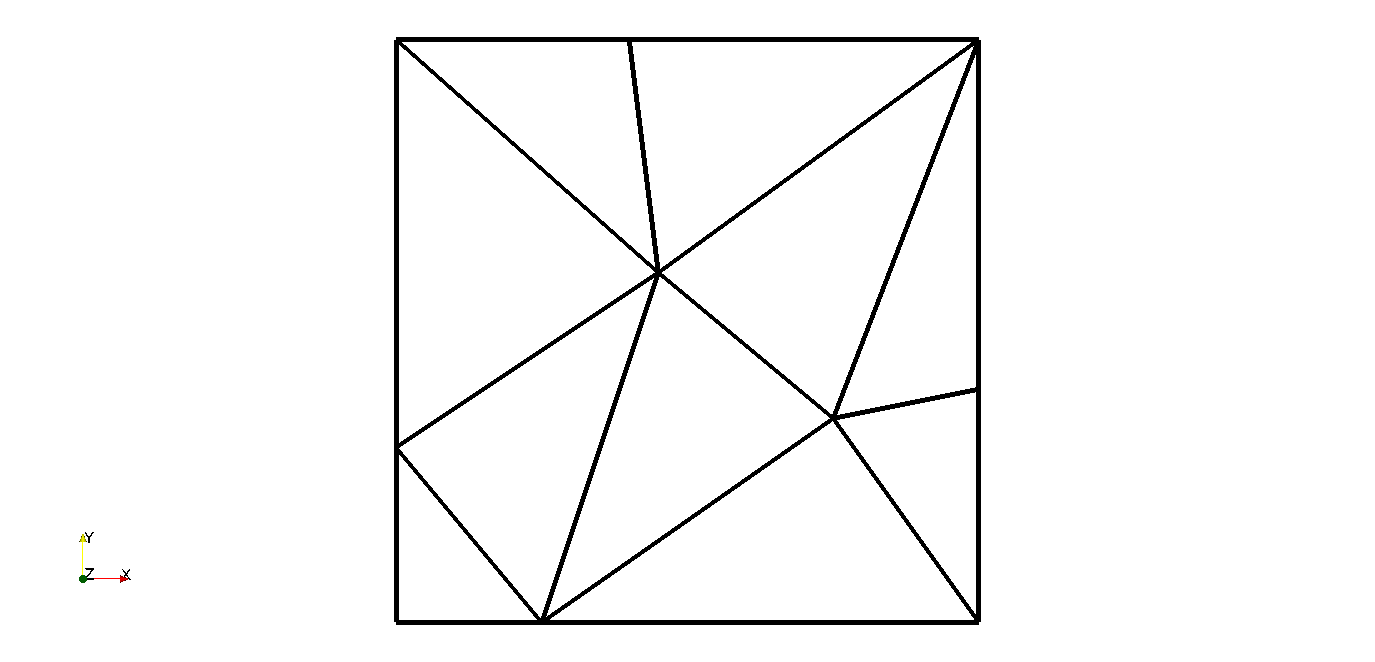}
  \includegraphics[width=0.47\textwidth]{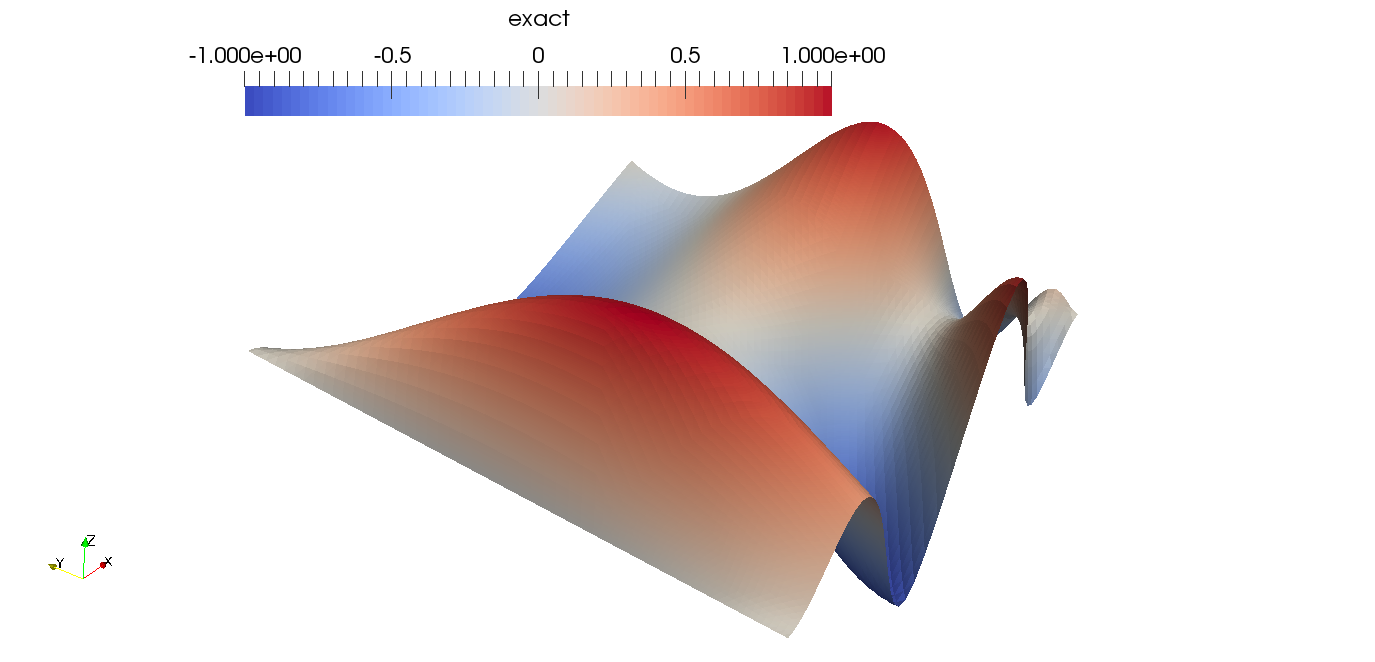}
  \caption{Macro grid and exact solution $u$ for smooth problem. Note that
  the solution has been scaled down by a factor of $4$.
    \label{fig:macrogrid2d}
}
\end{figure}

Figure~\ref{fig:global2dsmooth} shows $\leb2$ and $\sobh1$ errors and EOCs for the
three approximations $u_h,\uhs,\uhss$ with polynomial degrees $p=1,2,3$.
It can be seen that, in general, the postprocessor $\uhs$ improves the EOC
by an order of $1$ in the $\sobh1$ norm and that the EOC of the improved postprocessor
$\uhss$ is at least as good. While the actual error of $\uhs$ can be larger on
coarser grids  than the error computed with $u_h$, the error using
$\uhss$ is significantly better in all cases. 
{Focusing now on the $\leb2$ norm, we see that when computing the error using $\uhss$, the EOC is one order better then the convergence rate in the $\sobh1$ norm, as expected. 
For $p=2,3$, this is also true when using $\uhs$, while for
$p=1$ the EOC is only $2$ in this case, and an increase to $3$ is only achieved with the improved postprocessor $\uhss$.}
The same observation can be made when using the SIAC postprocessor in the previous section.
\begin{figure}
\centering
  \includegraphics[width=0.32\textwidth]{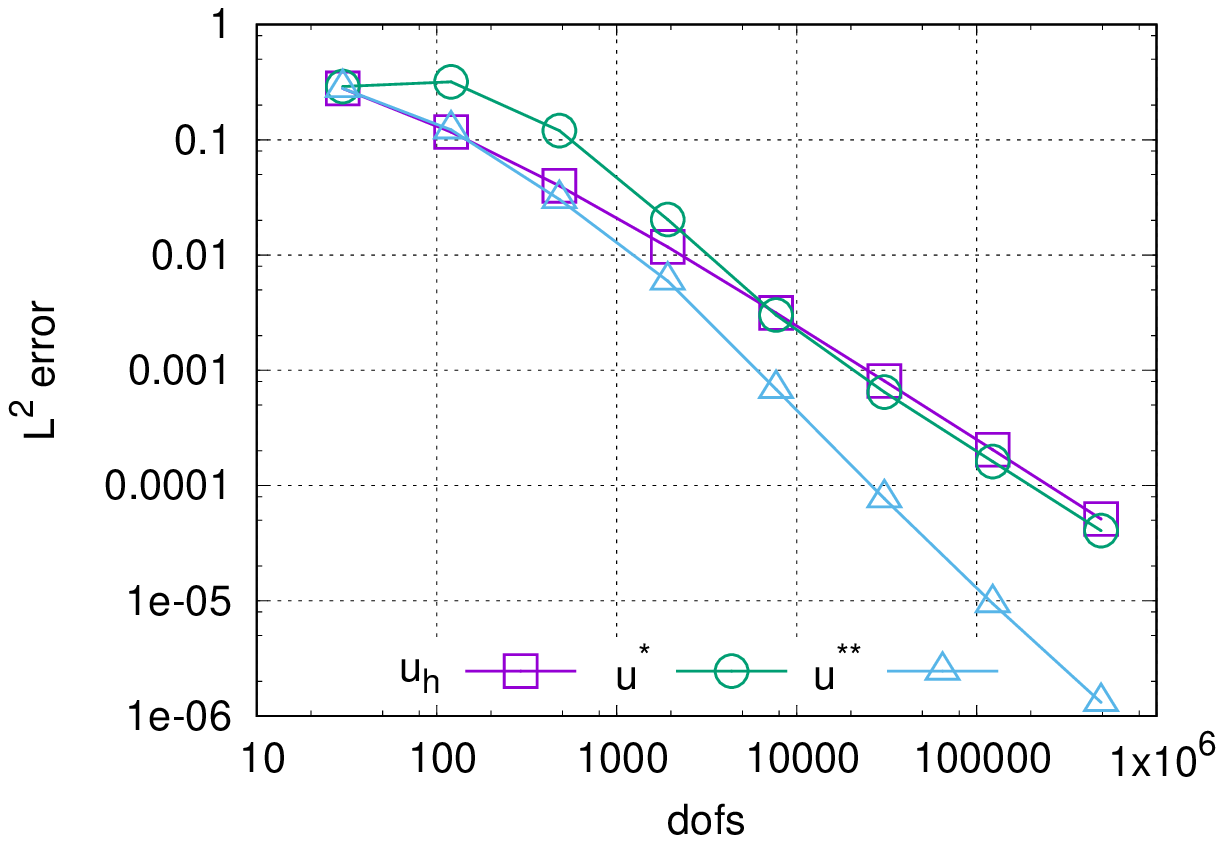}
  \includegraphics[width=0.32\textwidth]{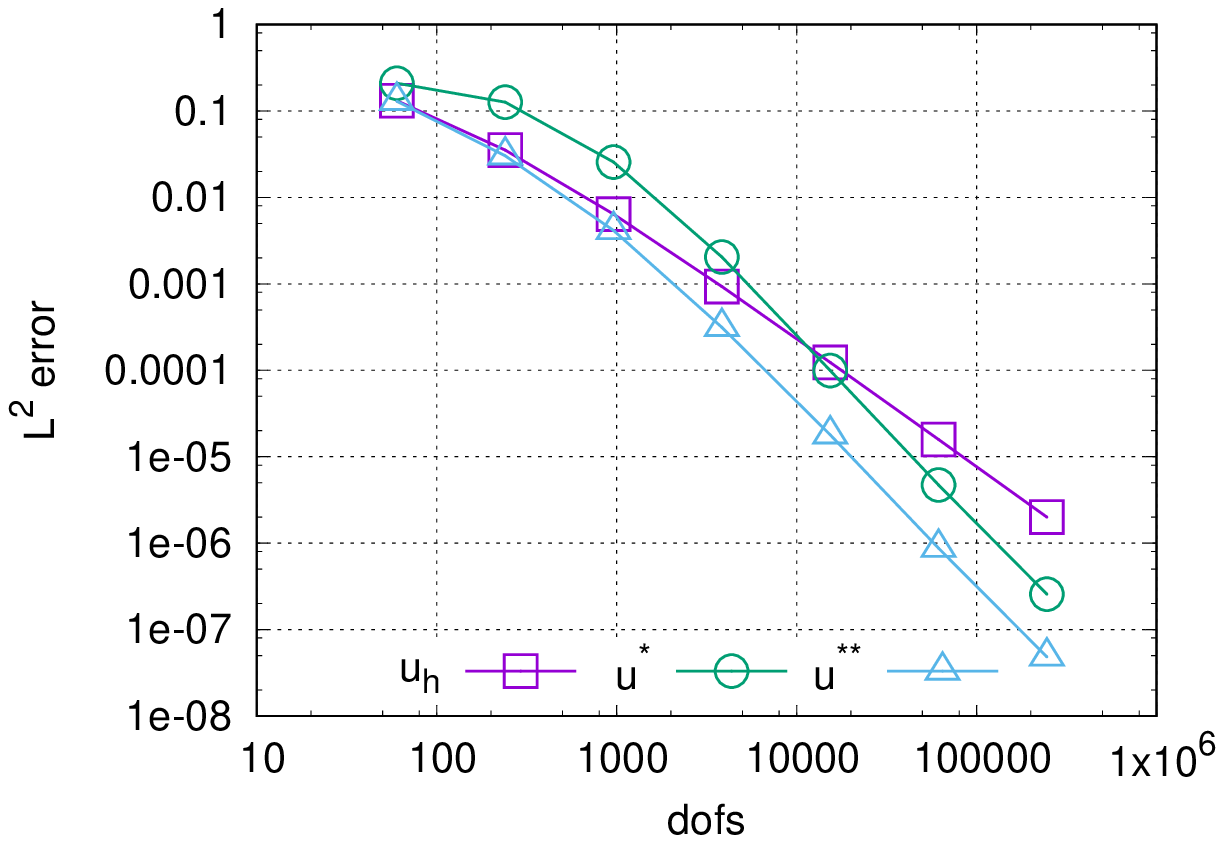}
  \includegraphics[width=0.32\textwidth]{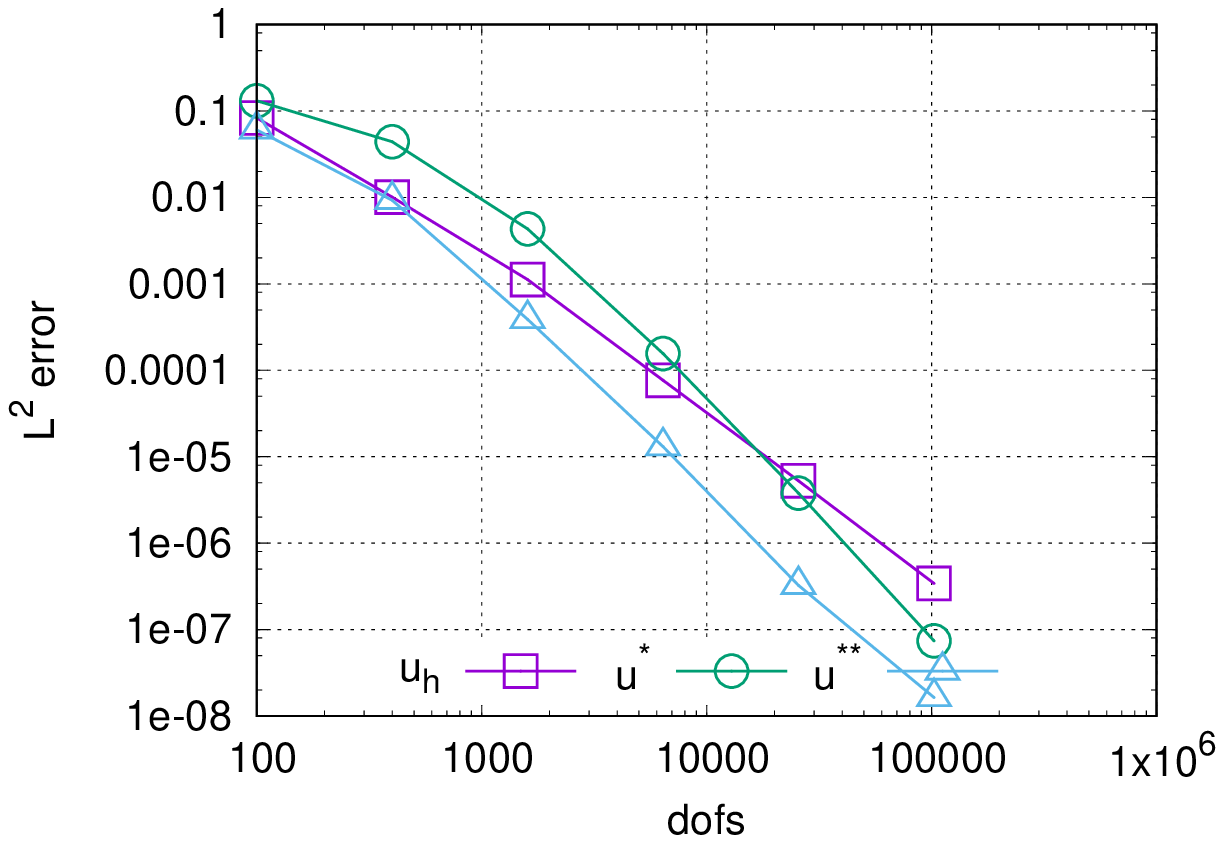} \\
  \includegraphics[width=0.32\textwidth]{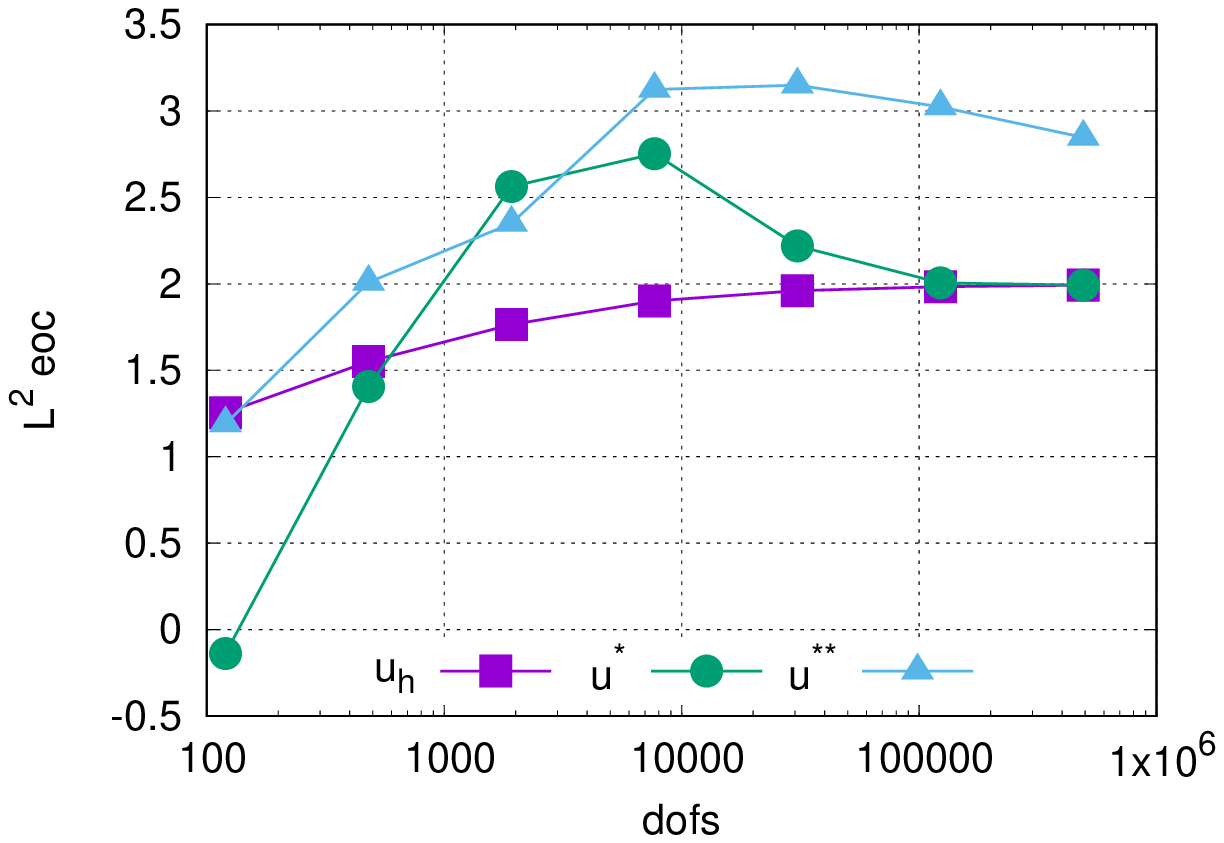}
  \includegraphics[width=0.32\textwidth]{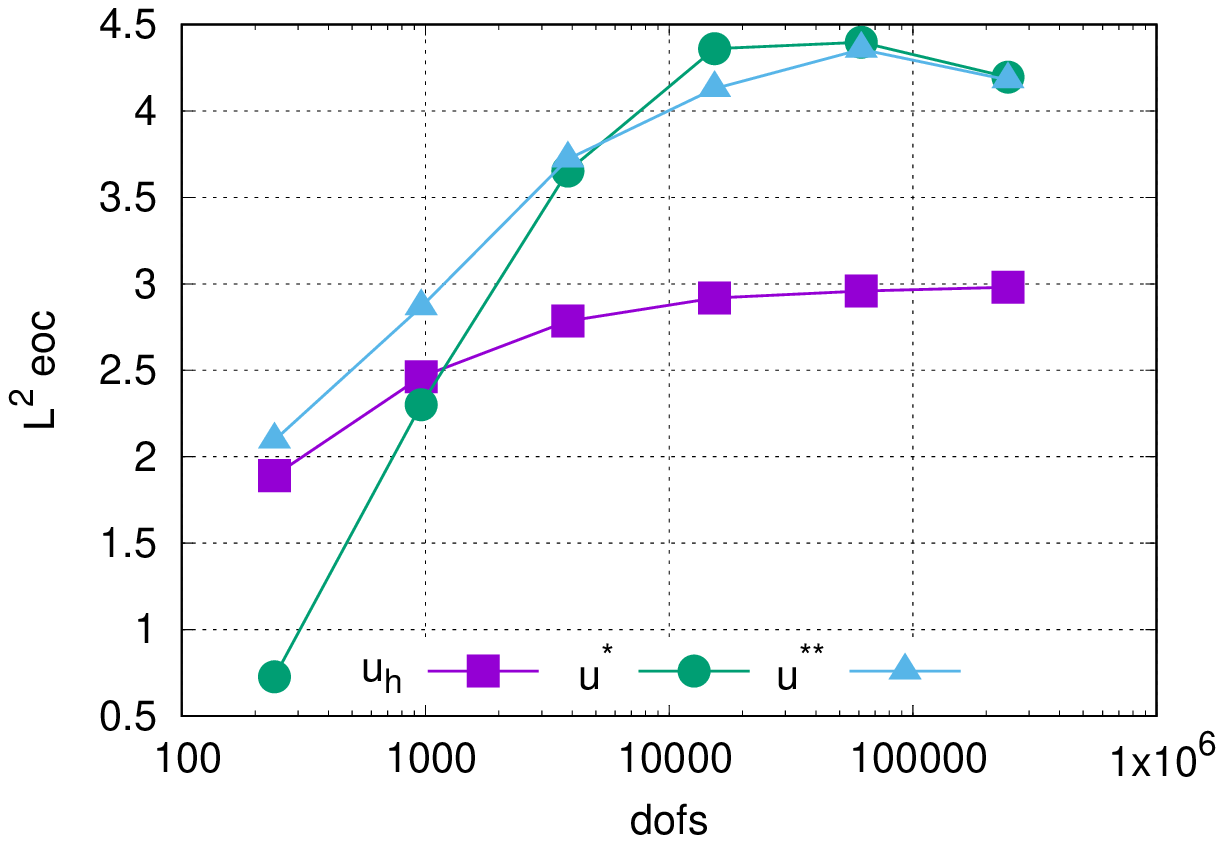}
  \includegraphics[width=0.32\textwidth]{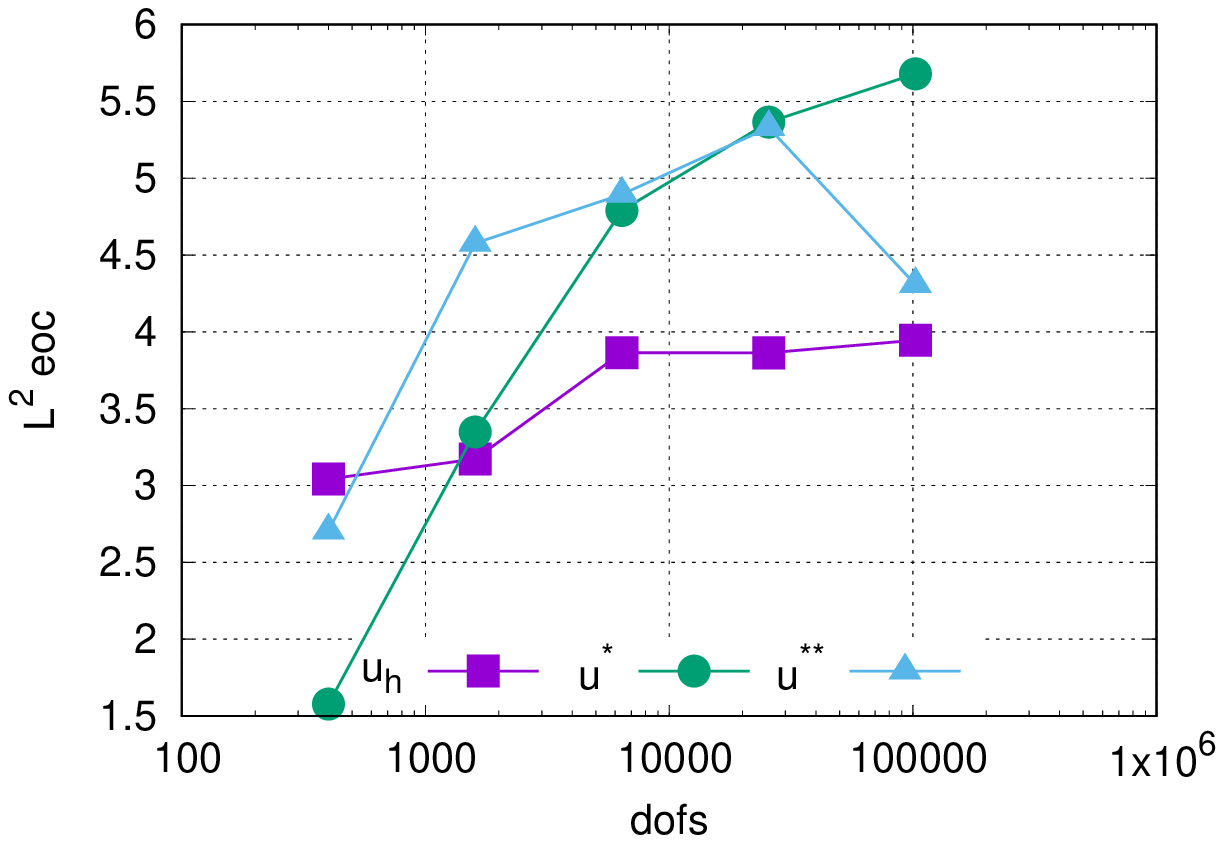} \\
  \includegraphics[width=0.32\textwidth]{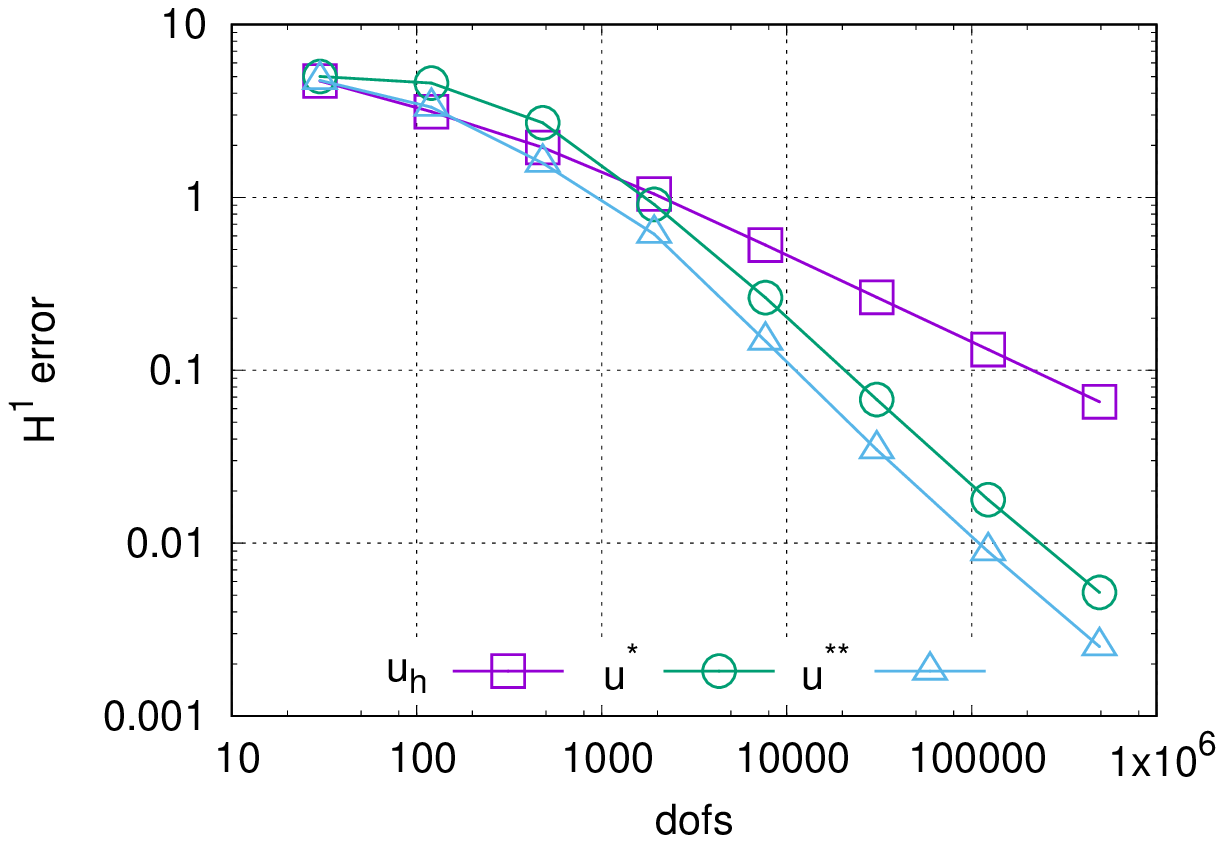}
  \includegraphics[width=0.32\textwidth]{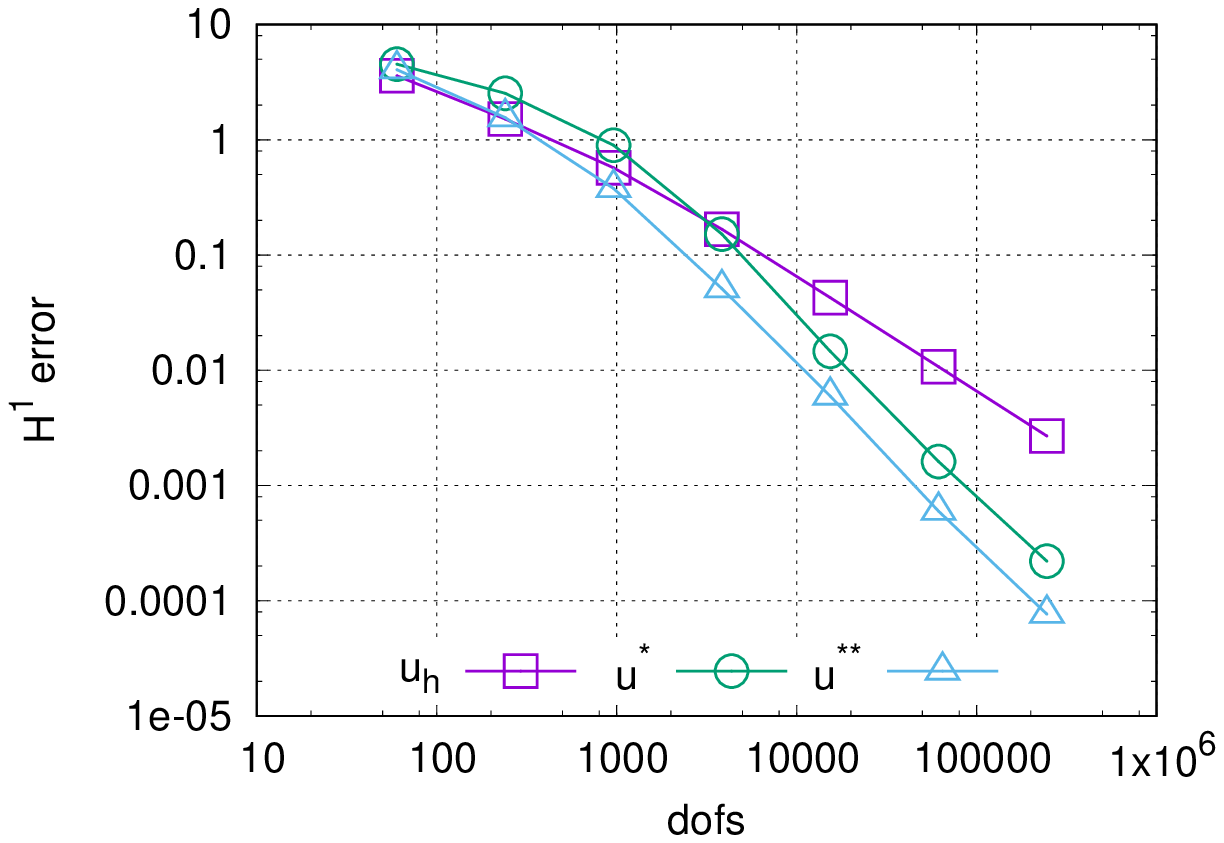}
  \includegraphics[width=0.32\textwidth]{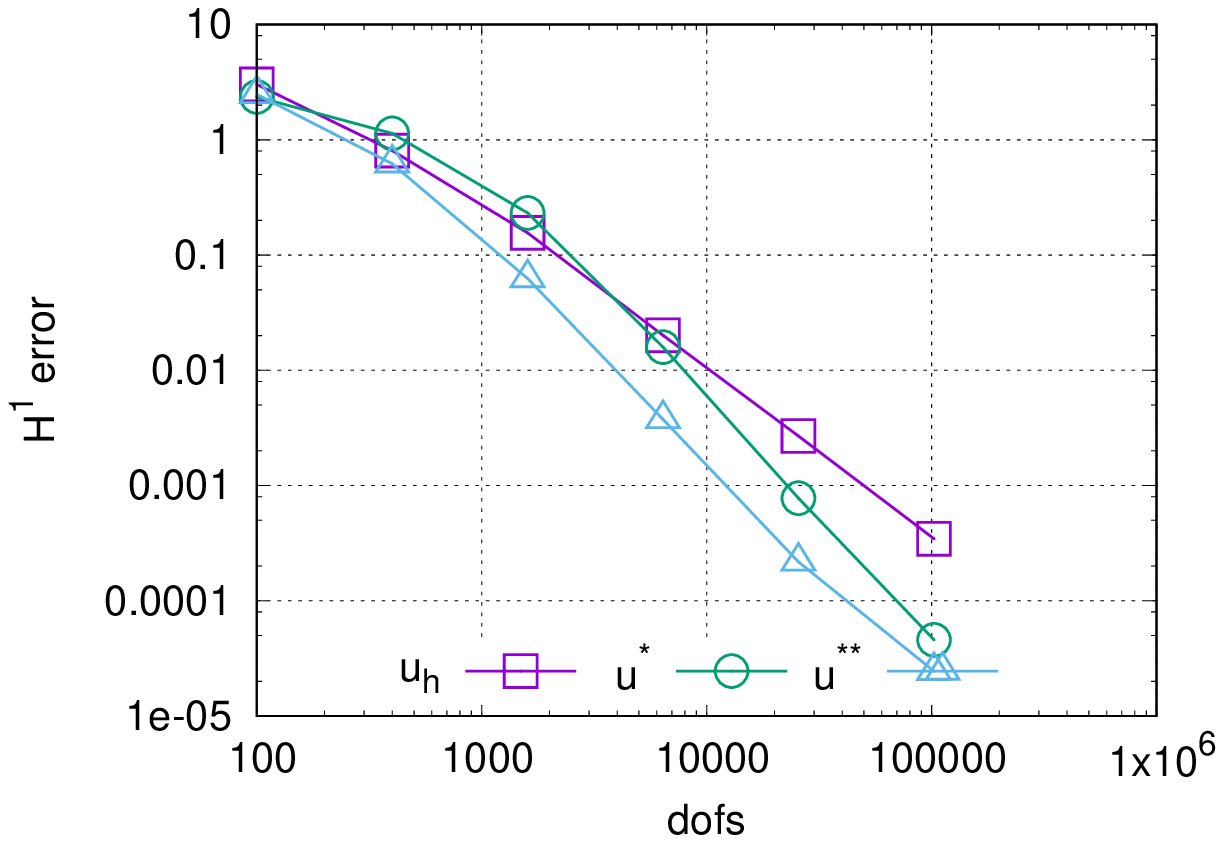} \\
  \includegraphics[width=0.32\textwidth]{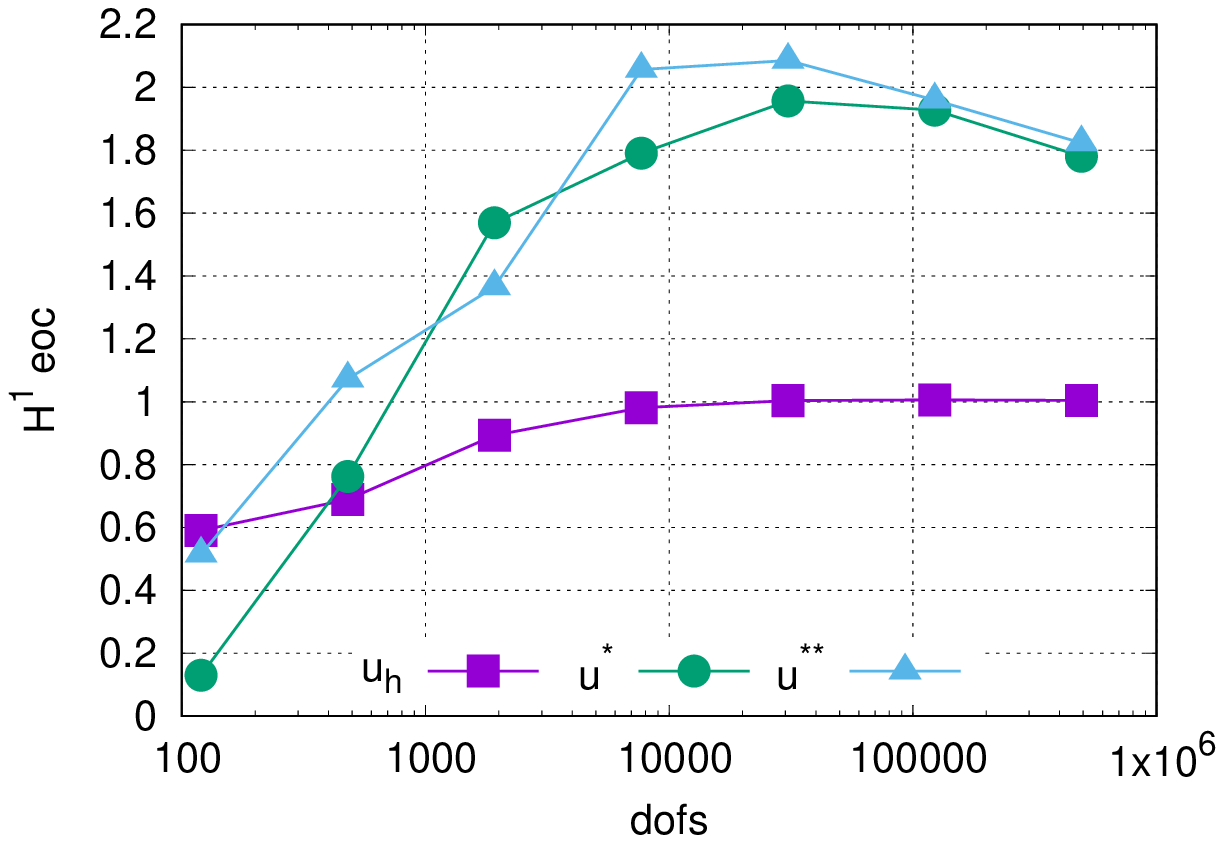}
  \includegraphics[width=0.32\textwidth]{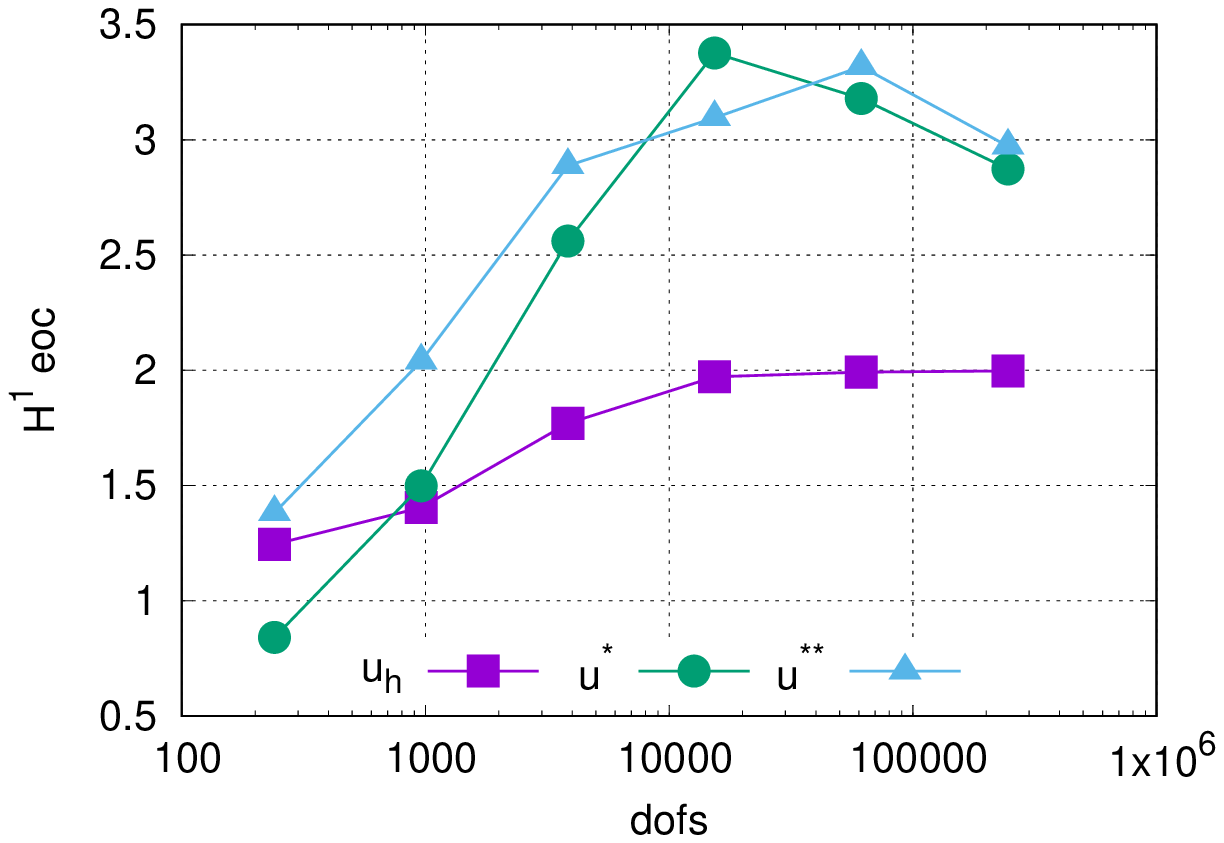}
  \includegraphics[width=0.32\textwidth]{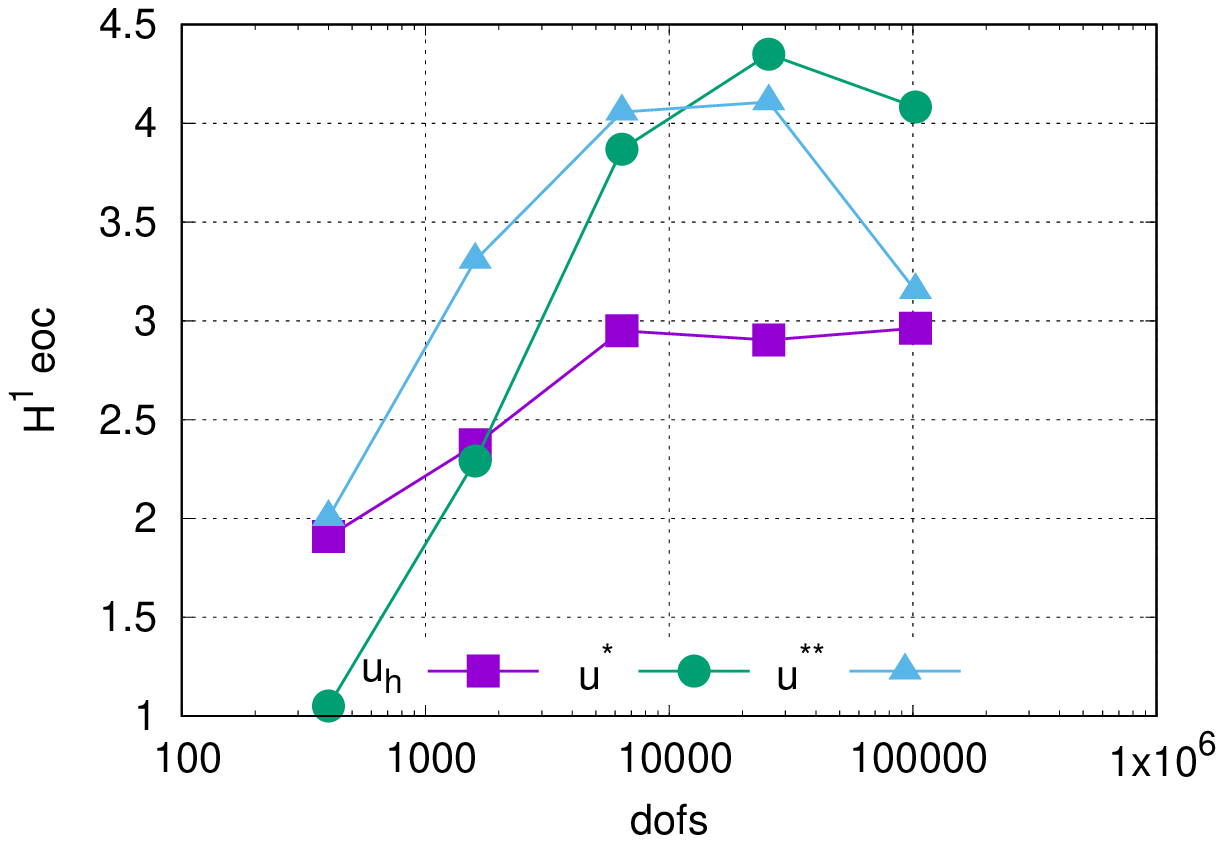}
  \caption{$L^2$ errors and EOC (top two rows)
  and $H^1$ errors and EOC (bottom two rows) for the smooth problem with $p=1,2,3$ (left to right)
  \label{fig:global2dsmooth}.}
\end{figure}

Using the same problem setting, we investigate the performance of an
adaptive algorithm in Figure~\ref{fig:adapt2dsmooth}. We use
a \emph{modified equal distribution} strategy where elements are
marked for refinement when the local indicator $\eta_K$ exceeds
$\frac{\sum\eta_K}{\text{\#elements}}$. We compute the local indicator
on either $u_h$ or on the improved reconstruction $\uhss$.  The
advantage of basing the marking strategy on $\uhss$ is clearly
demonstrated. While marking with respect to $u_h$ and then using the
postprocessor only on the final solution (filled upward triangles)
leads to a significant reduction of the final error, the difference in
the convergence rate between $R_h$ and $R^{**}$ results in a finer grid
than necessary for a given tolerance. A reduction in the number of degrees
of freedom by a factor of $10$ to $100$ can be easily achieved by
using $R^{**}$.

\begin{figure}
\centering
  \includegraphics[width=0.32\textwidth]{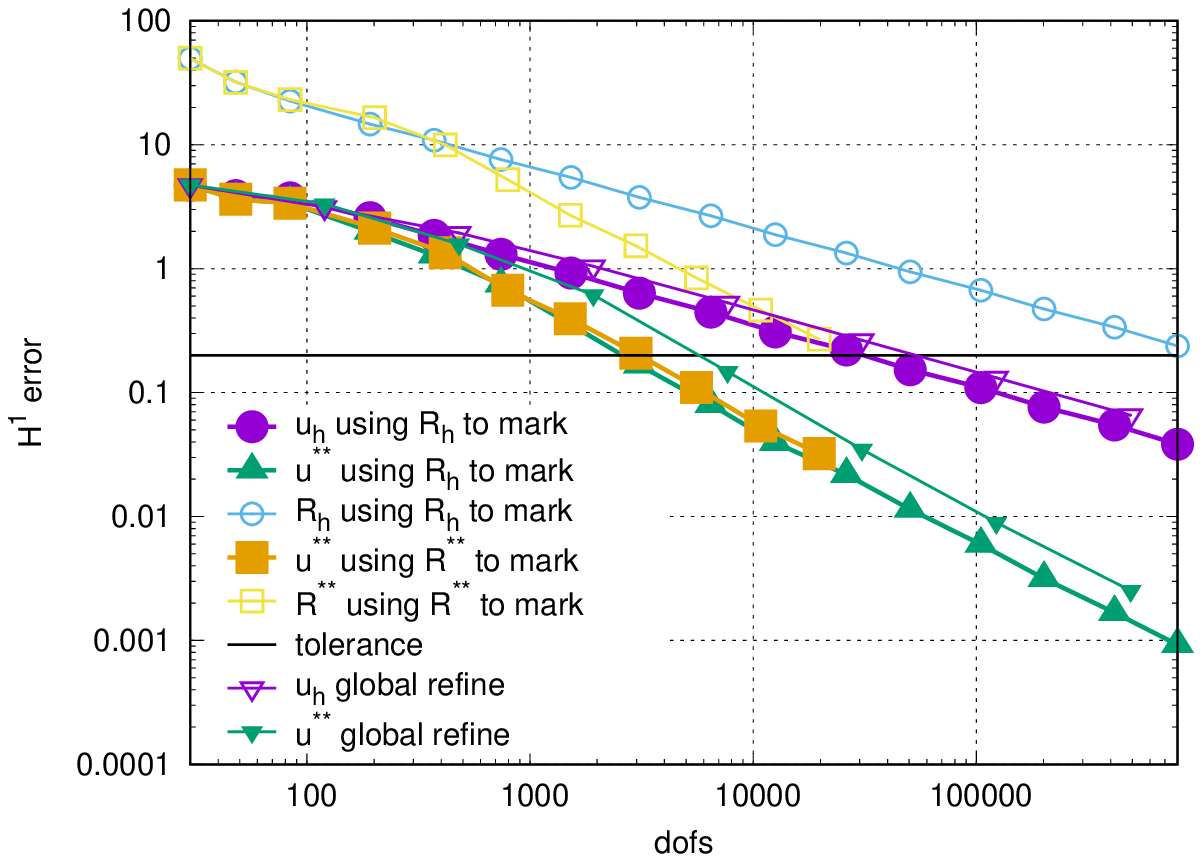}
  \includegraphics[width=0.32\textwidth]{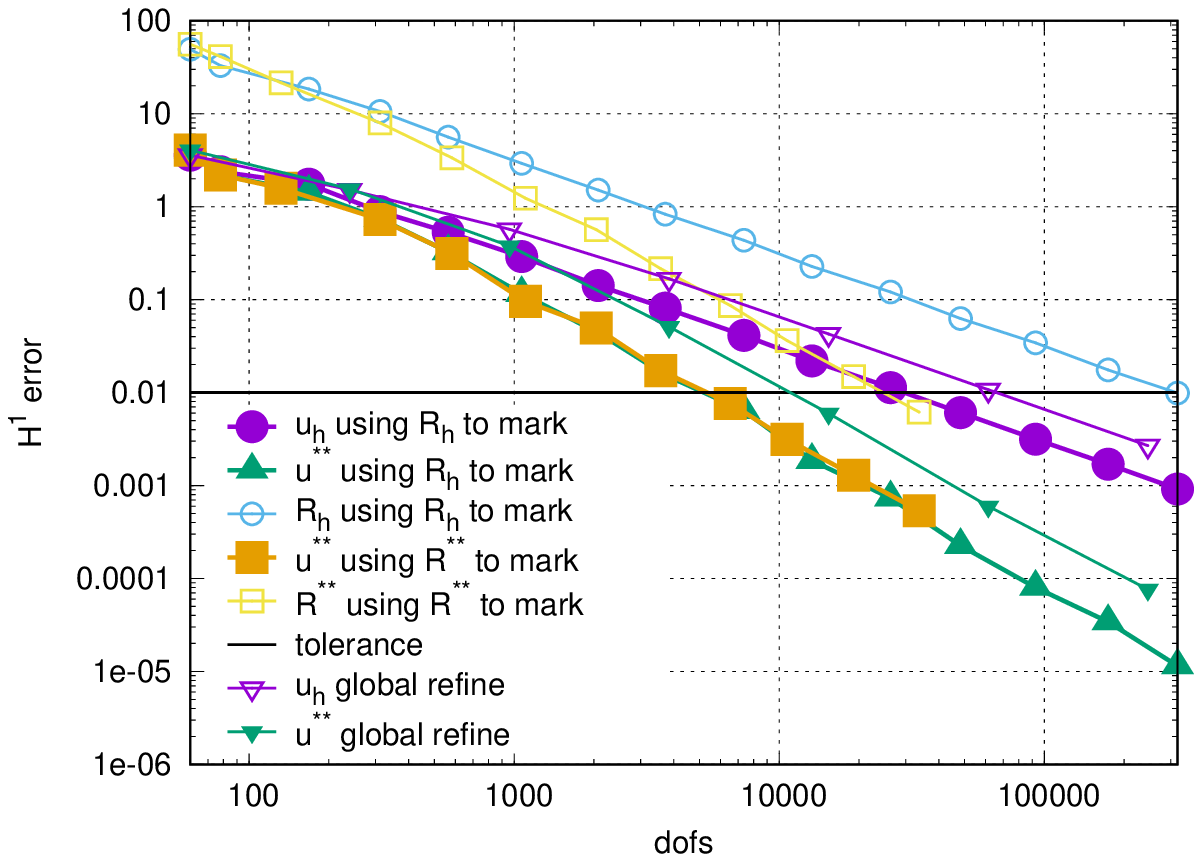}
  \includegraphics[width=0.32\textwidth]{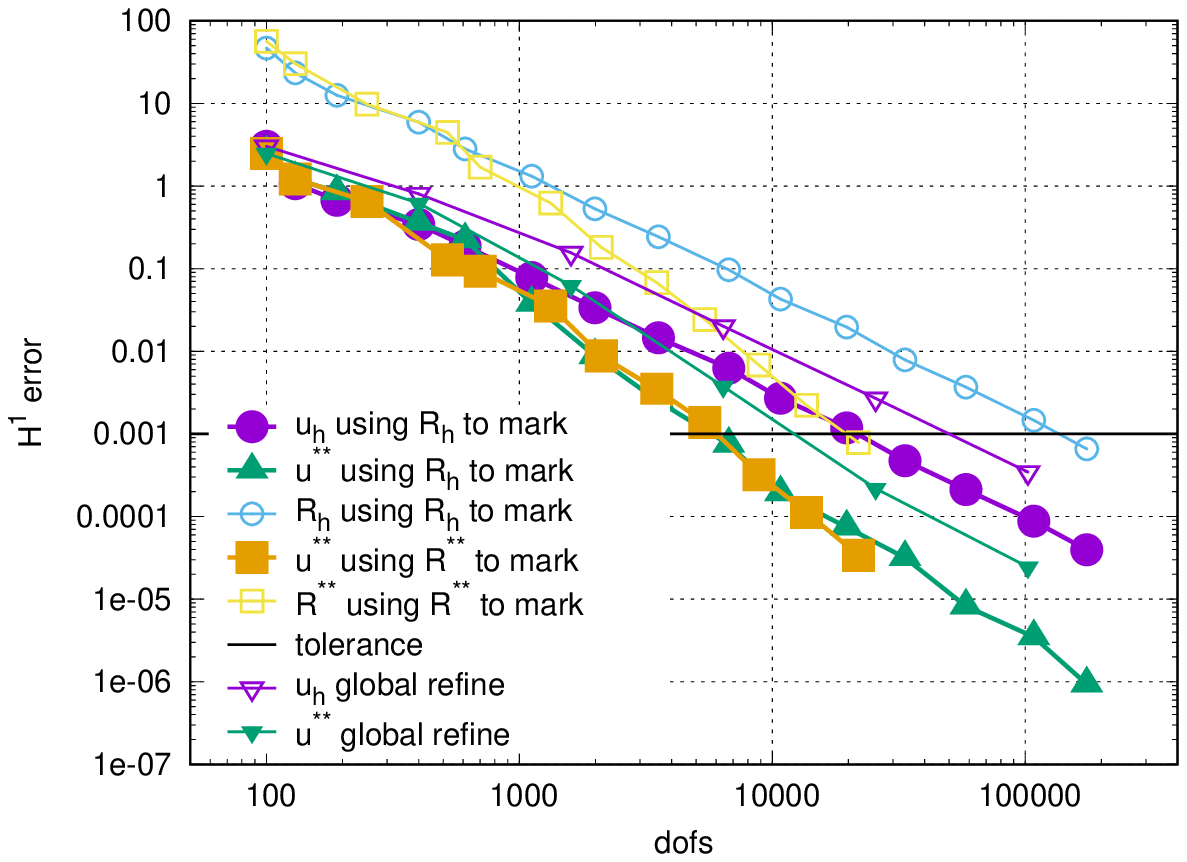}
  \caption{$H^1$ errors and residuals for
  smooth problem with $p=1,2,3$ (left to right).
  Results are shown for an adaptive mesh using a tolerance of
  $0.2,0.01,0.001$ for $p=1,2,3$, respectively.
  }
  \label{fig:adapt2dsmooth}
\end{figure}

For our final test we study a reentrant corner type problem, i.e.,
$\Omega=(-1,1)^2\setminus([0,1]\times[-1,0])$ using a regular triangulation.
First we choose the well known exact solution $u\in \sobh{\frac{3}{2}}$
leading to $f=0$. Since the solution is not even $\sobh2$ we can not expect
the postprocessed solution to have an increased convergence rate.  This 
is confirmed by our numerical tests summarized in
Figure~\ref{fig:global2dcorner}. Due to the reduced smoothness and the
simplicity of the solution away from the corner, the postprocessing does not
only not improve the EOC but can even lead to a slight increase in the
overall error clearly noticeable in the $\sobh1$ error for the $p=2$ case.
This is even more obvious when the postprocessor, $\uhs$, is used directly.
Alternatively, going
from $\uhs$ to $\uhss$ leads to an approximation which is very close to
the original $u_h$ in all cases.
Although the results for the globally refined grid are not that promising,
the postprocessing nevertheless has considerable benefits when adapting the
grid using the residual indicator based on $\uhss$.
Indeed Figure \ref{fig:adapt2dcorner} shows that, for a given number of dofs, mesh adaptation based on $R^{**}$  produces an approximation
$\uhss$ which has a much smaller error  than $u_h$ (on a mesh constructed using $R_h$).

\begin{figure}
\centering
  \includegraphics[width=0.45\textwidth]{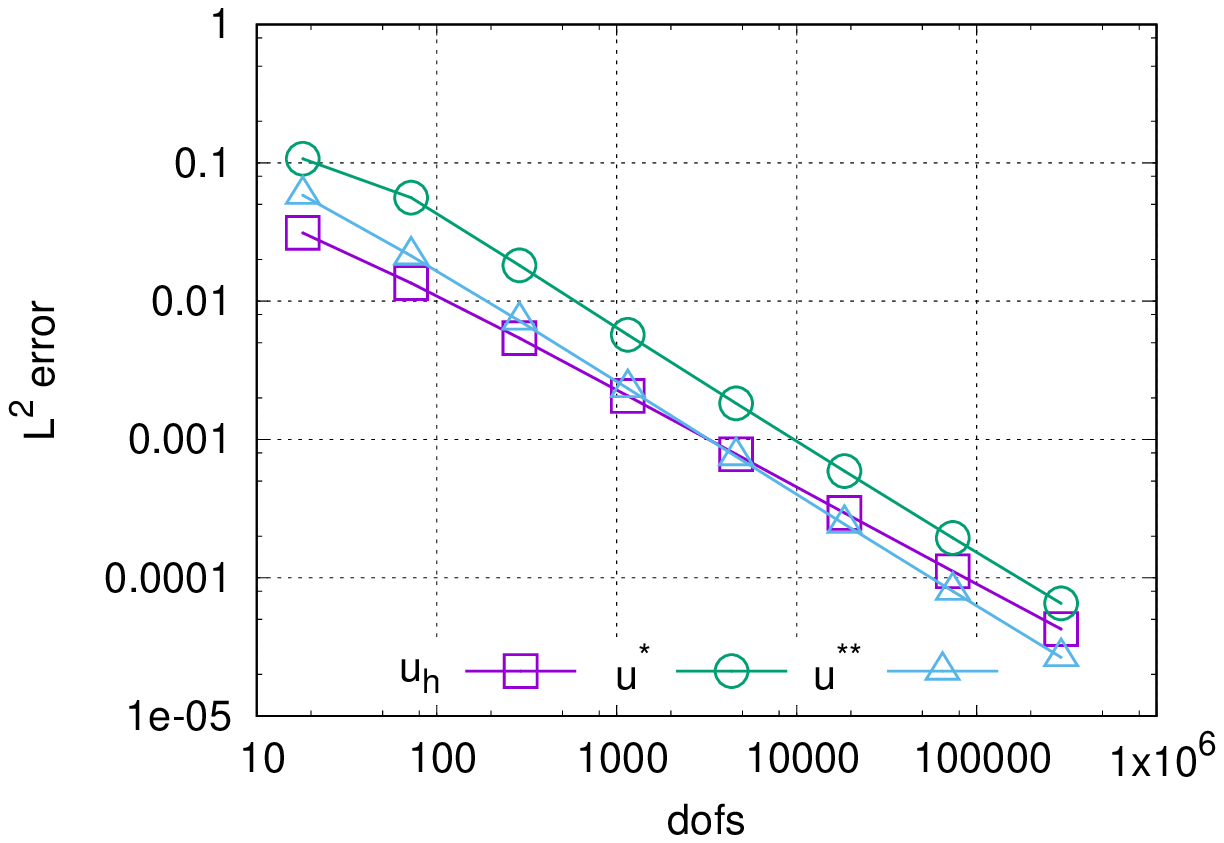}
  \includegraphics[width=0.45\textwidth]{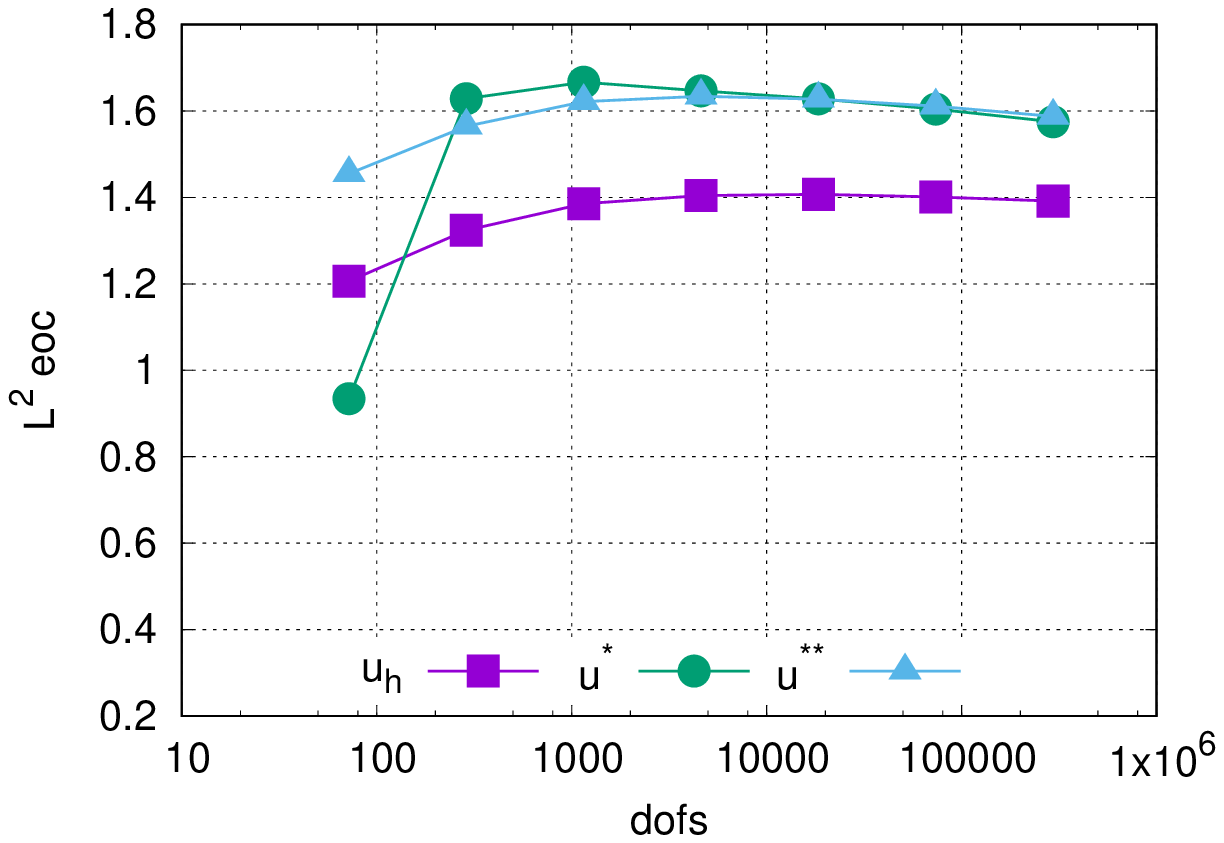}
  \\
  \includegraphics[width=0.45\textwidth]{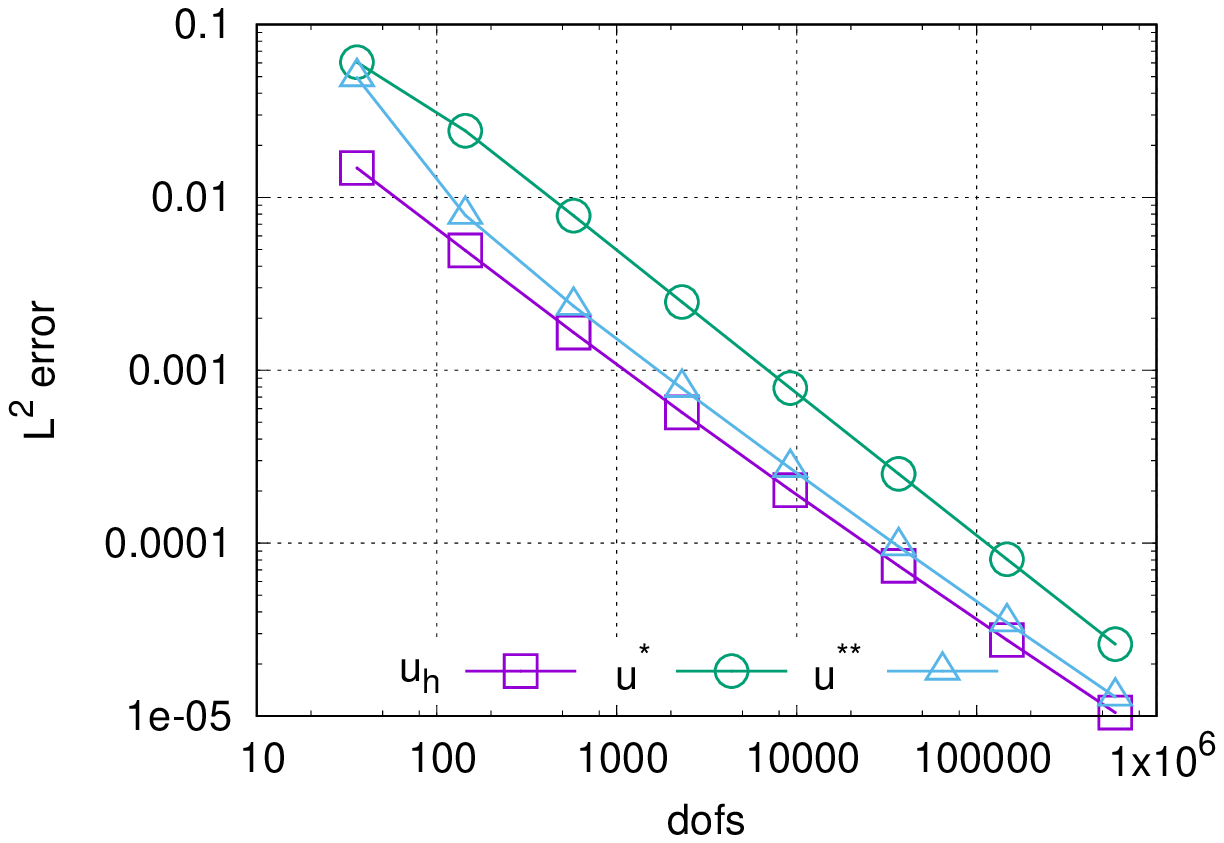}
  \includegraphics[width=0.45\textwidth]{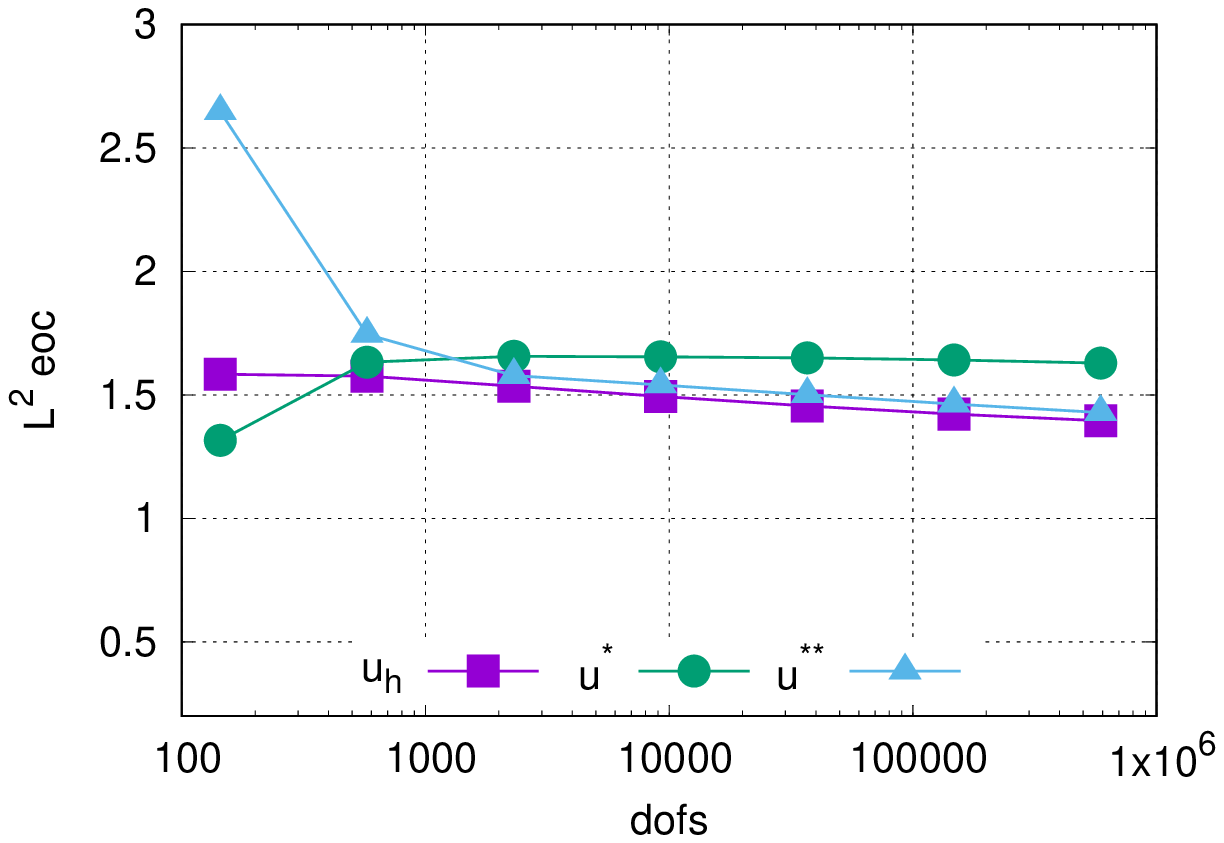}
  \\
  \includegraphics[width=0.45\textwidth]{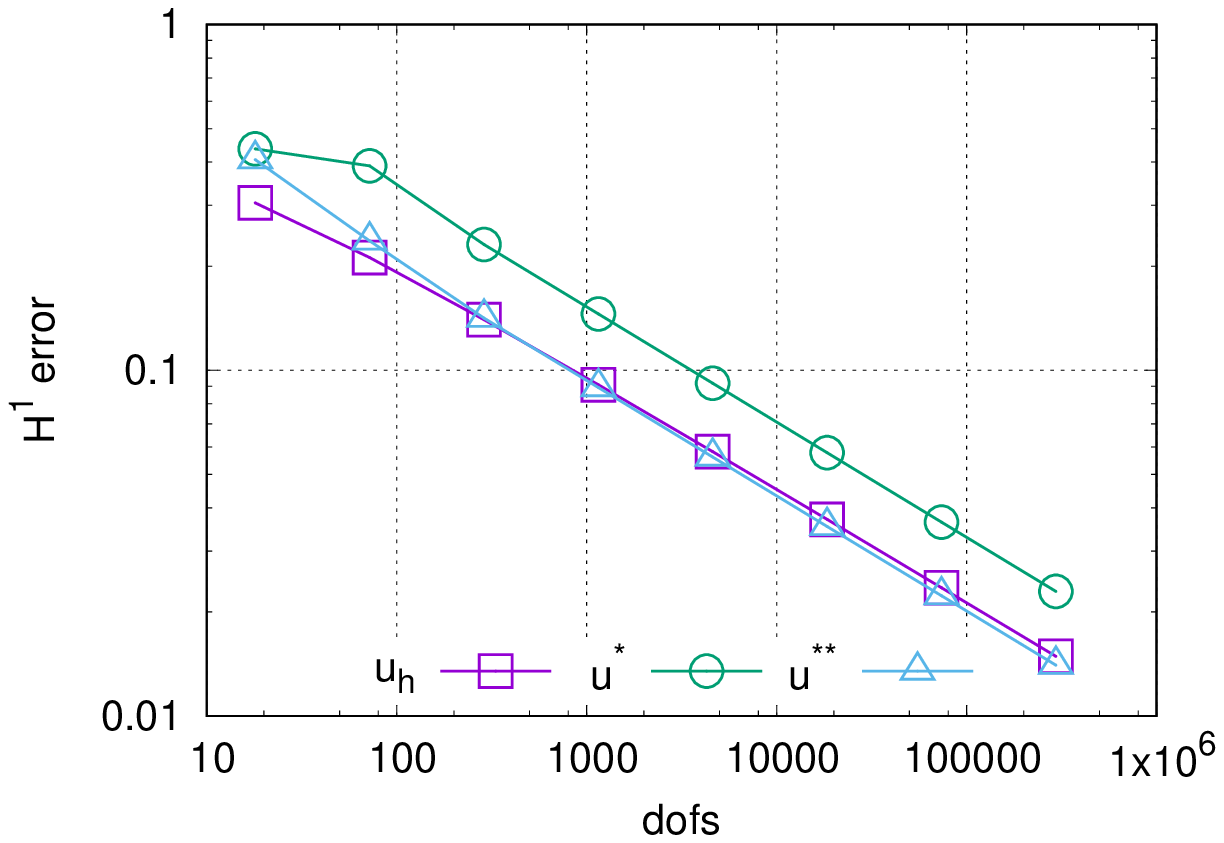}
  \includegraphics[width=0.45\textwidth]{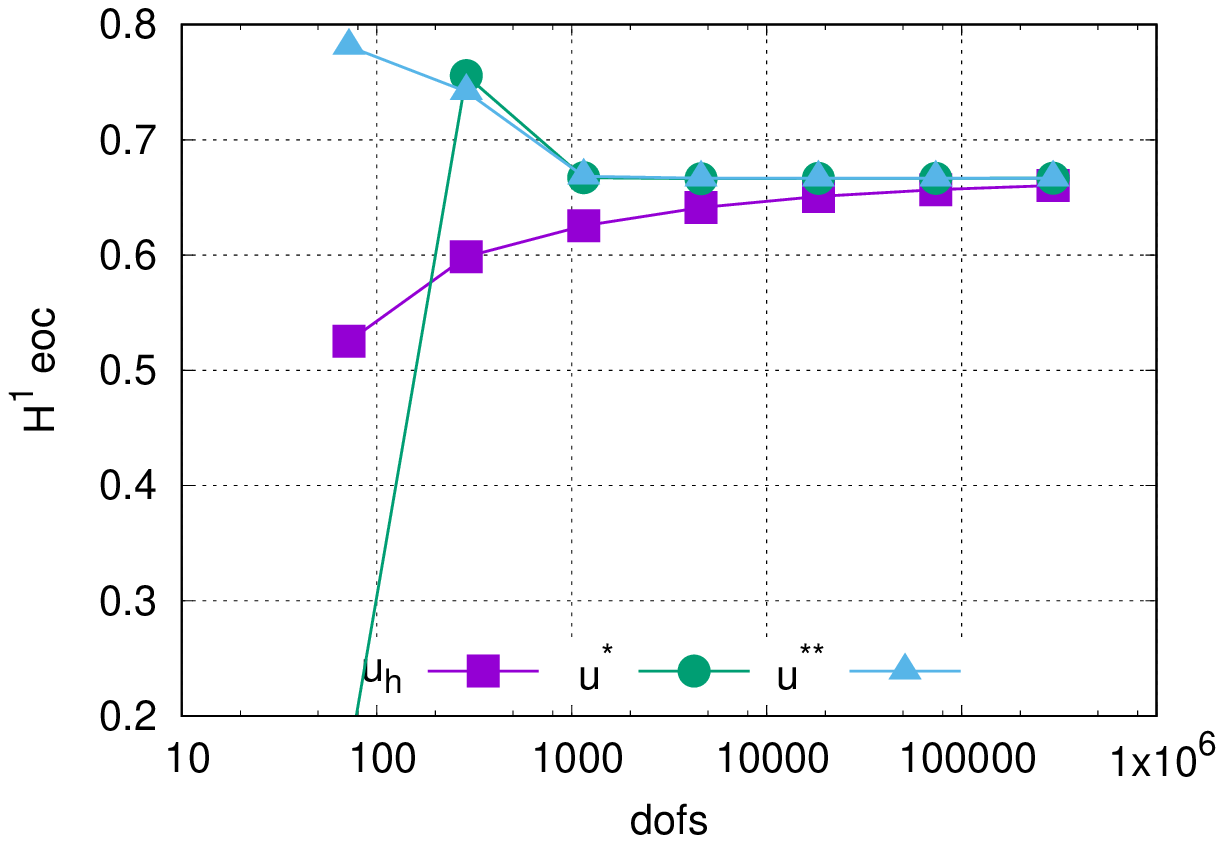}
  \\
  \includegraphics[width=0.45\textwidth]{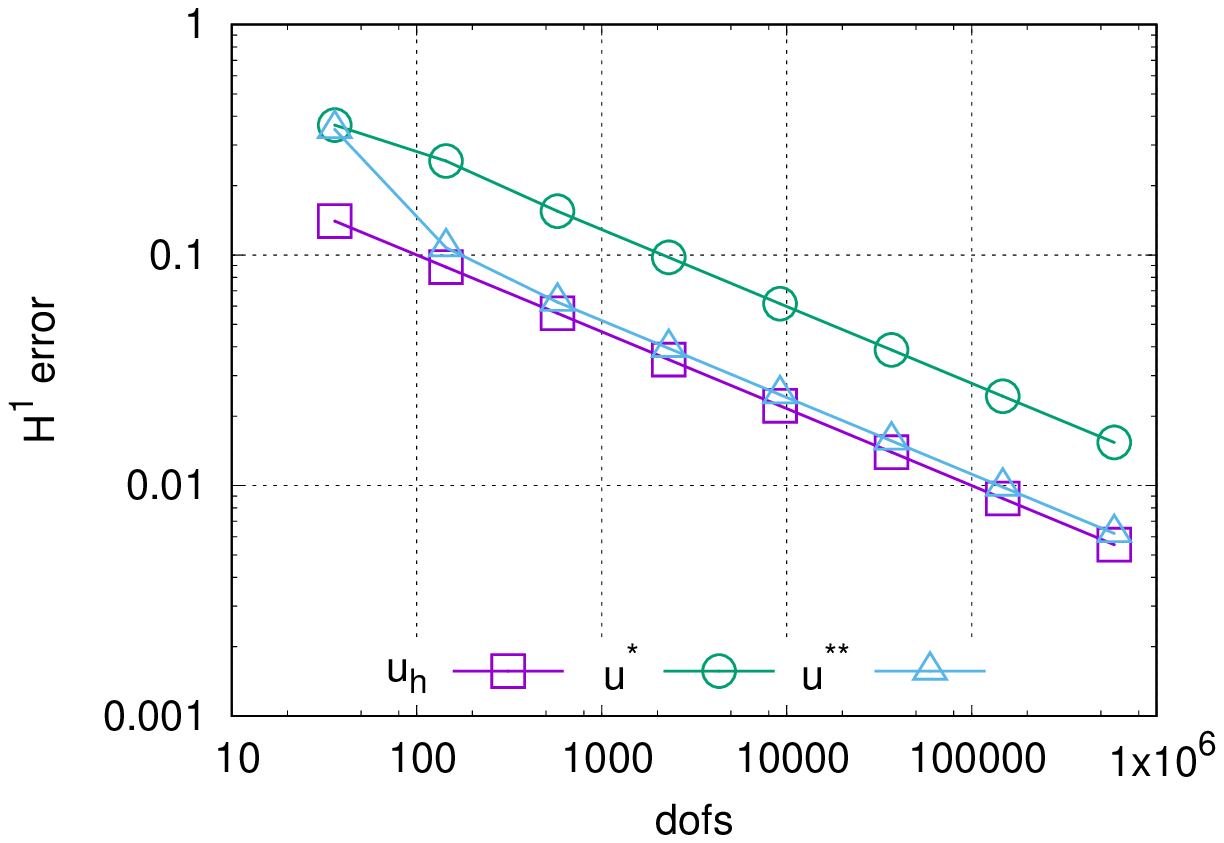}
  \includegraphics[width=0.45\textwidth]{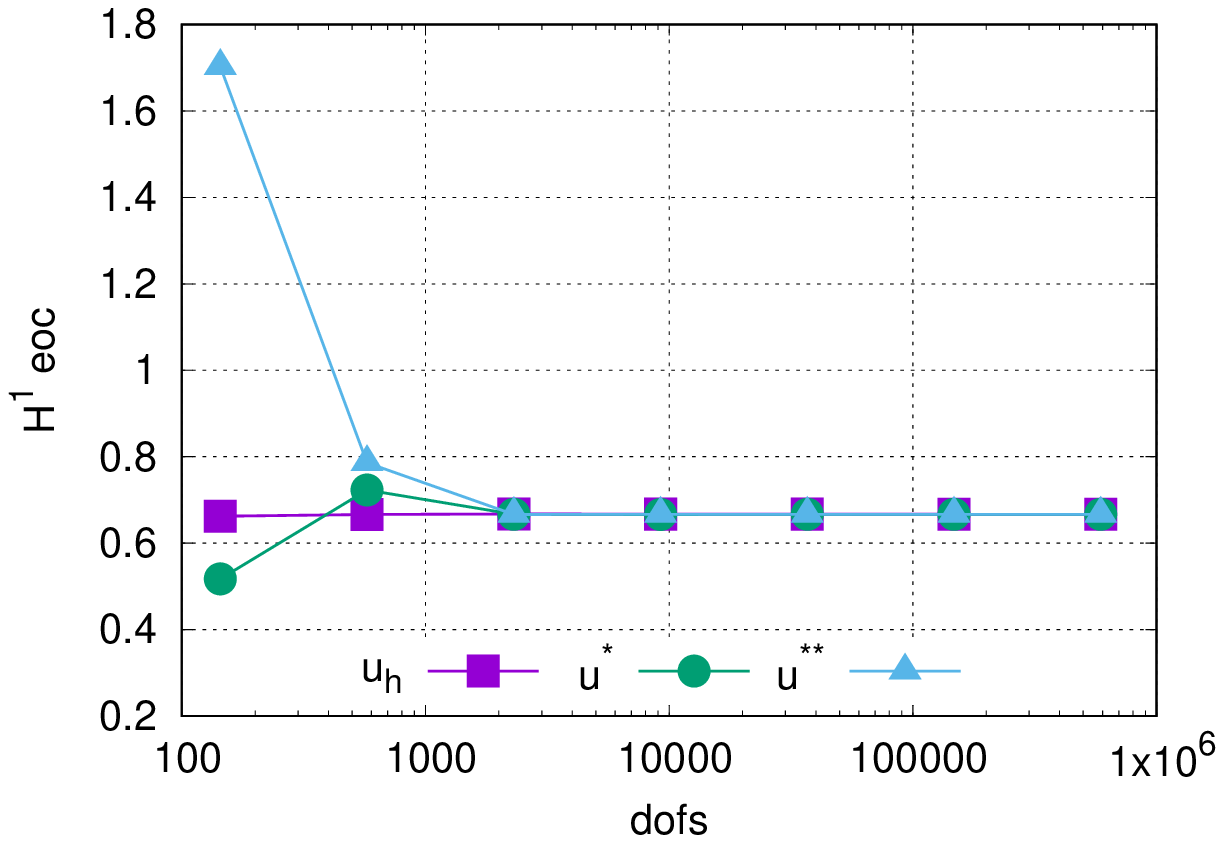}
  \caption{Errors (right) and EOCs (left) for simple corner problem.
  From top to bottom: $\leb2$ with $p=1,2$ and $\sobh1$ with $p=1,2$}
  \label{fig:global2dcorner}
\end{figure}
\begin{figure}
\centering
  \includegraphics[width=0.49\textwidth]{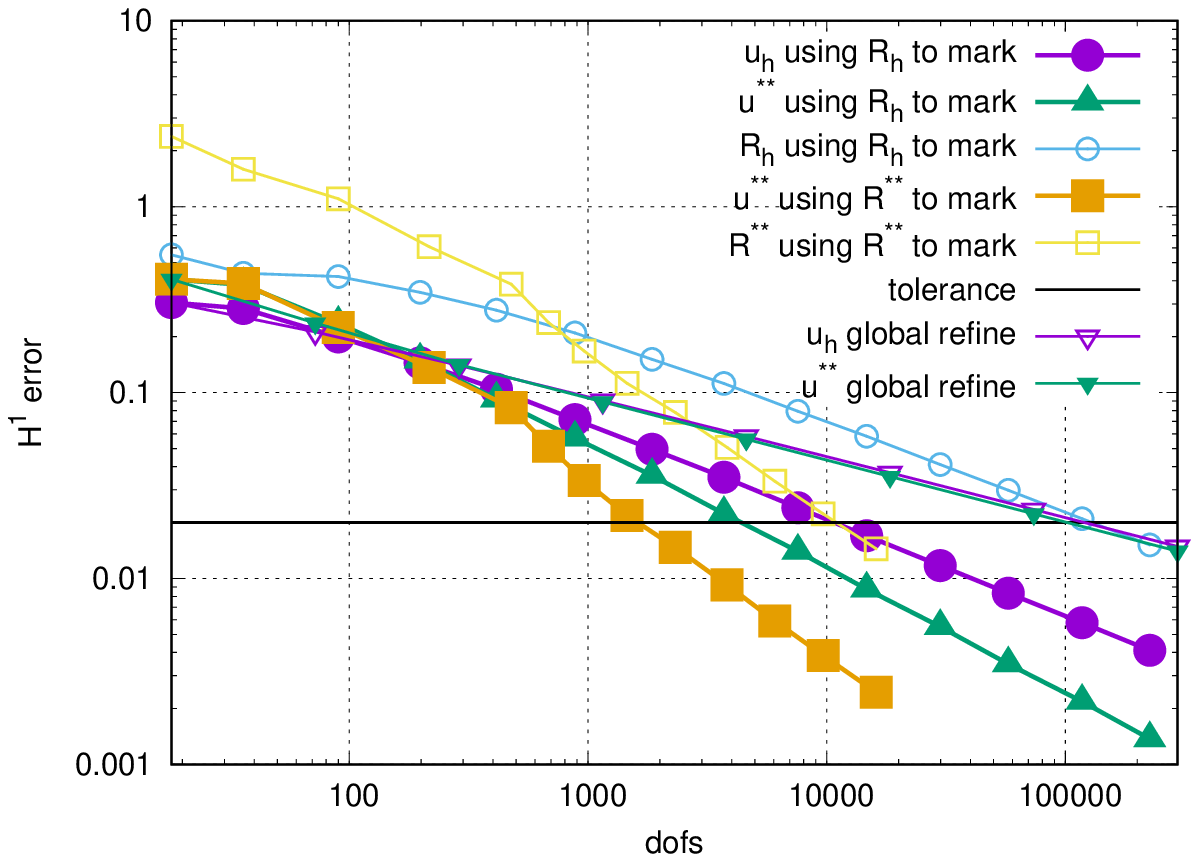}
  \includegraphics[width=0.49\textwidth]{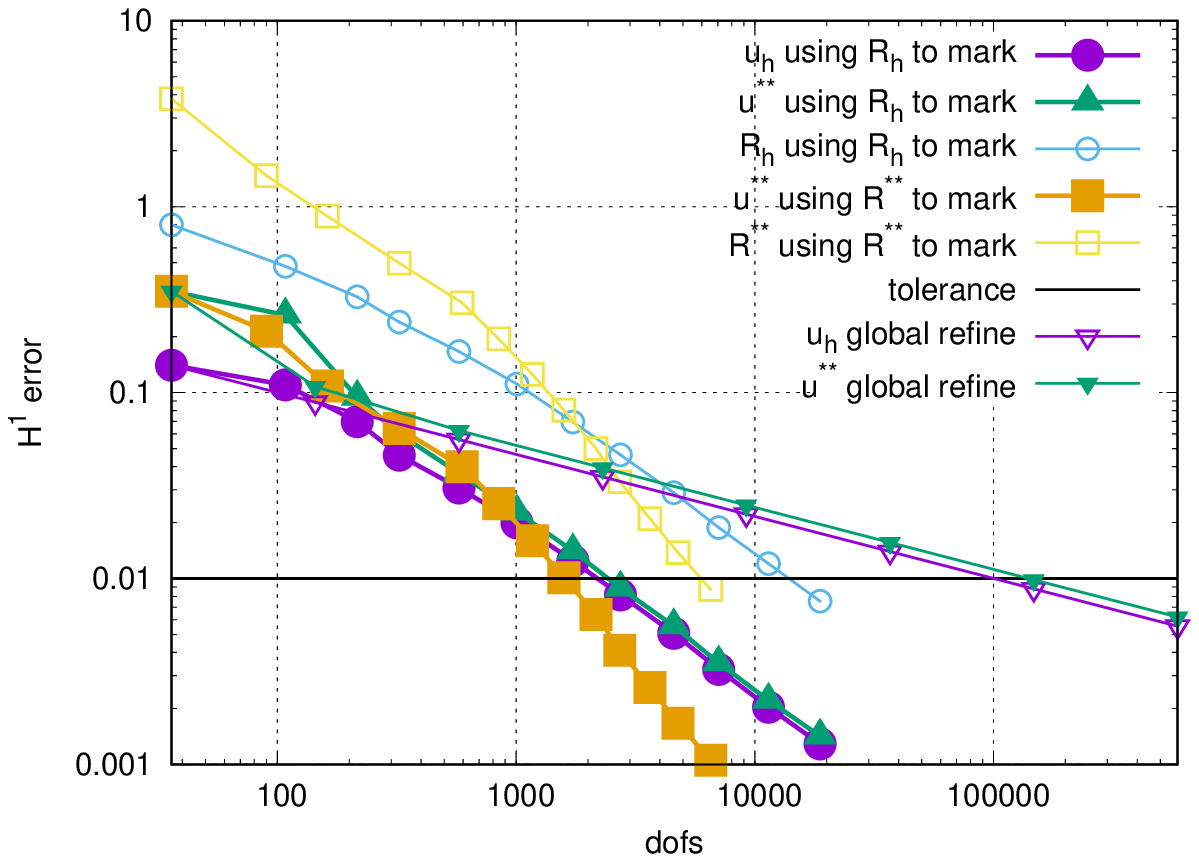}
  \caption{$\sobh1$ errors and residuals for
  simple corner problem with $p=1$ (left) and $p=2$ (right).
  Results are shown for an adaptive mesh using a tolerance of $0.01$.
  }
  \label{fig:adapt2dcorner}
\end{figure}

For a more challenging test, especially for $p=3$, we construct the forcing function
so that the exact solution is $u(x,y)=\omega(x,y)u_{\rm corner}(x,y)$,
where $u_{\rm corner}$ is the solution to the above corner problem and
$\omega(x,y)=-\sin{\frac{3}{2}\pi(1-x^2)(1-y^2)}$. The
function $u$ still has the same corner singularity but is also smooth. However, the challenging nature
of this solution is that it has large
gradients towards the outer boundaries.
Results for $p=2,3$ are summarized in Figures
\ref{fig:global2dextendedcorner} and \ref{fig:adapt2dextendedcorner}.
The final grids for $p=3$ are shown in Figure
\ref{fig:adapt2dextendedcornergrid} using $R_h$ and $R^{**}$ to mark
cells for refinement. In both cases $22$ steps were needed
and the resulting grids have $1597$ and $4540$ cells
($20725$ and $7381$ degrees of freedom), respectively.
While the corner is highly refined in both cases, the regions that are smooth
but strongly varying in their solution are far less refined when using $R^{**}$.
When using $R_h$ the final errors are
$\|u-u_h\|_{dG}\approx 7.4\cdot 10^{-4}$ and
$\|u-\uhss\|_{dG}\approx 6.8\cdot 10^{-4}$
while adaptivity based on
$R^{**}$ results in errors of the size
$\|u-u_h\|_{dG}\approx 3.9\cdot 10^{-3}$ and
$\|u-\uhss\|_{dG}\approx 7.0\cdot 10^{-4}$. Because of   the corner singularity,
using the postprocessor after finishing the refinement (based on $R_h$) does not lead to
a significant improvement while basing the adaptive process on $R^{**}$
leads to an almost identical error while requiring only $35\%$ of the
cells.

\begin{figure}
\centering
  \includegraphics[width=0.45\textwidth]{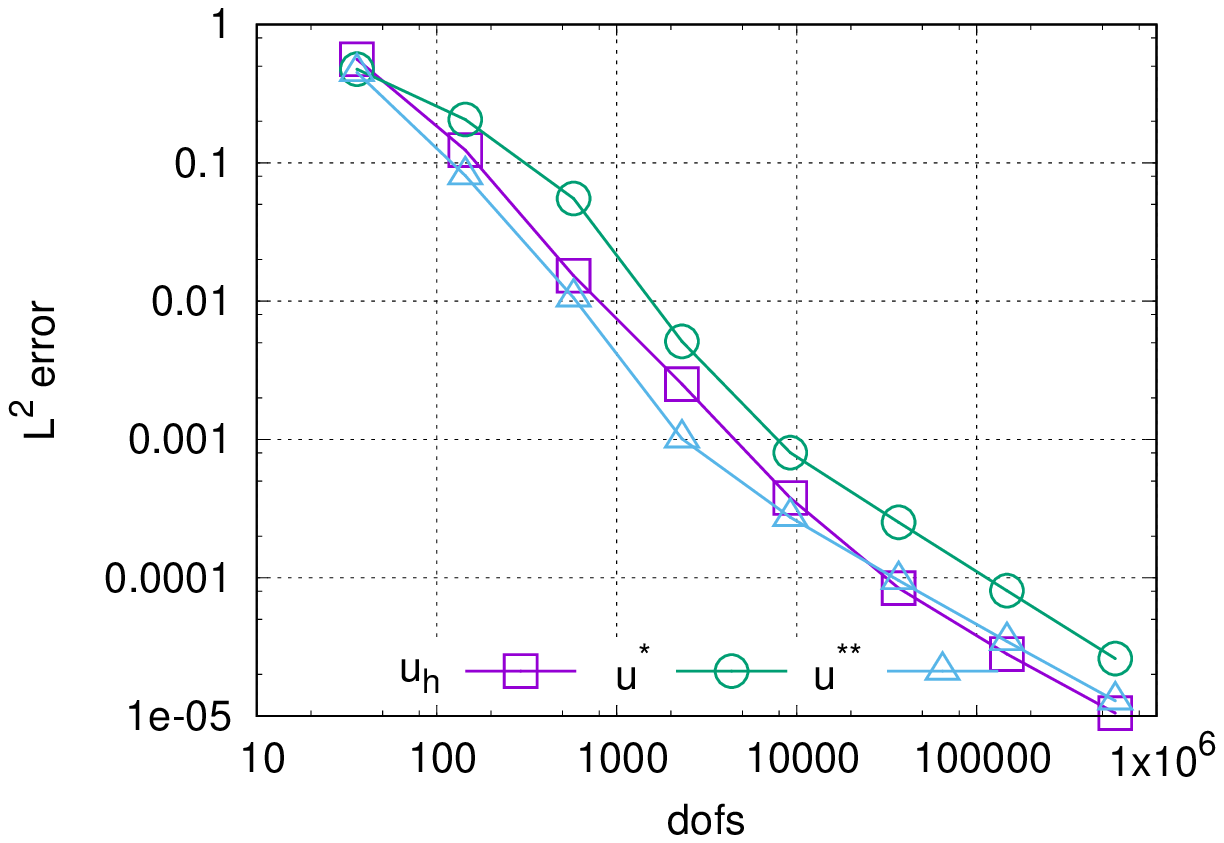}
  \includegraphics[width=0.45\textwidth]{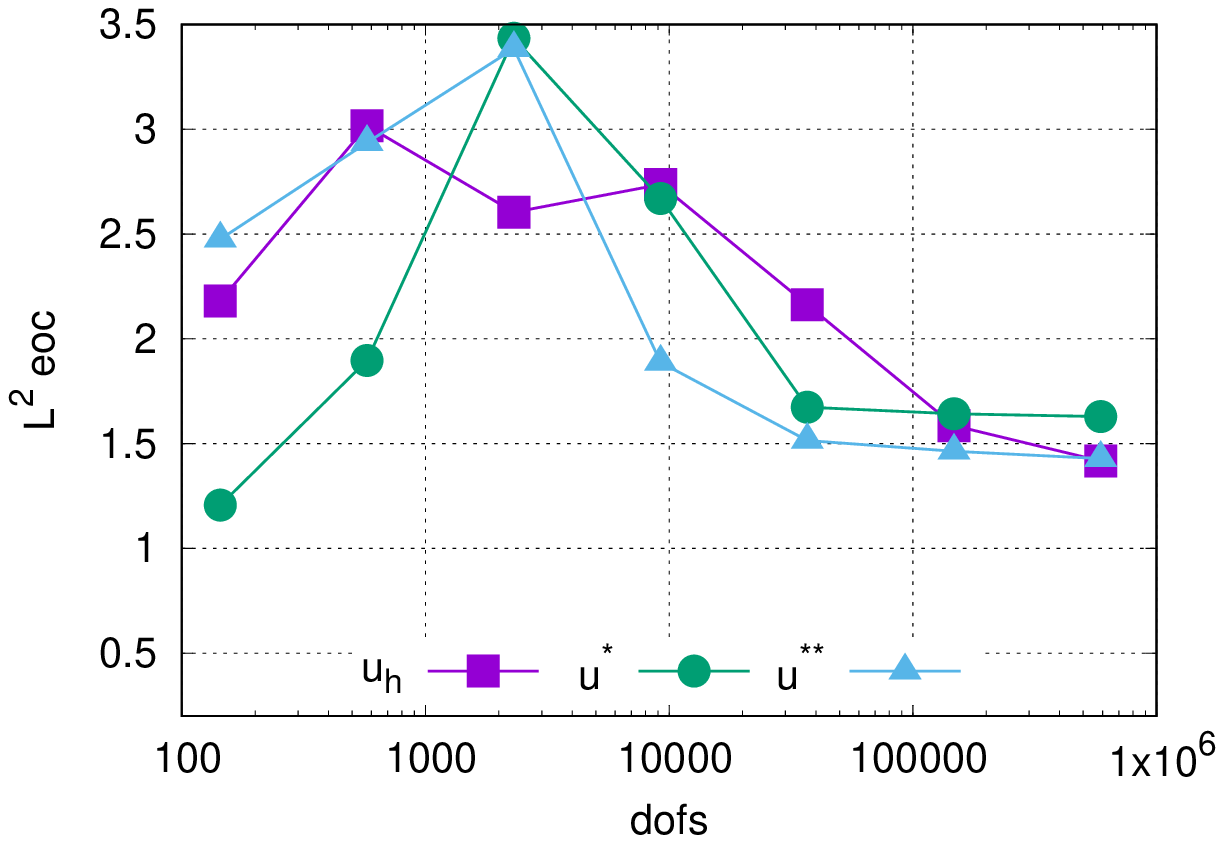}
  \\
  \includegraphics[width=0.45\textwidth]{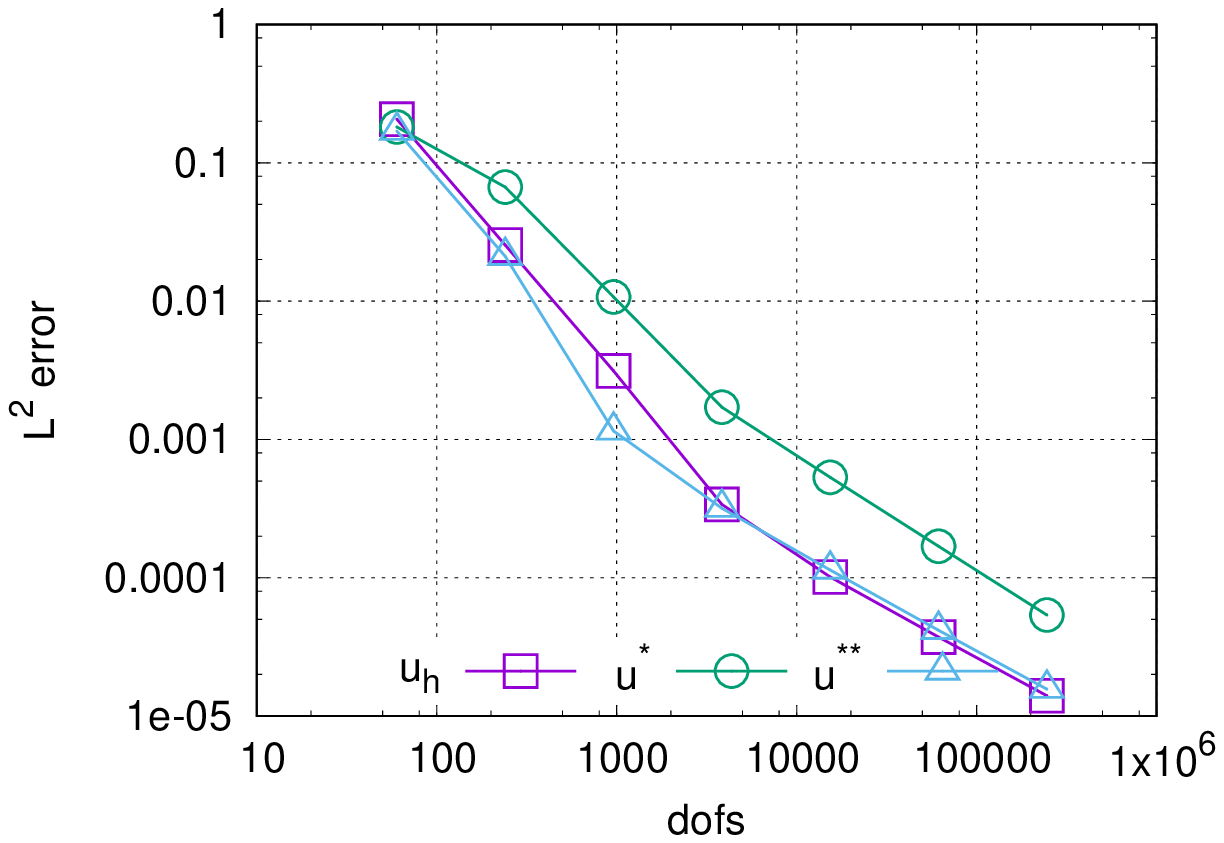}
  \includegraphics[width=0.45\textwidth]{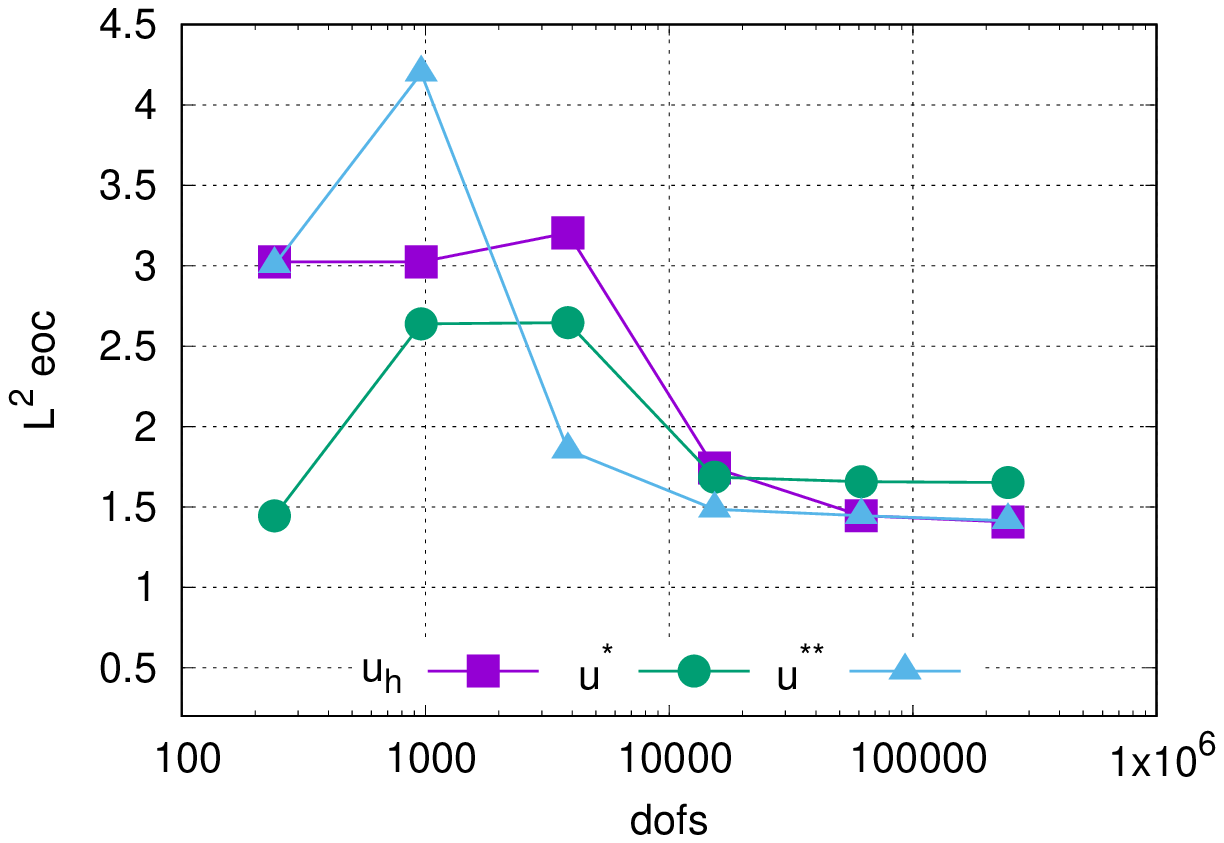}
  \\
  \includegraphics[width=0.45\textwidth]{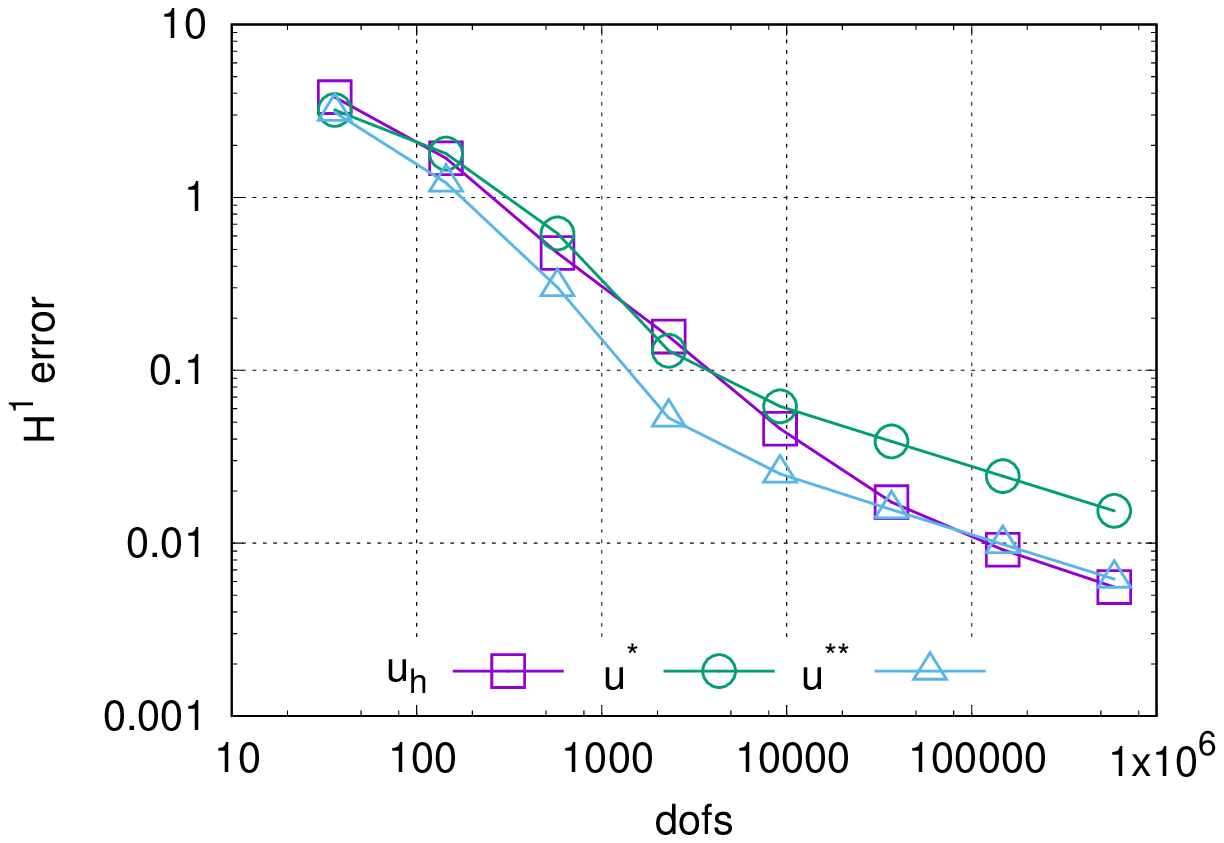}
  \includegraphics[width=0.45\textwidth]{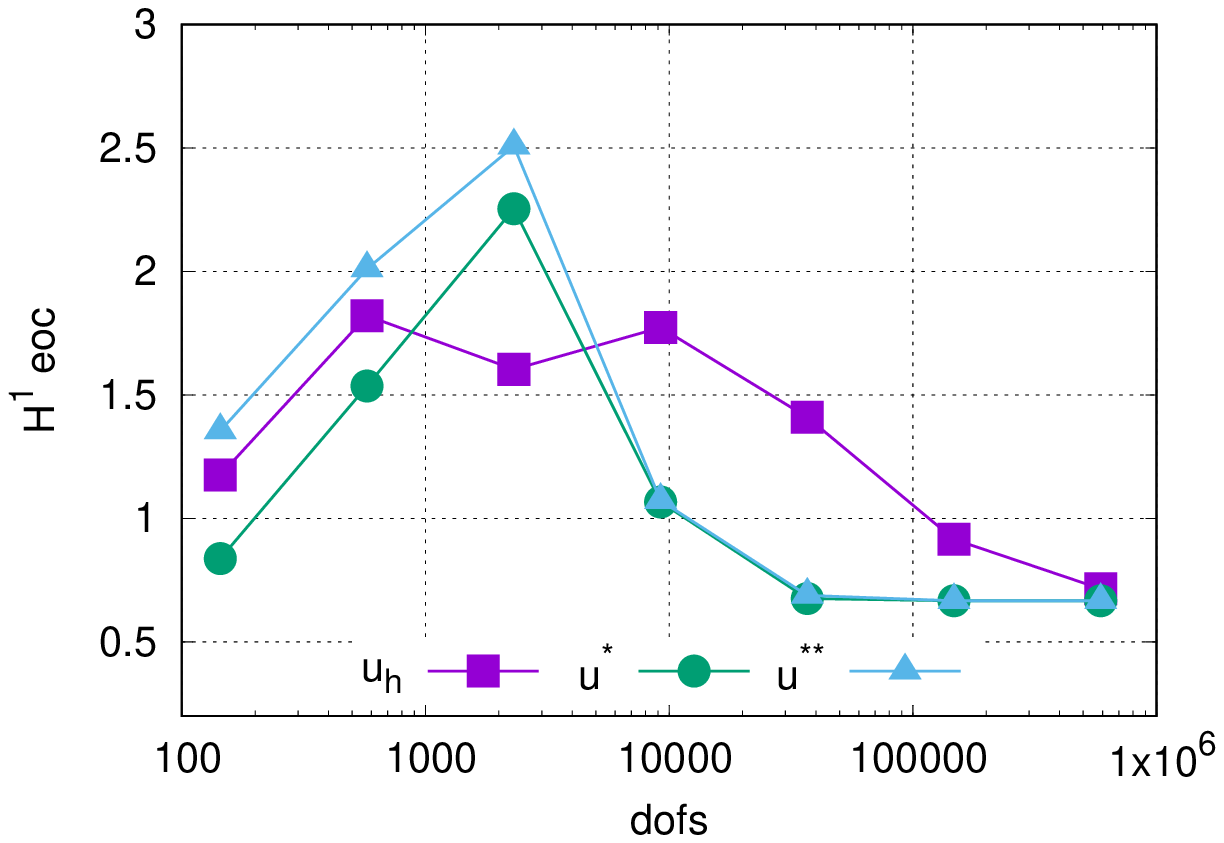}
  \\
  \includegraphics[width=0.45\textwidth]{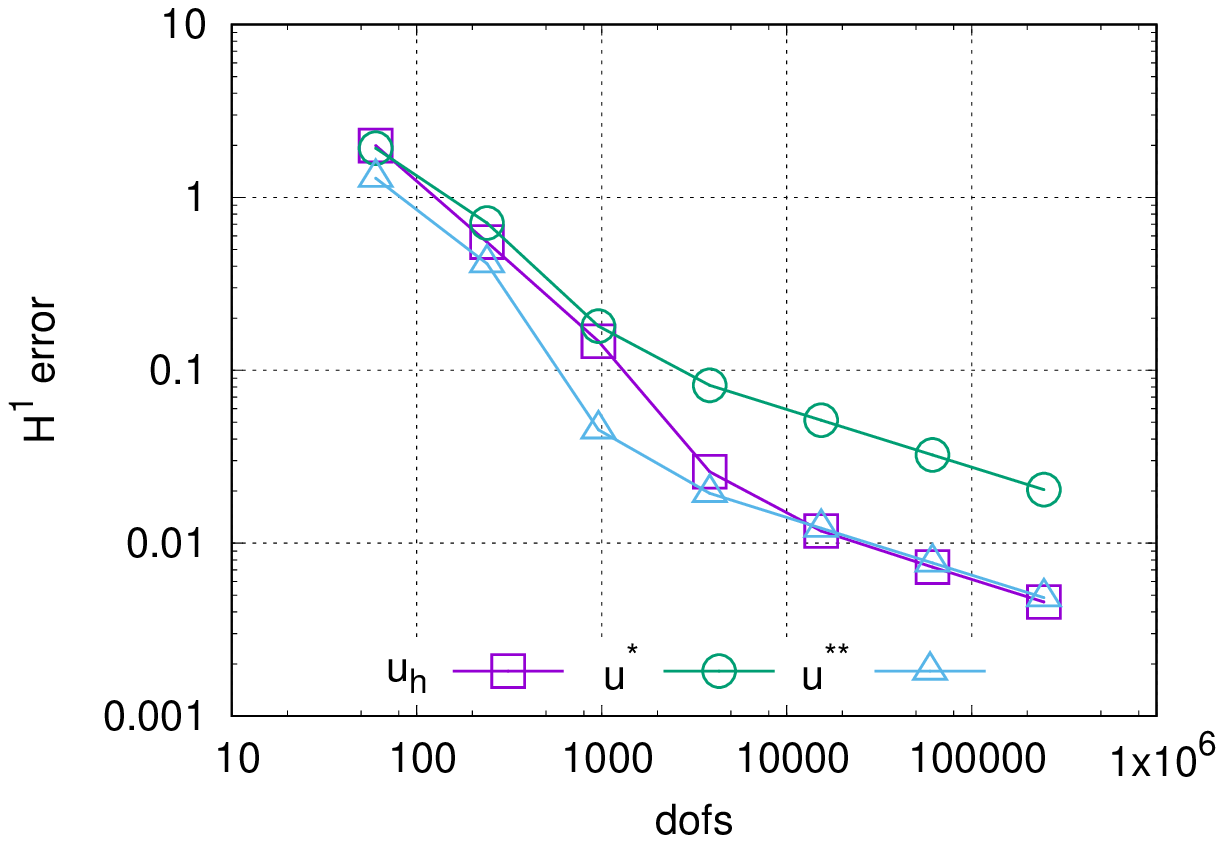}
  \includegraphics[width=0.45\textwidth]{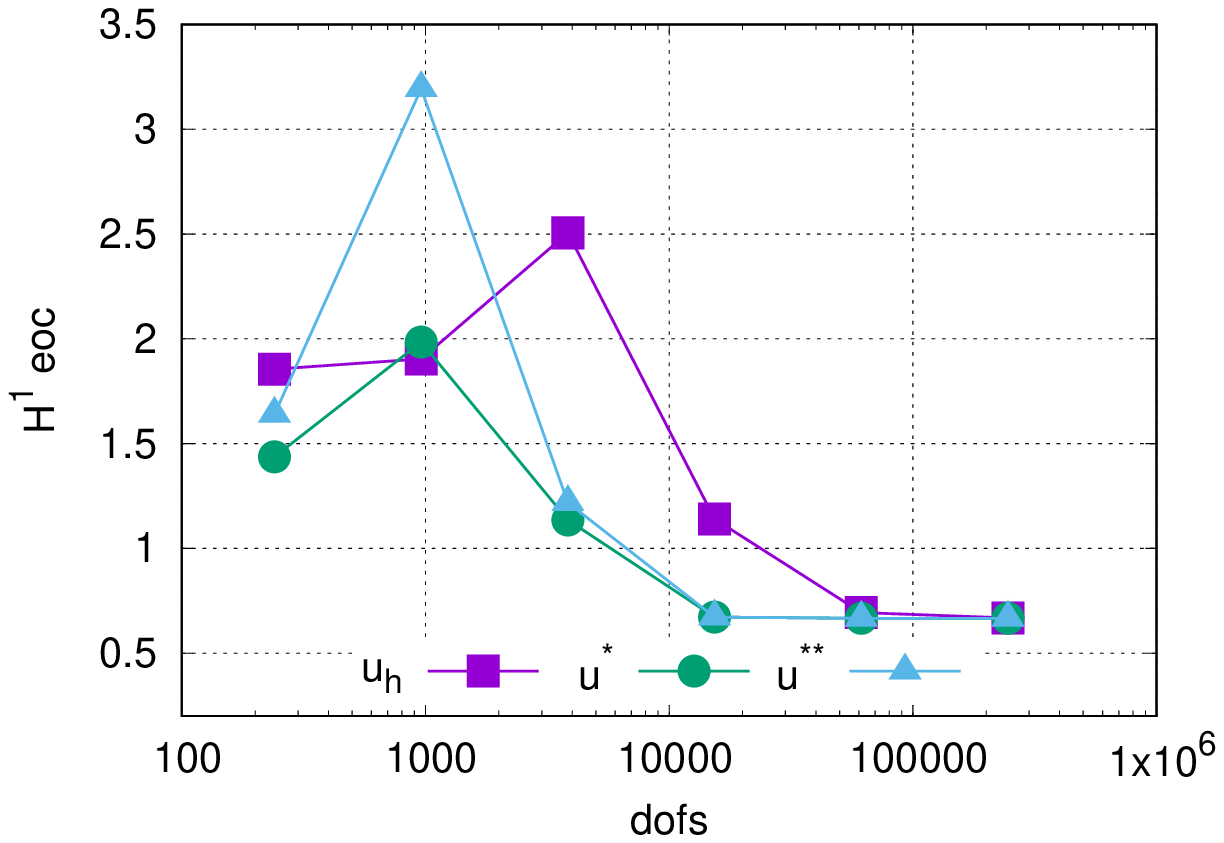}
  \caption{Errors (right) and EOCs (left) for extended corner problem.
  From top to bottom: $\leb2$ with $p=2,3$ and $\sobh1$ with $p=2,3$}
  \label{fig:global2dextendedcorner}
\end{figure}
\begin{figure}
\centering
  \includegraphics[width=0.49\textwidth]{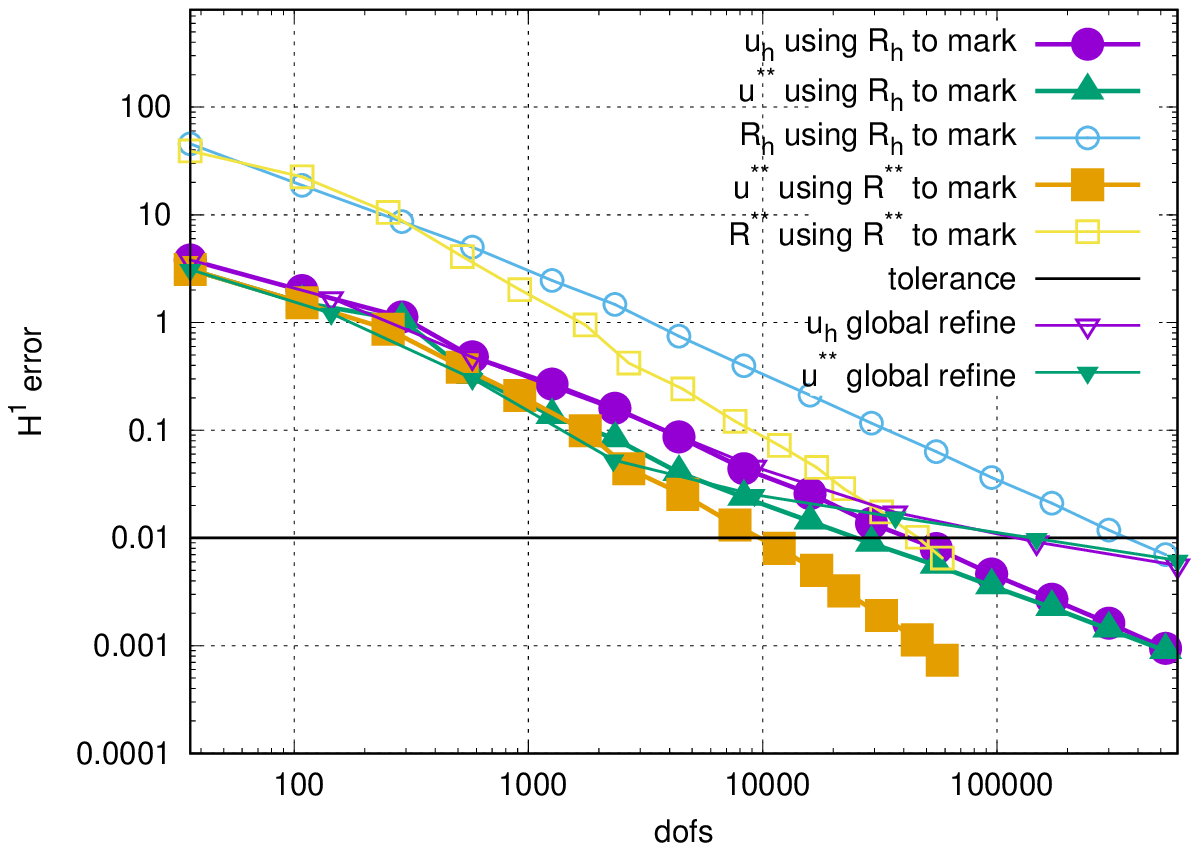}
  \includegraphics[width=0.49\textwidth]{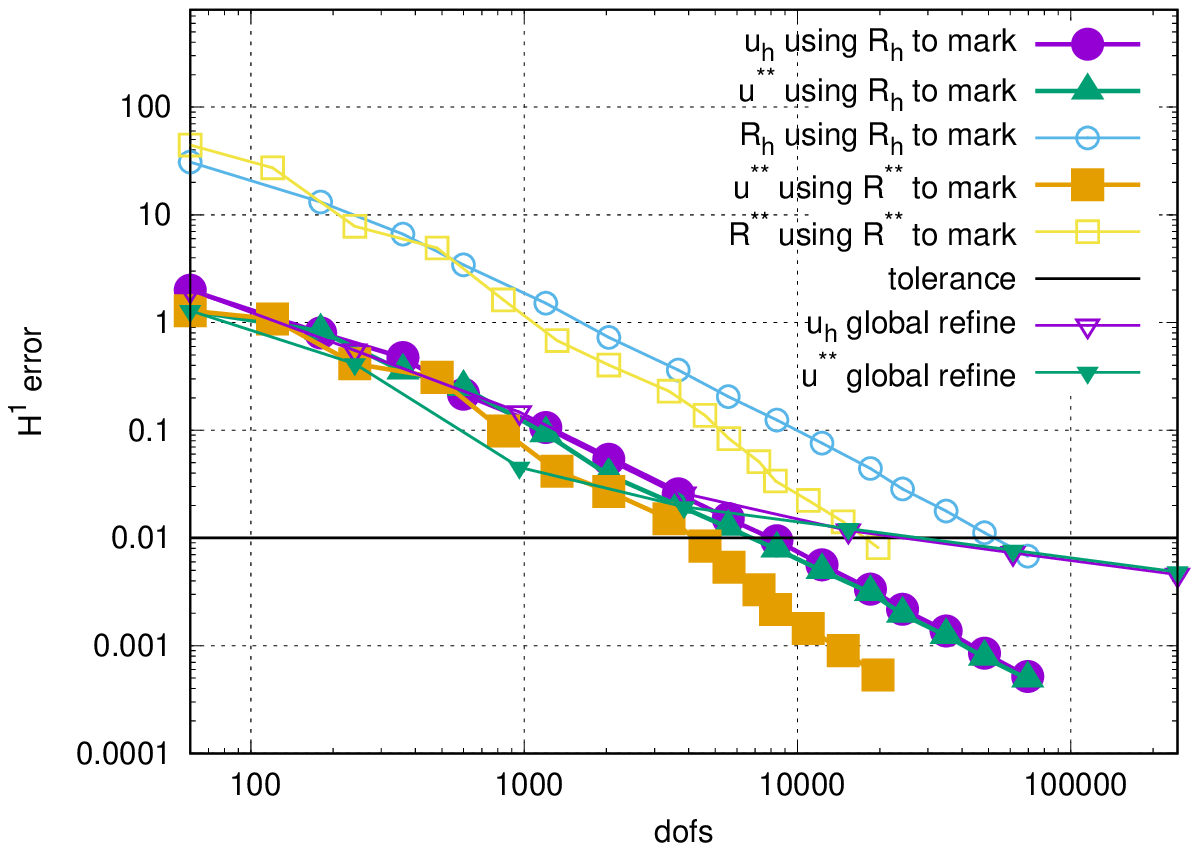}
  \caption{$\sobh1$ errors and residuals for
  extended corner problem with $p=2$ (left) and $p=3$ (right).
  Results are shown for an adaptive mesh using a tolerance of $0.01$.
  }
  \label{fig:adapt2dextendedcornergrid}
\end{figure}

\begin{figure}
\centering
  \includegraphics[width=0.9\textwidth]{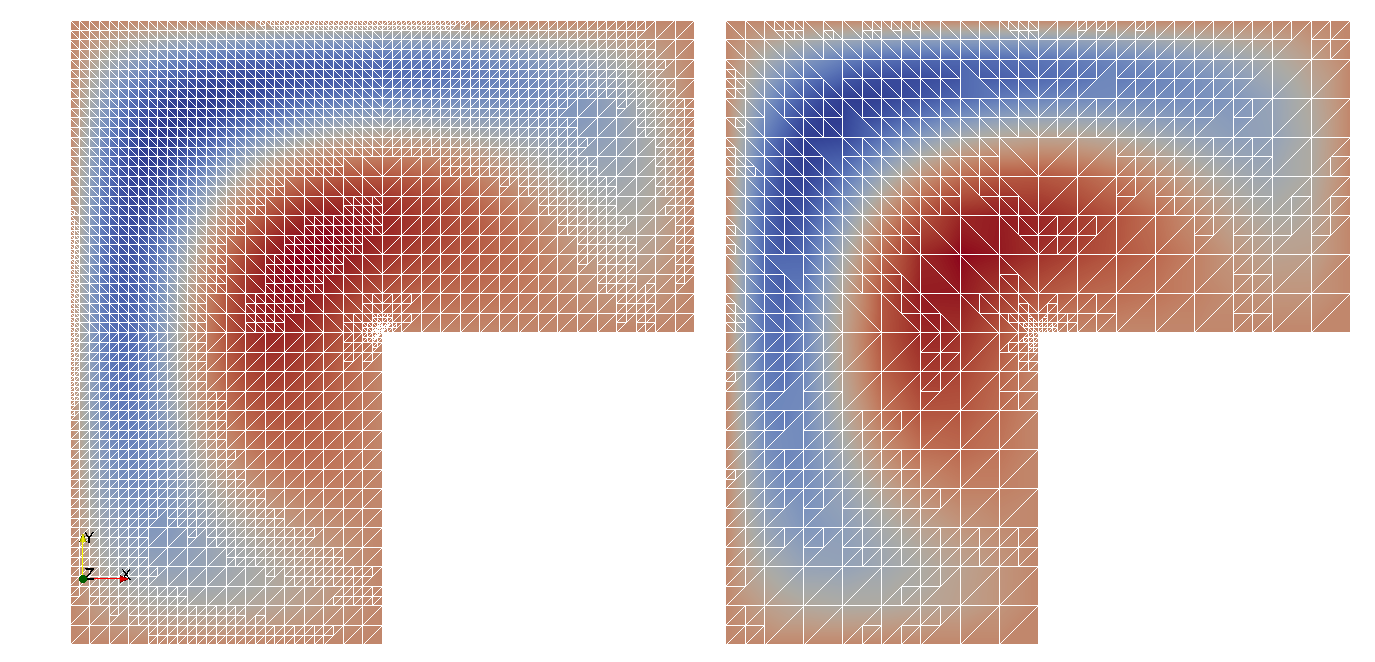}
  \caption{Discrete solution and adapted grid for extended corner problem
  with $p=2$ using $R_h$ (left) and $R^{**}$ (right) for marking.
  The iterate is chosen so that the
  resulting errors satisfty $e_h\approx e^{**}\approx 0.013$.
  }
  \label{fig:adapt2dextendedcorner}
\end{figure}

Figure \ref{fig:global2deffindex} shows the efficiency index for all three
test cases on globally refined grids. The results seem to indicate that
there is only a slight increase in the efficiency index
$\frac{R^{**}}{\|\nabla(\uhss-u)\|}$ compared to
$\frac{R_h}{\|\nabla(u_h-u)\|}$.

\begin{figure}
\centering
  \includegraphics[width=0.32\textwidth]{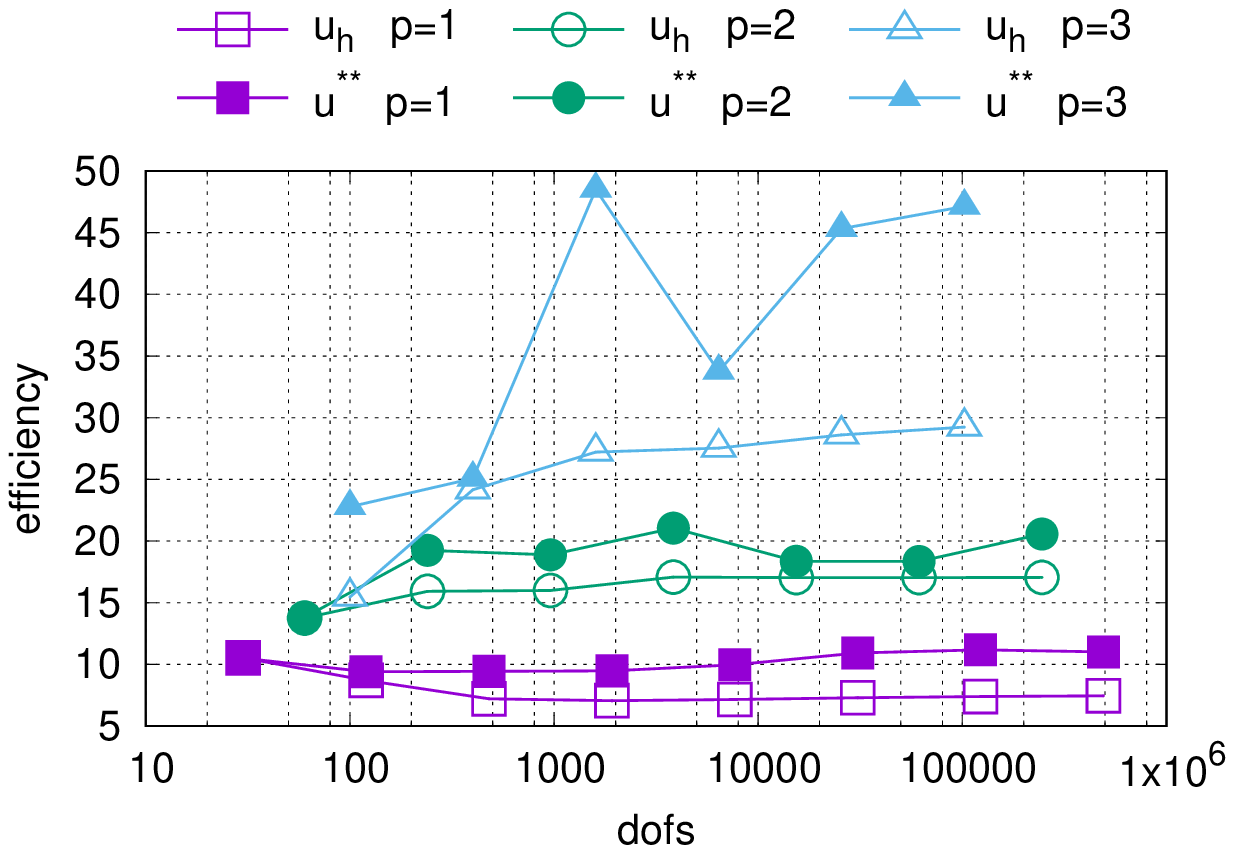}
  \includegraphics[width=0.32\textwidth]{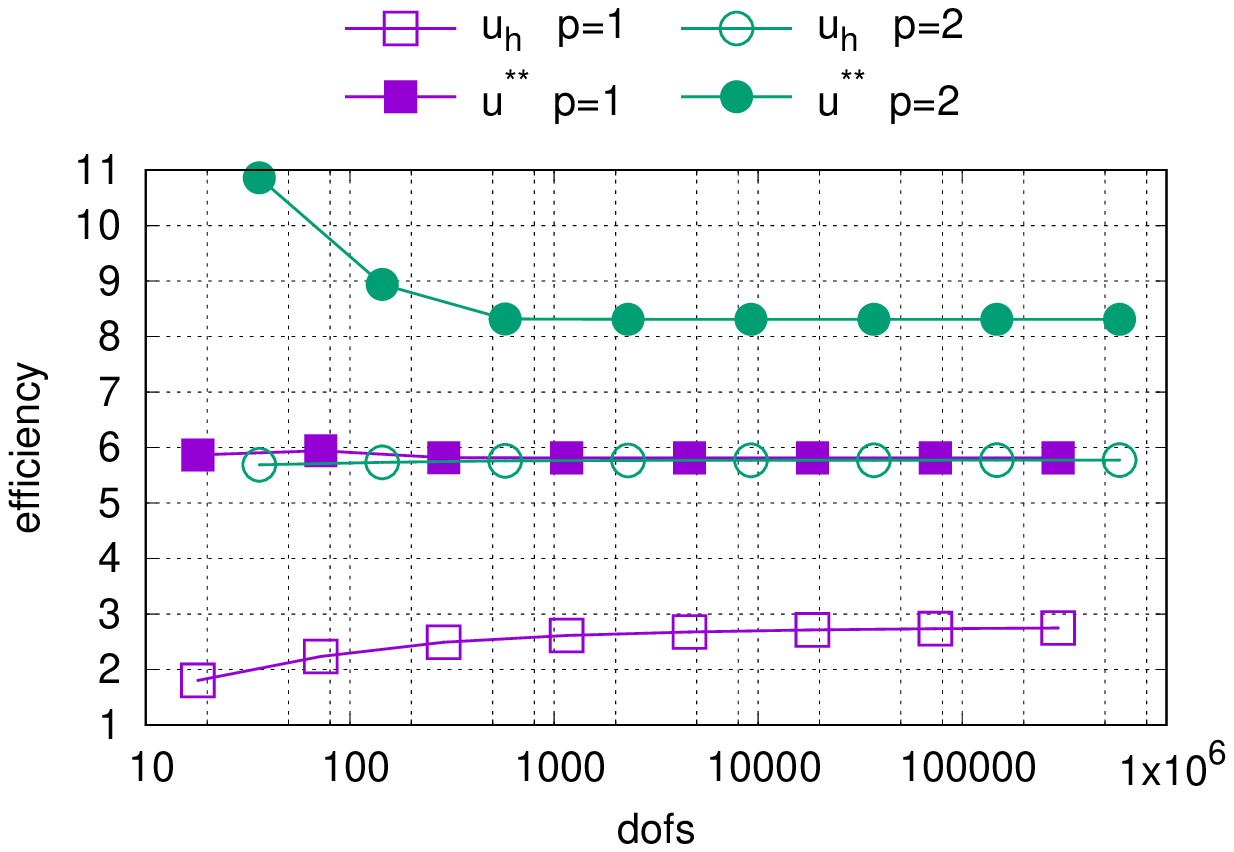}
  \includegraphics[width=0.32\textwidth]{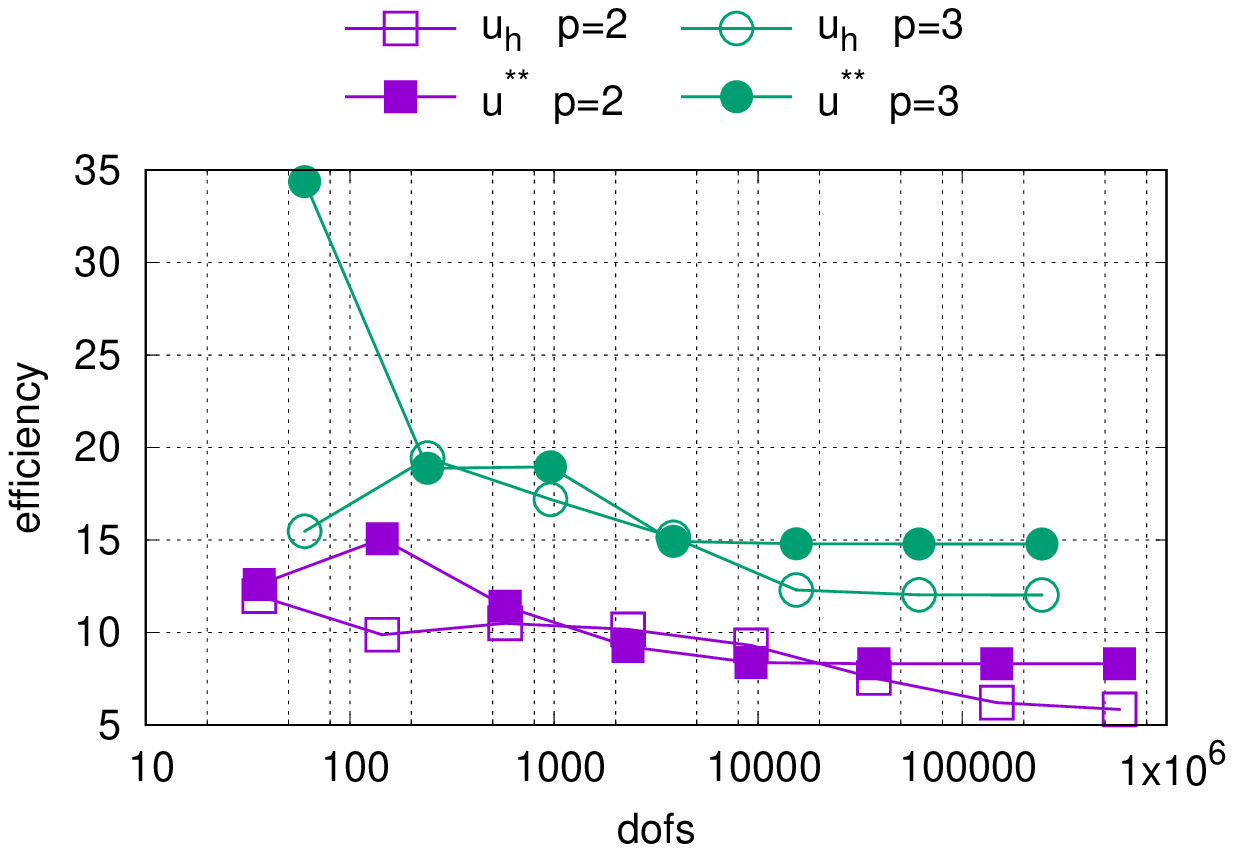}
  \caption{Efficiency index on globally refined grids for
  smooth problem with $p=1,2,3$ (left),
  simple corner problem with $p=1,2$ (middle),
  and extended corner problem with $p=2,3$ (right)}
  \label{fig:global2deffindex}
\end{figure}

\section{Summary Discussion}
In this article, we provide a strategy for improving existing post-processing strategies for 
numerical solutions of a model elliptic problem.
The main idea is to modify the post-processed solution so that it satisfies Galerkin orthogonality. We prove various {\it a priori} type results showing desirable
convergence properties of the orthogonal post-processor including an
increased order of accuracy in the $\leb{2}$ norm.
We supported the
analysis with numerical examples using two types of post-processors --
that of SIAC and SPR -- approximating smooth and non-smooth solutions. 

In addition to the a priori results, we provide a reliable and efficient {\it a
  posteriori} error estimator for the orthogonal post-processed solution, it should be noted tht no such estimator is available for $u - u^*$. 
  We demonstrate in several examples that much more efficient meshes are obtained
when adaptation is based on $R^{**}$ than when  refinement is based on $R_h$ and the post-processor
is only applied to the numerical solution on the final mesh.


\section*{Acknowledgement}
This work was initiated during the authors' stay in Edinburgh with an ICMS ``research-in-group'' grant. J.G. thanks
the German Research Foundation (DFG) for support of the project via DFG grant GI1131/1-1.  Work performed while the fourth author was visiting Heinrich Heine University, D\"{u}sseldorf, Germany and supported by a DAAD fellowship as well as the U.S. Air Force Office of Scientific Research (AFOSR), Computational Mathematics Program, under grant number FA9550-18-1-0486.

\bibliographystyle{alpha}
\bibliography{./tristansbib,./tristanswritings,./ryanbib}

\end{document}